\documentclass[a4paper,12pt]{article}
\usepackage[dvips]{graphicx}
\usepackage{amsmath}
\usepackage{hyperref}
\usepackage{amssymb}
\usepackage{amsfonts}
\usepackage{ucs}
\usepackage[utf8x]{inputenc}
\usepackage[T1]{fontenc}
\usepackage{url}
\usepackage{psfrag}
\usepackage{dsfont}
\usepackage{color}

\usepackage{float}
\usepackage{rotating}
\def\del  {\partial}
\def\eps{\varepsilon}

\def\R{\mathbb{R}}

 \def\dx{{\rm d}x}

\def\Div{\operatorname{div}}

\newtheorem{remark}{\textbf{Remark}}
\newtheorem{theorem}{\textbf{Theorem}}
\newtheorem{lemma}{\textbf{Lemma}}
\newtheorem{definition}{\textbf{Definition}}

\def\dx{{\rm d}x}

\def\del  {\partial}
\def\eps{\varepsilon}

\def\R{\mathbb{R}}

\def\dx{{\rm d}x}

\newenvironment{proof}{%
{\noindent \bf Proof : }%
}{%
\hfill$\Box$\\%
}

\begin{document}

\title{MODELS OF NONLINEAR ACOUSTICS VIEWED AS\\ AN APPROXIMATION OF THE NAVIER-STOKES AND EULER COMPRESSIBLE ISENTROPIC SYSTEMS}

\author{Adrien DEKKERS\footnote{ Laboratory  Math\'ematiques et Informatique pour la Complexit\'e et les Syst\`emes CentraleSup\'elec, Univ\'ersit\'e Paris-Saclay Campus de Gif-sur-Yvette, Plateau de Moulon, 
3 rue Joliot Curie 91190 Gif-sur-Yvette, France, 
adrien.dekkers@centralesupelec.fr}, Vladimir KHODYGO  \footnote{Department of Biological, Environmental and Rural Sciences 
Aberystwyth University
Penglais Campus, 
Aberystwyth, 
Ceredigion, 
SY23 3FL,  vlk@aber.ac.uk }\\ and Anna ROZANOVA-PIERRAT \footnote{Laboratory  Math\'ematiques et Informatique pour la Complexit\'e et les Syst\`emes CentraleSup\'elec, Univ\'ersit\'e Paris-Saclay Campus de Gif-sur-Yvette, Plateau de Moulon, 
3 rue Joliot Curie 91190 Gif-sur-Yvette, France, 
anna.rozanova-pierrat@centralesupelec.fr}}

\maketitle

\begin{abstract}
 The derivation of different models of non linear acoustic in thermo-ellastic media as the Kuznetsov equation, the Khokhlov-Zabolotskaya-Kuznetsov (KZK) equation and the Nonlinear Progressive wave Equation (NPE) from an isentropic Navier-Stokes/Euler system is systematized using the Hilbert type expansion in the corresponding perturbative and (for the KZK and NPE equations) paraxial \textit{ansatz}. The use of small, to compare to the constant state perturbations, correctors allows to obtain the approximation results for the solutions of these models and to estimate the time during which they keep closed in the $L^2$ norm. The KZK and NPE equations are also considered as paraxial approximations of the Kuznetsov equation, which is a model obtained only by perturbations from the Navier-Stokes/Euler system. The Westervelt equation is obtained as a nonlinear approximation of the Kuznetsov equation. In the aim to compare the solutions of the exact and approximated systems in found approximation domains the 
 well-posedness results (for the Navier-Stokes system and the Kuznetsov equation in a half-space with periodic in time initial and boundary data) were obtained. 
\end{abstract}

\section{Introduction.}\label{intro}
There is a renewed interest in the study of nonlinear wave propagation, in particular because of recent applications to ultrasound imaging ($i.e.$ HIFU) or technical and medical applications such as lithotripsy or thermotherapy. Such new techniques rely heavily on the ability to model accurately the nonlinear propagation of a finite-amplitude sound pulse in thermo-viscous elastic media. The most known nonlinear acoustic models, which we consider in this paper, are
\begin{enumerate}
 \item the Kuznetsov equation (see Eq.~(\ref{KuzEq}) and Eq.~(\ref{KuzEqA})), which is actually a quasi-linear (damped) wave equation, initially introduced by Kuznetsov\cite{Kuznetsov} for the velocity potential, see also Refs.~\cite{Hamilton,Jordan,Lesser,Barbara} for other different methods of its derivation;
 \item the Khokhlov-Zabolotskaya-Kuznetsov (KZK) equation (see Eq.~(\ref{KZKI})), which can be written for the perturbations of the density or of the pressure (see the systematic physical studies in the book\cite{Bakhvalov});
 \item the Nonlinear Progressive wave Equation (NPE) (see Eq.~(\ref{NPE}) and Eq.~(\ref{NPE2})) derived in Ref.~\cite{McDonald};
 \item the Westervelt equation (see Eq.~(\ref{West})), which is similar to the Kuznetsov equation with only one of two nonlinear terms, derived initially by Westervelt\cite{Westervelt} and later by other authors\cite{Aanonsen,TJO}.
\end{enumerate}
All these models were derived from a compressible nonlinear isentropic Navier-Stokes (for viscous media) and Euler (for the inviscid case) systems up to some small negligible terms. But all cited physical derivations of these models don't allow to say that their solutions approximate the solution of the Navier-Stokes or Euler system. The first work explaining it for the KZK equation is Ref.~\cite{Roz3}. 
Starting in Section~\ref{secderisNS} to present the initial context of the isentropic Navier-Stokes system (actually, it is also an approximation of the compressible Navier-Stokes system~(\ref{N-S1})--(\ref{N-S4})), which describes the acoustic wave motion in an homogeneous thermo-ellastic medium\cite{Bakhvalov,Hamilton,Makarov}, we systematize in this article the derivation of all these models using the ideas of Ref.~\cite{Roz3}, consisting to use correctors in the Hilbert type expansions of corresponding physical \textit{ansatzs}. 

More precisely, we show that all these models are approximations of the isentropic Navier-Stokes or Euler system
up to third order terms of a small dimensionless parameter $\eps>0$ measuring the size of the perturbations of the pressure, the density and the velocity to compare to their constant state  $(p_0, \rho_0, 0)$ (see Fig.~\ref{fig1}).
\begin{figure}[!h]
\begin{center}
                   \psfrag{P}{$P$}
                   \psfrag{P:}{$P$: small perturbations~(\ref{rhoKuz})--(\ref{uKuz})}
                   \psfrag{AK}{$A_{KZK}$}
                   \psfrag{NS}{Navier-Stokes/Euler systems}
                   \psfrag{KE}{Kuznetsov equation}
                   \psfrag{AK:}{$A_{KZK}$: KZK-paraxial approximation (Fig.~\ref{fig2})}
                   \psfrag{AN:}{$A_{NPE}$: NPE-paraxial approximations (Fig.~\ref{fig3})}
                   \psfrag{AN}{$A_{NPE}$}
                   \psfrag{NPE}{NPE equation}
                   \psfrag{KZK}{KZK equation}
                    \psfrag{B}{$B$} \psfrag{B:}{$B$: bijection~(\ref{bijKZKNPE})}
                    \psfrag{PAK}{$P\circ A_{KZK}$}
                    \psfrag{PAN}{$P\circ A_{NPE}$}
  \includegraphics[width=0.9\textwidth]{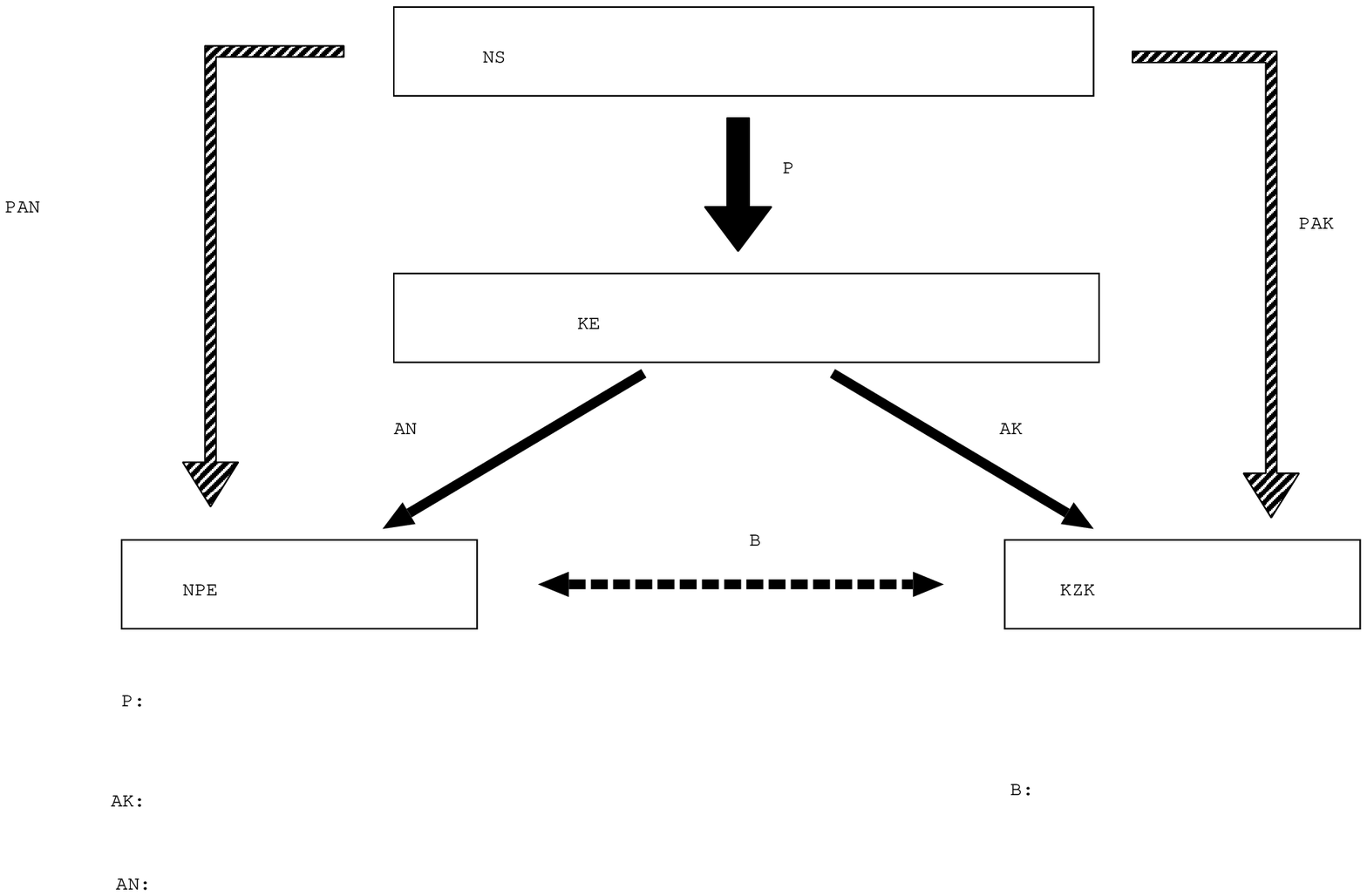}
                  \end{center}
   \caption{Schema of derivation of the models of the nonlinear acoustics. All models, the Kuznetsov, the KZK and the NPE equations are approximations up to terms of the order of $\eps^3$ of the isentropic Navier-Stokes or Euler system.}\label{fig1}
\end{figure}
As it is shown in Fig.~\ref{fig1}, the Kuznetsov equation comes from the Navier-Stokes or Euler system only by small perturbations, but to obtain the KZK and the NPE equations 
we also need to perform in addition to the small perturbations a paraxial change of variables. Moreover, the KZK and the NPE equations can be  also obtained from the Kuznetsov equation  just performing the corresponding paraxial change of variables. 
We can notice that the Kuznetsov equation~(\ref{KuzEqA}) is a non-linear wave equation containing the terms of different order on $\eps$. But the KZK- and NPE-paraxial approximations allow to have the  approximate equations with all terms of the same order, $i.e.$ the KZK and NPE equations.

The Westervelt equation is also an approximation of the Kuznetsov equation, but this time by a nonlinear perturbation. Actually the only difference between these two models is that the Westervelt equation keeps only one of two non-linear terms of the Kuznetsov equation, producing cumulative effects in a progressive wave propagation\cite{Aanonsen}. 

The NPE equation is usually used to describe short-time pulses and a long-range propagation, for instance, in an ocean wave-guide, where the refraction phenomena are important\cite{Kuperman2,Kuperman1}, while the KZK equation typically models the  ultrasonic propagation with strong diffraction phenomena,  combining with finite amplitude effects (see Ref.~\cite{Roz3} and the references therein).
Although the physical context and the physical using of the KZK and the NPE equations are different (see also Sections~\ref{secderNSKZK} and~\ref{secderNSNPE} respectively), there is a bijection (see Eq.~(\ref{bijKZKNPE})) between the variables of these two models and they can be presented by the same type differential operator with  constant positive coefficients:
$$Lu=0, \quad L=\partial^2_{t x}-c_1\partial_x(\partial_x \cdot )^2-c_2\partial^3_x \pm c_3\Delta_y, \quad \hbox{for } t\in \R^+,\; x\in\R,\; y\in \R^{n-1} .$$
Therefore, the results on the solutions of the KZK equation from Ref.~\cite{Roz2} are valid for the NPE equation. See also Ref.~\cite{Ito} for the exponential decay of the  solutions of these models in the viscous case.
 
 The well-posedness results for boundary value problems for the Kuznetsov equation are given in Refs.~\cite{Kalt2,Kalt1,Meyer} and for the Cauchy problem in Ref.~\cite{Perso}.
 
 Let us make attention that  \textit{ansatz}~(\ref{AnsaKZK1})--(\ref{AnsaKZK2}), proposed initially in Ref.~\cite{Bakhvalov} and used in Ref.~\cite{Roz3} to obtain the KZK equation from the Navier-Stokes or Euler systems, is different to   \textit{ansatz}~(\ref{EqAnsKZK1})--(\ref{EqAnsKZK2}) in Subsection~\ref{secderNSKZK}: this time it is the composition of the Kuznetsov perturbative \textit{ansatz} with the KZK paraxial change of variables\cite{Kuznetsov} (see Figs.~\ref{fig1} and~\ref{fig2}). 
 Moreover, this new approximation of the Navier-Stokes and  the Euler systems is an improvement to compare to the derivation developed in Ref.~\cite{Roz3} (see Subsection~\ref{secderNSKZK} for more details), as in Ref.~\cite{Roz3}  the Navier-Stokes/Euler system could be only approximated up to $O(\varepsilon^{\frac{5}{2}})$-terms (instead of $O(\eps^3)$ in our case).
 
In Section~\ref{secNS}, we validate the approximations of the compressible isentropic Navier-Stokes system by the different models:
by the Kuznetsov (Subsection~\ref{secNSKuz}), the KZK (Subsection~\ref{secNSKZK}) and the NPE equations (Subsection~\ref{secNSNPE}).

In Section~\ref{secEul} we do the same for the Euler system in the inviscid case. The main difference between the viscous and the inviscid case is the time existence and regularity of the solutions. Typically in the inviscid case, the solutions of the models and also of the Euler system itself (actually strong solutions), due to their non-linearity, can provide shock front formations at a finite time\cite{Alinhac2,Perso,Sideris2,Roz2,Yin1}. Thus, they are only locally well-posed, while in the viscous media all approximative models are globally well-posed for  small enough initial data\cite{Perso,Roz2}.  These existence properties of solutions for the viscous 
and the inviscid cases may also imply the difference in the definition of the domain where the approximations hold: for example\cite{Roz3}, for the approximation between the KZK equation and the Navier-Stokes system the approximation domain is  a half-space, but for the analogous inviscid case of the KZK and the Euler system it is a cone (see also the concluding Table~\ref{TABLE}).

The main hypothesis for the derivation of all these models are the following
\begin{itemize}
 \item the motion is potential;
 \item the constant state  of the medium given by $(p_0, \rho_0, 0)$ ($0$ for the velocity) is perturbed proportionally to an  dimensionless parameter $\eps>0$ (for instance, equal to $10^{-5}$ in water with an initial power of the order of $0.3 \, \mathrm{W} /\mathrm{cm}^2$);
 \item all viscosities are small (of order $\eps$).
 \end{itemize}
 To keep a physical sense of the approximation problems, we consider especially the two or three dimensional cases, $i.e.$ $\R^n$ with $n=2$ or $3$, and in the following we use the notation
 $x=(x_1,x')\in \R^n$ with one axis $x_1\in \R$ and the traversal variable $x'\in \R^{n-1}$.
 
In Sections~\ref{secNS} and~\ref{secEul} we denote by $\mathbf{U}_{\varepsilon}$ a solution of the ``exact'' Navier-Stokes/Euler system 
$Exact(\mathbf{U}_{\varepsilon})=0$ (see Eq.~(\ref{NSmatr}))
and by $\overline{\mathbf{U}}_{\varepsilon}$ an approximate  solution, constructed by the derivation \textit{ansatz} from a regular solution of one of the approximate models (typically of the Kuznetsov, the KZK or the NPE  equations), $i.e.$ a function which solves the Navier-Stokes/Euler system up to $\eps^3$ terms, denoted by $\eps^3 \textbf{R}$: $Approx(\overline{\mathbf{U}}_{\varepsilon})=Exact(\overline{\mathbf{U}}_{\varepsilon})-\eps^3 \textbf{R}=0$ (see Eq.~(\ref{NSmatrA})). 
To have the remainder term $\textbf{R}\in C([0,T],L^2(\Omega))$ we ensure that $Exact(\overline{\mathbf{U}}_{\varepsilon}) \in C([0,T],L^2(\Omega))$, $i.e.$ we need a sufficiently regular solution $\overline{\mathbf{U}}_{\varepsilon}$. The minimal regularity of the initial data to have a such $\overline{\mathbf{U}}_{\varepsilon}$ is given in Table~\ref{TABLE} (see also Table~\ref{TABLE2} for the approximations of the Kuznetsov equation).

Choosing for the exact system the same initial-boundary data found by the \textit{ansatz} for $\overline{\mathbf{U}}_{\varepsilon}$ (the regular case)  or the initial data taken  in their small $L^2$-neighborhood, $i.e.$ \begin{equation}\label{EqIndelta}
   \|\mathbf{U}_{\varepsilon}(0)-\overline{\mathbf{U}}_{\varepsilon}(0)\|_{L^2(\Omega)}\le \delta\le \eps,                                                                                                                                                                                                                                     \end{equation}
   with $\mathbf{U}_{\varepsilon}(0)$ not necessarily smooth, but ensuring the existence of an admissible weak solution of a bounded energy (see Definition~\ref{DefAdm}),
  we prove the existence of  constants $C>0$ and $K>0$ independent of $\varepsilon$, $\delta$ and   the time $t$ such that
\begin{equation}\label{validaproxintro}
 \hbox{for all }0\leq t\leq\frac{C}{\varepsilon}\;\;\;\;\;\;\Vert (\mathbf{U}_{\varepsilon}-\overline{\mathbf{U}}_{\varepsilon})(t)\Vert_{L^2(\Omega)}^2\leq K(\eps^3t+\delta^2)  e^{K\varepsilon t}\leq 9\varepsilon^2
 \end{equation}
  with $\Omega$ a domain where the both solutions $\mathbf{U}_{\varepsilon}$ and  $\overline{\mathbf{U}}_{\varepsilon}$ exist (see Theorems~\ref{ThapproxNSKuz},~\ref{ThAprKZKNS} and~\ref{ThapproxNSNPE}). 
  
   In the viscous case all aproximative models have a global unique classical solution for  small enough initial data in their corresponding approximative domains ($\Omega$  varies for different models, see Table~\ref{TABLE}): it is equal to $\R^n$, $\mathbb{T}_{x_1}\times\mathbb{R}^{n-1}$ and $\mathbb{R}_+\times \mathbb{R}^{n-1}$ for the Kuznetsov equation,   the NPE equation and  the KZK equation respectively.  If we take regular initial data $\mathbf{U}_{\varepsilon}(0)=\overline{\mathbf{U}}_{\varepsilon}(0)$, the same thing is true for the Navier-Stokes system with the same regularity for the solutions\cite{Matsumura}. But in the case of the half-space for the approximation between the Navier-Stokes system and the KZK equation, firstly considered in Ref.~\cite{Roz3}, when, due to the periodic in time boundary conditions, coming from the initial conditions for the KZK equation, we  prove the well-posedness for all finite time. To obtain it we use Ref.~\cite{Roz3} Theorem~5.5. We updated it in the framework of the new \textit{ansatz}~(\ref{EqAnsKZK1})--(\ref{EqAnsKZK2}) and corrected several misleading  in its proof (see Subsection~\ref{sNSe} Theorem~\ref{thENS}), what allows us in Theorem~\ref{ThAprKZKNS} of Subsection~\ref{secValNSKZK}  to establish the approximation result between the KZK equation and the Navier-Stokes system following Ref.~\cite{Roz3} Theorem~5.7 just updating the stability approximation estimate. 

To  obtain estimate~(\ref{validaproxintro}) we don't need  the regularity of the classical solution of the Navier-Stokes (or Euler) system, it can be a weak solution (in the sense of Hoff\cite{Hoff1} for the Navier-Stokes system or one of solutions in the sense of Luo and al.\cite{LUO-2016} for the Euler system) satisfying the admissible conditions given in Definition~\ref{DefAdm} (see also Ref.~\cite{Dafermos} p.52 and Ref.~\cite{Roz3} Definition~5.9).

 For the inviscid case, given in Section~\ref{secEul}, we verify that  the existence time of (strong) solutions of  all models is not less than $O(\frac{1}{\eps})$ and estimate~(\ref{validaproxintro}) still holds. 

As the KZK and NPE equations can be seen as approximations of the Kuznetsov equation due to their derivation (see Figure~\ref{fig1}), we also validate the approximation of the Kuznetsov equation by the KZK and NPE equations, and also by the Westervelt equation, in Section~\ref{secKuzKZK},~\ref{secKuzNPE} and~\ref{SecWest} (see Table~\ref{TABLE2}). 

To be able to consider the approximation of the Kuznetsov equation by the KZK equation (see Section~\ref{secKuzKZK}), we firstly establish  global well-posedness results for the Kuznetsov equation in the half space similar to the previous framework for the KZK and the Navier-Stokes system in Subsection~\ref{sNSe}. We study two cases: the purely time periodic boundary problem in the \textit{ansatz} variables $(z,\tau,y)$ moving with the wave and the initial boundary-value problem for the Kuznetsov equation in the initial variables $(t,x_1,x')$ with data coming from the solution of the KZK equation. We  validate these two types  approximations in Subsection~\ref{secValKuzKZK} for the viscous and inviscid cases.
Finally in Sections~\ref{secKuzNPE} and~\ref{SecWest} we validate the approximation between the Kuznetsov and NPE equation and the Kuznetsov and Westervelt equations respectively (see Table~\ref{TABLE2}). 
We can summarize the approximation results of the Kuznetsov equation in the following way:
 if $u$ is a solution of the Kuznetsov equation and $\overline{u}$ is a solution of the NPE or of the the KZK (for the initial boundary value problem) or of the Westervelt equations found for rather closed initial data $$\|\nabla_{t,\mathbf{x}} (u(0)-\overline{u}(0))\|_{L^2(\Omega)}\le \delta\le \eps,$$
then there exist $K>0$, $C_1>0$, $C_2>0$ and $C>0$ independent on $\eps$, $\delta$ and on time, such that for all $t\le \frac{C}{\eps}$ it holds
$$\|\nabla_{t,\mathbf{x}} (u-\overline{u})\|_{L^2(\Omega)}\le C_1(\eps^2t+\delta)e^{C_2\eps t}\le K\eps.$$

\section{Isentropic Navier-Stokes system for a subsonic potential motion.}\label{secderisNS} To describe the acoustic wave motion in an homogeneous thermo-ellastic medium, we start from the Navier-Stokes system in $\mathbb{R}^n$
\begin{align}
&\partial_t \rho+ \Div(\rho \mathbf{v})=0, \label{N-S1}\\ 
&\rho[\partial_t \mathbf{v}+(\mathbf{v}.\nabla) \mathbf{v}]=-\nabla p+\eta \Delta \mathbf{v}+\left(\zeta+\frac{\eta}{3} \right)\nabla.\Div (\mathbf{v}), \\ 
&\rho T [\partial_t S+(\mathbf{v}.\nabla)S]=\kappa \Delta T+\zeta (\Div \mathbf{v})^2 \nonumber\\
&+\frac{\eta}{2}\left( \partial_{x_k} v_i+\partial_{x_i}v_k-\frac{2}{3}\delta_{ik}\partial_{x_i} v_i\right)^2,\label{N-S3}\\
&p=p(\rho,S),\label{N-S4}
\end{align}
where 
the pressure $p$ is given by the state law $p=p(\rho,S)$. The density $\rho$, the velocity $\mathbf{v}$, the temperature $T$ and the entropy $S$ are unknown functions in  system~(\ref{N-S1})--(\ref{N-S4}). The coefficients $\beta,\;\kappa$ and $\eta$ are constant viscosity coefficients. The wave motion is supposed to be potential and the viscosity coefficients are supposed to be small in terms of a dimensionless small parameter $\eps>0$:
$$\eta \Delta \mathbf{v}+\left(\zeta+\frac{\eta}{3} \right)\nabla.\Div (\mathbf{v})=\left(\zeta+\frac{4}{3}\eta\right)\Delta \mathbf{v}:=\beta\Delta \mathbf{v},\quad \hbox{ with }\beta=\varepsilon\tilde{\beta}.$$ Any constant state $(\rho_0 ,\mathbf{v}_0, S_0, T_0)$ is a stationary solution of system~(\ref{N-S1})--(\ref{N-S4}). 
 Further we always take $\mathbf{v}_0=0$ using a Galilean transformation.
 Perturbation near
this constant state $(\rho_0, 0, S_0, T_0)$ introduces small increments in terms of the same dimensionless small parameter $\eps>0$:
\begin{align*}
 &T(x,t)=T_0+\varepsilon \tilde{T}(x,t)\quad\hbox{and}\quad S(x,t)=S_0+\varepsilon^2 \tilde{S}(x,t),\\
 &\rho_{\varepsilon}(x,t)=\rho_0+\varepsilon \tilde{\rho}_{\varepsilon}(x,t)\quad\hbox{and}\quad\mathbf{v}_{\varepsilon}(x,t)=\varepsilon \tilde{\mathbf{v}}_{\varepsilon}(x,t),
\end{align*}
where the perturbation of the entropy is of order $O(\eps^2)$, since it is the smallest size on $\eps$ of right hand terms in Eq~(\ref{N-S3}), due to the smallness of the viscosities (see~Eq.~(\ref{EqS})).

Actually, $\eps$ is the Mach number, which is supposed to be small\cite{Bakhvalov} ($\epsilon=10^{-5}$ for the propagation in water with an initial
power of the order of $0.3 \, \mathrm{W} /\mathrm{cm}^2$): 
\begin{equation*}
  \dfrac{\rho-\rho_0}{\rho_0}\sim\dfrac{T-T_0}{T_0}\sim\dfrac{|\mathbf{v}|}{c_0}\sim\epsilon,
 \end{equation*}
where  $c_0 = \sqrt{p' (\rho_0)}$  is the  speed  of  sound  in  the  unperturbed  media.

Using the transport heat equation~(\ref{N-S3}) up to the terms of the order of $\varepsilon^3$
\begin{equation}\label{EqS}
 \varepsilon^2 \rho_0 T_0 \partial_t \tilde{S}=\varepsilon^2 \tilde{\kappa} \Delta \tilde{T}+\mathcal{O}(\varepsilon^3),
\end{equation}
the approximate state equation
$$p=p_0+c^2\varepsilon \tilde{\rho}_{\varepsilon}+\frac{1}{2}(\partial^2_{\rho}p)_S \varepsilon^2 \tilde{\rho}^2_{\varepsilon}+(\partial_S p)_{\rho} \varepsilon^2 \tilde{S}+\mathcal{O}(\varepsilon^3)$$
(where the notation $(.)_S$ means that the expression in brackets is constant in $S$), can be replaced~\cite{Bakhvalov,Makarov,Hamilton}   by
$$p=p_0+c^2\varepsilon \tilde{\rho}_{\varepsilon}+\frac{(\gamma-1)c^2}{2\rho_0}\varepsilon^2 \tilde{\rho}^2_{\varepsilon}-\varepsilon \tilde{\kappa}\left(\frac{1}{C_V}-\frac{1}{C_p}\right)\nabla.\mathbf{v}_\varepsilon+\mathcal{O}(\varepsilon^3),$$
 using $T=\frac{p}{\rho R}$ from the theory of ideal gaze and taking 
$$p(\rho, S)=R\rho^{\gamma} e^{\dfrac{S-S_0}{C_V}}.$$
Here $\gamma=C_p/C_V$ denotes the ratio of the heat capacities at constant pressure and at constant volume respectively.

Hence, system~(\ref{N-S1})--(\ref{N-S4}) becomes an isentropic Navier-Stokes system
\begin{gather} 
\partial_t \rho_\varepsilon  +\operatorname{div}( \rho_\varepsilon  \mathbf{v}_\varepsilon)=0\,,\label{NSi1}\\
 \rho_\varepsilon[\partial_t
\mathbf{v}_\varepsilon+(\mathbf{v}_\varepsilon \cdot\nabla) \,
\mathbf{v}_\varepsilon] = -\nabla p(\rho_\varepsilon)
+\varepsilon \nu \Delta \mathbf{v}_\varepsilon\,,\label{NSi2}
\end{gather}
with the approximate state equation $p(\rho, S)=p(\rho_\eps)+O(\eps^3)$:
\begin{equation}\label{press}
p(\rho_{\varepsilon})
=p_0+c^2(\rho_{\varepsilon}-\rho_0)+\frac{(\gamma-1)c^2}{2\rho_0} (\rho_{\varepsilon}-\rho_0)^2,
\end{equation}
and with a small enough and positive viscosity coefficient:
$$\varepsilon \nu=\beta+\kappa \left(\frac{1}{C_V}-\frac{1}{C_p}\right).$$
\section{Approximation of the Navier-Stokes system.}\label{secNS}

\subsection{Navier-Stokes system and the Kuznetsov equation.}\label{secNSKuz} 

We consider system~(\ref{NSi1})--(\ref{press}) as the exact model.
The state law~(\ref{press})  is a Taylor  expansion of the pressure up to the terms of the third order on $\eps$. Therefore an approximation of system~(\ref{NSi1})--(\ref{press}) for $\mathbf{v}_\varepsilon$ and $\rho_\eps$ up to  terms $O(\eps^3)$ would be optimal.
In the framework of the nonlinear acoustic between the known approximative models derived from system~(\ref{NSi1})--(\ref{press}) are 
 the Kuznetsov, the KZK and the NPE equations. 
In this section we focus on the first of these models, $i.e.$ on the Kuznetsov equation.

Initially the Kuznetsov equation was 
derived by Kuznetsov~\cite{Kuznetsov} from the isentropic Navier-Stokes system~(\ref{NSi1})--(\ref{press}) for 
the small velocity potential $\mathbf{v}_{\varepsilon}(\mathbf{x},t)=-\nabla \tilde{u}(\mathbf{x},t),$ $\mathbf{x}\in \mathbb{R}^n,$  $t\in \mathbb{R}^+$:
\begin{equation}\label{KuzEq}
 \partial^2_t \tilde{u} -c^2 \triangle \tilde{u}=\partial_t\left( (\nabla \tilde{u})^2+\frac{\gamma-1}{2c^2}(\partial_t \tilde{u})^2+\frac{\varepsilon\nu}{\rho_0}\Delta \tilde{u}\right).
\end{equation}
The derivation was latter  discussed by a lot of authors~\cite{Hamilton,Jordan,Lesser}.

In the difference to these physical derivations we introduce a Hilbert expansion
type construction with a corrector $\varepsilon^2\rho_2(\mathbf{x},t)$ for the density perturbation, considering the following \emph{ansatz}
\begin{gather}
\rho_\varepsilon(\mathbf{x}, t)=\rho_0+\varepsilon\rho_1(\mathbf{x},t)+\varepsilon^2\rho_2(\mathbf{x},t),\label{rhoKuz}\\
\mathbf{v}_\varepsilon (\mathbf{x}, t)=-\varepsilon\nabla u(\mathbf{x},t).\label{uKuz}
\end{gather}
The use of the second order corrector in~(\ref{rhoKuz}) allows to ensure the approximation of~(\ref{NSi2}) up to terms of order $\eps^3$ (see Subsection~\ref{secderNSKuz}) and to open the question about the approximation between the exact solution of the isentropic Navier-Stokes system~(\ref{NSi1})--(\ref{press}) and its approximation given by the solution of the Kuznetsov equation, as it was done for the KZK equation\cite{Roz3}.

\subsubsection{Derivation of the Kuznetsov equation from an isentropic Navier-Stokes system.}\label{secderNSKuz}

Putting expressions for the density and velocity~(\ref{rhoKuz})--(\ref{uKuz}) into the  isentropic Navier-Stokes system~(\ref{NSi1})--(\ref{press}), we obtain for the momentum conservation~(\ref{NSi2})
\begin{multline}
\rho_\varepsilon[\partial_t
\mathbf{v}_\varepsilon+(\mathbf{v}_\varepsilon \cdot\nabla) \,
\mathbf{v}_\varepsilon] +\nabla p(\rho_\varepsilon)
-\varepsilon \nu \Delta \mathbf{v}_\varepsilon= \varepsilon \nabla(-\rho_0 \partial_t u+c^2 \rho_1)\\
 +\varepsilon^2 \left[-\rho_1 \nabla(\partial_t u)+\frac{\rho_0}{2}\nabla((\nabla u)^2)
+c^2 \nabla\rho_2+\frac{(\gamma-1)c^2}{2 \rho_0}\nabla(\rho_1^2)+\nu \nabla\Delta u\right]
+O(\varepsilon^3). \label{EqAppNS1p}
\end{multline}
In order to have an approximation up to the terms $O(\varepsilon^3)$ we put the terms of order one and two in $\varepsilon$ equal to $0$, what allows us to find the expressions for the density correctors: 
\begin{align}
 \rho_1(\mathbf{x},t) =&\frac{\rho_0}{c^2}\partial_t u(\mathbf{x},t),\label{rho1K}\\ \rho_2(\mathbf{x},t)=&-\frac{\rho_0(\gamma-2)}{2c^4} (\partial_t u)^2-\frac{\rho_0}{2c^2 }(\nabla u)^2-\frac{\nu}{c^2}\Delta u. \label{rho2K}
\end{align}
Indeed, we start by making $\varepsilon \nabla(-\rho_0 \partial_t u+c^2 \rho_1)=0$ and find the first order perturbation of the density $\rho_1$ given by Eq.~(\ref{rho1K}). Consequently, if $\rho_1$ satisfies~(\ref{rho1K}), then Eq.~(\ref{EqAppNS1p}) becomes
\begin{multline}
\rho_\varepsilon[\partial_t
\mathbf{v}_\varepsilon+(\mathbf{v}_\varepsilon \cdot\nabla) \,
\mathbf{v}_\varepsilon] +\nabla p(\rho_\varepsilon)
-\varepsilon \nu \Delta \mathbf{v}_\varepsilon=\varepsilon \nabla(-\rho_0 \partial_t u+c^2 \rho_1) \\
 \varepsilon^2 \nabla\left[-\frac{\rho_0}{2c^2} (\partial_t u)^2+\frac{\rho_0}{2}(\nabla u)^2
+c^2 \rho_2+\frac{(\gamma-1)\rho_0}{2c^2 }(\partial_t u)^2+\nu \Delta u\right]
+O(\varepsilon^3). \label{EqAppNS1p2}
\end{multline}
Thus, taking the corrector $\rho_2$ by formula~(\ref{rho2K}), we ensure that 
\begin{equation}
 \rho_\varepsilon[\partial_t
\mathbf{v}_\varepsilon+(\mathbf{v}_\varepsilon \cdot\nabla) \,
\mathbf{v}_\varepsilon] +\nabla p(\rho_\varepsilon)
-\varepsilon \nu \Delta \mathbf{v}_\varepsilon= O(\varepsilon^3). \label{EqAppNS1p3}
\end{equation}
Now we put these expressions of $\rho_1$ from~(\ref{rho1K}) and $\rho_2$ from~(\ref{rho2K}) with  \emph{ansatz}~(\ref{rhoKuz})--(\ref{uKuz}) in  Eq.~(\ref{NSi1}) of the mass conservation to obtain
\begin{align}
&\partial_t \rho_\varepsilon  +\operatorname{div}( \rho_\varepsilon  \mathbf{v}_\varepsilon)=\varepsilon\frac{\rho_0}{c^2}\left[\partial^2_t u -c^2 \Delta u-\right.\nonumber\\
     &\left.\varepsilon\partial_t\left( (\nabla u)^2+\frac{\gamma-2}{2c^2}(\partial_t u)^2+\frac{\nu}{\rho_0}\Delta u\right)-\varepsilon u_t \Delta u\right]+O(\varepsilon^3).\label{APrNS1K}
\end{align}
Then we notice that the  right hand term of the order $\eps$ in Eq.~(\ref{APrNS1K}) is actually the linear wave equation up to smaller on $\eps$ therms:
$$\partial^2_t u-c^2 \Delta u=O(\varepsilon).$$
Hence, we express 
$$\varepsilon u_t \Delta u=\varepsilon\frac{1}{c^2} u_t u_{tt}+O(\varepsilon^2)=\varepsilon\frac{1}{2c^2}\partial_t((u_t)^2)+O(\varepsilon^2), $$
and putting it in Eq.~(\ref{APrNS1K}),  we finally have
 \begin{align}
     \partial_t \rho_\varepsilon  +\operatorname{div}( \rho_\varepsilon  \mathbf{v}_\varepsilon)&=\varepsilon\frac{\rho_0}{c^2}\left[\partial^2_t u -c^2 \Delta u-\right.\nonumber\\
     &\left.\varepsilon\partial_t\left( (\nabla u)^2+\frac{\gamma-1}{2c^2}(\partial_t u)^2+\frac{\nu}{\rho_0}\Delta u\right)\right]+O(\varepsilon^3). \label{ns1K}
     \end{align}
   The right hand side of Eq.~(\ref{ns1K}) gives us  the Kuznetsov equation %
\begin{equation}\label{KuzEqA}
 \partial^2_t u -c^2 \Delta u=\varepsilon\partial_t\left( (\nabla u)^2+\frac{\gamma-1}{2c^2}(\partial_t u)^2+\frac{\nu}{\rho_0}\Delta u\right),
\end{equation}
which is  the first order approximation of the isentropic Navier-Stokes system %
up to the terms $O(\varepsilon^3)$. Moreover,
if $u$ is a solution of the Kuznetsov equation, then 
with the relations for the density perturbations~(\ref{rho1K}) and~(\ref{rho2K}) and with  \emph{ansatz}~(\ref{rhoKuz})--(\ref{uKuz})
we have
\begin{gather} 
\partial_t \rho_\varepsilon  +\operatorname{div}( \rho_\varepsilon  \mathbf{v}_\varepsilon)=O(\eps^3)\,,\label{NSi1re}\\
 \rho_\varepsilon[\partial_t
\mathbf{v}_\varepsilon+(\mathbf{v}_\varepsilon \cdot\nabla) \,
\mathbf{v}_\varepsilon] +\nabla p(\rho_\varepsilon)
-\varepsilon \nu \Delta \mathbf{v}_\varepsilon=O(\eps^3).\label{NSi2re}\end{gather}
Hence, it is clear that the standard physical perturbative approach without the corrector $\rho_2$ (it is sufficient to   take $\rho_2=0$ in our calculus) can't ensure~(\ref{NSi1re})--(\ref{NSi2re}).

Let us also notice, as it was  originally mentioned  by Kuznetsov, that the Kuznetsov equation~(\ref{KuzEqA}) contains terms of different orders, and hence, it is a wave equation with small size non-linear perturbations  $\partial_t(\nabla u)^2$, $\partial_t(\partial_t u)^2$ and the viscosity term $\partial_t\Delta u$.

\subsubsection{Approximation of the solutions of the isentropic Navier-Stokes system by the solutions of the Kuznetsov equation.}\label{secapNSKuz}

Let us calculate the remainder terms in~(\ref{NSi1re})--(\ref{NSi2re}), which are denoted respectively by 
$\varepsilon^3 R_1^{NS-Kuz}$ and $\varepsilon^3 \mathbf{R}_2^{NS-Kuz}$:
\begin{align}
&\varepsilon^3 R_1^{NS-Kuz}=\varepsilon^3\left[\frac{1}{c^2}\partial_t u\left(\frac{\rho_0(\gamma-2)}{2c^4}\partial_t[(\partial_t  u)^2]+\frac{\rho_0}{c^2}\partial_t[(\nabla u)^2]+\frac{\nu}{c^2}\partial_t\Delta u\right)\right.\nonumber\\
 &\left.\;\;\;\;\;-\frac{\rho_0}{c^2}\partial_t u\; \Delta u-\nabla \rho_2.\nabla u-\rho_2 \Delta u\right]+\varepsilon^4 \frac{1}{c^2}\partial_t u\left(\nabla \rho_2.\nabla u+\rho_2 \Delta u\right),\label{rNS1}\\
&\varepsilon^3 \mathbf{R}_2^{NS-Kuz}=\varepsilon^3\left[\frac{\rho_1}{2}\nabla[(\nabla u)^2]-\rho_2 \nabla \partial_t u\right]
+\varepsilon^4\frac{\rho_2}{2}\nabla\left[(\nabla u)^2\right].\label{rNS2}
\end{align}
If $u$ is a sufficiently regular solution of the Cauchy problem for the Kuznetsov equation in $\mathbb{R}^n$
\begin{equation}\label{CauProbKuz}
\left\lbrace
\begin{array}{l}
 \partial^2_t u -c^2 \Delta u=\varepsilon\partial_t\left( (\nabla u)^2+\frac{\gamma-1}{2c^2}(\partial_t u)^2+\frac{\nu}{\rho_0}\Delta u\right),\\
u(0)=u_0,\;\;u_t(0)=u_1,
\end{array}
\right.
\end{equation}
then, taking $\rho_1$ and $\rho_2$ according to formulas~(\ref{rho1K})-(\ref{rho2K}), we define
$\overline{\rho}_{\varepsilon}$ and $\overline{ \mathbf{v}}_{\varepsilon}$ by formulas~(\ref{rhoKuz})-(\ref{uKuz}) and obtain a solution of the following approximate system
\begin{align}
&\partial_t\overline{ \rho}_{\varepsilon}+\operatorname{div}(\overline{\rho}_{\varepsilon}\overline{ v}_{\varepsilon})=\varepsilon^3 R_1^{NS-Kuz},\label{NSAp1}\\
&\overline{\rho}_{\varepsilon} [\partial_t \overline{ \mathbf{v}}_{\varepsilon}+( \overline{ \mathbf{v}}_{\varepsilon}.\nabla) \overline{ \mathbf{v}}_{\varepsilon}]+\nabla p(\overline{\rho}_{\varepsilon})-\varepsilon \nu \Delta \overline{ \mathbf{v}}_{\varepsilon}=\varepsilon^3 \mathbf{R}_2^{NS-Kuz}\label{NSAp2}
\end{align}
with $p(\overline{\rho}_{\varepsilon})$ from the state law~(\ref{press}).  
 With notations 
$\mathbf{U}_{\varepsilon}=(\rho_{\varepsilon},\;\rho_{\varepsilon}\mathbf{v}_{\varepsilon})^t$ and $\overline{\mathbf{U}}_{\varepsilon}=(\overline{\rho}_{\varepsilon},\; \overline{\rho}_{\varepsilon} \overline{ \mathbf{v}}_{\varepsilon})^t $,  
the  exact~(\ref{NSi1})--(\ref{NSi2}) and the approximated~(\ref{NSAp1})--(\ref{NSAp2}) Navier-Stokes systems  can be respectively rewritten 
in the following forms\cite{Dafermos,Roz3}:
\begin{align}
& \partial_t \mathbf{U}_{\varepsilon}+\sum_{i=1}^n \partial_{x_i} \mathbf{G}_i(\mathbf{U}_\varepsilon)-\varepsilon \nu \begin{bmatrix}
0 \\
\Delta \mathbf{v}_{\varepsilon}
\end{bmatrix}
=0,\label{NSmatr}\\
 &\partial_t \overline{\mathbf{U}}_{\varepsilon}+\sum_{i=1}^n \partial_{x_i} \mathbf{G}_i(\overline{\mathbf{U}}_\varepsilon)-\varepsilon \nu \begin{bmatrix}
0 \\
\Delta \overline{\mathbf{v}}_{\varepsilon}
\end{bmatrix}
=\eps^3\mathbf{R}^{NS-Kuz}\label{NSmatrA}
\end{align}
with $\mathbf{R}^{NS-Kuz}=\begin{bmatrix}  R_1^{NS-Kuz}\\\mathbf{R}_2^{NS-Kuz} \end{bmatrix}$ from~(\ref{rNS1})--(\ref{rNS2}) and 
\begin{equation}\label{NSmatrfact}
\mathbf{G}_i(\mathbf{U}_{\varepsilon} )=\begin{bmatrix}
\rho_{\varepsilon} v_i\\
\rho_{\varepsilon}v_i\mathbf{v}_{\varepsilon}+p(\rho_{\varepsilon})\mathbf{e}_i
\end{bmatrix}, \quad\partial_{x_i} \mathbf{G}_i(\mathbf{U}_{\varepsilon})=D\mathbf{G}_i(\mathbf{U}_{\varepsilon})\partial_{x_i}\mathbf{U}_{\varepsilon}.
\end{equation}

The well-posedness results for the   Cauchy problems~(\ref{NSi1})-(\ref{press})\cite{Matsumura} and~(\ref{CauProbKuz})\cite{Perso} 
allow us to establish the global existence and the unicity of the classical solutions $\mathbf{U}_{\varepsilon}$ and $\overline{\mathbf{U}}_{\varepsilon}$, considered in the Kuznetsov approximation framework:
\begin{theorem}\label{ThExistUUbarNSK}
 There exists  a constant $k>0$ such that if the initial data $u_0\in H^5(\mathbb{R}^3)$ and $u_1\in H^4(\mathbb{R}^3)$ for the Cauchy problem for the Kuznetsov equation~(\ref{CauProbKuz}) 
 are sufficiently small
 $$\Vert u_0\Vert_{H^5(\mathbb{R}^3)}+\Vert u_1\Vert_{H^4(\mathbb{R}^3)}<k,$$
 then there exist global in time solutions $\overline{\mathbf{U}}_\eps=(\overline{\rho}_{\varepsilon},\; \overline{\rho}_{\varepsilon} \overline{ \mathbf{v}}_{\varepsilon})^t$ of the approximate Navier-Stokes system~(\ref{NSmatrA})  and $\mathbf{U}_\eps=(\rho_{\varepsilon},\;\rho_{\varepsilon}\mathbf{v}_{\varepsilon})^t$ of the exact  Navier-Stokes system~(\ref{NSmatr}) respectively, with the same regularity corresponding to
 \begin{equation}\label{EqregRhoNS}
 \overline{\rho}_{\varepsilon}-\rho_0,\;\rho_{\varepsilon}-\rho_0 \in C([0,+\infty[;H^3(\mathbb{R}^3))\cap C^1([0,+\infty[;H^2(\mathbb{R}^3))
\end{equation} 
and 
\begin{equation}\label{EqregvNS}
   \mathbf{\overline{v}}_{\varepsilon},\;   \mathbf{v}_{\varepsilon}\in C([0,+\infty[;H^3(\mathbb{R}^3))\cap C^1([0,+\infty[;H^1(\mathbb{R}^3)),
    \end{equation}
 %
 both considered with the state law~(\ref{press}) and with the same initial data
 \begin{eqnarray}
    &&(\bar{\rho}_\eps-\rho_\eps)|_{t=0} =0,
\quad (\bar{\mathbf{v}}_\eps-\mathbf{v}_\eps)|_{t=0} = 0,\label{bban}\end{eqnarray}
where $\bar{\rho}_\eps|_{t=0}$ and $\bar{\mathbf{v}}_\eps|_{t=0}$ are constructed as the functions of the initial data for the Kuznetsov equation $u_0$ and $u_1$ according to formulas~(\ref{rhoKuz})--(\ref{uKuz}) and~(\ref{rho1K})--(\ref{rho2K}):
\begin{align}
 &\bar{\rho}_\eps|_{t=0}= \rho_0+\eps\frac{\rho_0}{c^2}u_1-\eps^2\left[\frac{\rho_0(\gamma-2)}{2c^4} u_1^2+\frac{\rho_0}{2c^2 }(\nabla u_0)^2+\frac{\nu}{c^2}\Delta u_0\right],\label{InCondrob}\\
 &\bar{\mathbf{v}}_\eps|_{t=0}=-\eps \nabla u_0.\label{InCondub}
\end{align}
\end{theorem}
\begin{proof}
From one hand,  Theorem~1.2 in Ref.~\cite{Perso} applied for $n=3$ with $m=4$ ensures that for  $u_0\in H^5(\mathbb{R}^3)$ and $u_1\in H^4(\mathbb{R}^3)$  there exists a constant $k_2>0$ such that if 
 \begin{equation}\label{smallcondWPkuz}
 \Vert u_0\Vert_{H^5(\mathbb{R}^3)}+ \Vert  u_1\Vert_{H^4(\mathbb{R}^3)}<k_2,
 \end{equation}
 then the Cauchy problem for the Kuznetsov equation~(\ref{CauProbKuz}) has a unique global in time solution 
 \begin{equation}\label{RegKuzR3}
 u\in C([0,+\infty[,H^5(\mathbb{R}^3))\cap C^1([0,+\infty[,H^4(\mathbb{R}^3))\cap C^2([0,+\infty[,H^2(\mathbb{R}^3)).
 \end{equation}

 From the other hand,  the Cauchy problem for the Navier-Stokes system is also globally well-posed in $\R^3$ for sufficiency small initial data (see Ref.~\cite{Matsumura} Theorem~7.1, p.~100): 
there exists a constant $k_1>0$ such that if  the initial data 
 \begin{equation}\label{NSinitdata}
 \rho_{\varepsilon}(0)-\rho_0\in H^3(\mathbb{R}^3) ,\;\;\;\mathbf{v}_{\varepsilon}(0)\in H^3(\mathbb{R}^3)
 \end{equation}
 satisfy $$\Vert \rho_{\varepsilon}(0)-\rho_0\Vert_{H^3(\mathbb{R}^3)}+  \Vert \mathbf{v}_{\varepsilon}(0)\Vert_{H^3(\mathbb{R}^3)}<k_1,$$ then the  Cauchy problem~(\ref{NSi1})-(\ref{press}) with the  initial data~(\ref{NSinitdata}) has a unique solution $(\rho_{\varepsilon},$ $\mathbf{v}_{\varepsilon} )$ globally in time satisfying~(\ref{EqregRhoNS}) and~(\ref{EqregvNS}). 
 
 Thus, for the initial solutions of the Kuznetsov equation we need to impose $u_0\in H^5(\mathbb{R}^3)$ to have $\Delta u_0\in H^3(\mathbb{R}^3)$  to be able to ensure that $\rho_\eps-\rho_0|_{t=0}\in H^3(\mathbb{R}^3)$. The regularity $u_1\in H^4(\mathbb{R}^3)$ comes from the well-posedness of the Kuznetsov problem and obviously ensures $\mathbf{v}_\eps|_{t=0}\in H^3(\mathbb{R}^3)$, what is necessary\cite{Matsumura} 
 to have a global solution of the exact Navier-Stokes system~(\ref{NSmatr}).

As $\overline{\rho}_{\varepsilon}$ and $\overline{ \mathbf{v}}_{\varepsilon}$ are defined by \textit{ansatz} ~(\ref{rhoKuz})-(\ref{uKuz}) with  $\rho_1$ and $\rho_2$ given in~(\ref{rho1K}) and~(\ref{rho2K}) respectively,
the regularity of $u$ ensures for $\overline{\rho}_{\varepsilon}-\rho_0$ and $\overline{ \mathbf{v}}_{\varepsilon}$ at least the same regularity as given in~(\ref{EqregRhoNS}) and~(\ref{EqregvNS}). 
To find it we use 
the following Sobolev embedding for the multiplication (see for example Ref.~\cite{Bers} or~\cite{Kato1}):
\begin{align}
H^s(\mathbb{R}^n)\times H^s(\mathbb{R}^n)& \hookrightarrow H^s (\mathbb{R}^n)\;\;\text{for }s>\frac{n}{2},\label{Sobolalg}\\ 
(u,v) & \mapsto uv.\nonumber
\end{align}
Moreover, considering formulas~(\ref{rNS1})--(\ref{rNS2}) with $u$ as defined in~(\ref{RegKuzR3}),
 all terms in $R_1^{NS-Kuz}$ and $\mathbf{R}_2^{NS-Kuz}$ are in $H^2(\mathbb{R}^3)$. Therefore, as $2>\frac{3}{2}$, we  use embedding~(\ref{Sobolalg}) to find that
$$R_1^{NS-Kuz}\in C([0,+\infty[,H^2(\mathbb{R}^3)) \quad \hbox{and}\quad \mathbf{R}_2^{NS-Kuz}\in C([0,+\infty[,H^2(\mathbb{R}^3)).$$ Hence, the $L^2(\mathbb{R}^3)$ and $L^{\infty}(\mathbb{R}^3)$ norms of the remainder terms $R_1^{NS-Kuz}(t)$ and $\mathbf{R}_2^{NS-Kuz}(t)$
 are bounded for $t\in [0,+\infty[$.

Finally, it is important to notice that, as $\mathbf{U}_{\varepsilon}(0)=\overline{\mathbf{U}}_{\varepsilon}(0)$,
 \begin{align*}
\Vert \rho_{\varepsilon}(0)-\rho_0\Vert_{H^3(\mathbb{R}^3)}+  \Vert \mathbf{v}_{\varepsilon}(0)\Vert_{H^3(\mathbb{R}^3)}= &\Vert \overline{\rho}_{\varepsilon}(0)-\rho_0\Vert_{H^3(\mathbb{R}^3)}+  \Vert \overline{ \mathbf{v}}_{\varepsilon}(0)\Vert_{H^3(\mathbb{R}^3)}\\
\leq & C ( \Vert u_0 \Vert_{H^5(\mathbb{R}^3)}+  \Vert u_1 \Vert_{H^4(\mathbb{R}^3)} ).
\end{align*}
Thus, there exists $k>0$ (necessarily $k\leq k_2$) such that $\Vert u_0\Vert_{H^5}+\Vert u_1\Vert_{H^4}<k$ implies the global existences of $\mathbf{U}_{\varepsilon}$ and $\overline{\mathbf{U}}_{\varepsilon}$.
\end{proof}

The stability estimate which we obtain between the exact solution of the Navier-Stokes system $\mathbf{U}_{\varepsilon}$ and the solution of the Kuznetsov equation presented by  $\overline{\mathbf{U}}_{\varepsilon}$ does not require for $\mathbf{U}_{\varepsilon}$ to have the regularity of a classical solution and allows to approximate  less regular solutions of the Navier-Stokes system with initial data in a small $L^2$ neighborhood of $\overline{\mathbf{U}}_{\varepsilon}(0)$. To define the minimal regularity property of $\mathbf{U}_{\varepsilon}$ for which  stability estimate~(\ref{validaproxintro}) holds, we introduce admissible weak solutions of a bounded energy using the entropy of the Euler system (system~(\ref{NSmatr}) with $\nu=0$)
\begin{equation}\label{entropy}
\eta(\mathbf{U}_{\varepsilon})=\rho_{\varepsilon}h(\rho_{\varepsilon})+\rho_{\varepsilon} \frac{\mathbf{v}_{\varepsilon}^2}{2}=H(\rho_{\varepsilon})+\frac{1}{\rho_{\varepsilon}}\frac{\mathbf{m}^2}{2},
\end{equation}
which is convex~\cite{Dafermos}
with $h'(\rho_{\varepsilon})=\frac{p(\rho_{\varepsilon})}{\rho_{\varepsilon}^2}$, $\mathbf{v}_{\varepsilon}=\frac{\mathbf{m}}{\rho_{\varepsilon}}$. Thus,  the first and second derivatives of $\eta$ are~\cite{Roz3}
\begin{equation}\label{firstderent}
\eta'(\mathbf{U}_{\varepsilon})=\begin{bmatrix}
H'(\rho_{\varepsilon})-\frac{1}{\rho_{\varepsilon}^2}\frac{\mathbf{m}^2}{2}\\
\frac{\mathbf{m}}{\rho_{\varepsilon}}
\end{bmatrix}^t
=\begin{bmatrix}
H'(\rho_{\varepsilon})-\frac{\mathbf{v}_{\varepsilon}^2}{2}\\
\mathbf{v}_{\varepsilon}
\end{bmatrix}^t,
\end{equation}
\begin{equation}\label{secderent}
\eta''(\mathbf{U}_{\varepsilon})=\begin{bmatrix}
H''(\rho_{\varepsilon})+\frac{\mathbf{m}^2}{\rho_{\varepsilon}^3} & -\frac{\mathbf{m}}{\rho_{\varepsilon}^2}\\
-\frac{\mathbf{m}}{\rho_{\varepsilon}^2} & \frac{1}{\rho_{\varepsilon}}
\end{bmatrix}
=\begin{bmatrix}
H''(\rho_{\varepsilon})+\frac{\mathbf{v}_{\varepsilon}^2}{\rho_{\varepsilon}} & - \frac{\mathbf{v}_{\varepsilon}}{\rho_{\varepsilon}}\\
- \frac{\mathbf{v}_{\varepsilon}}{\rho_{\varepsilon}} & \frac{1}{\rho_{\varepsilon}}
\end{bmatrix},
\end{equation}
 knowing that $\eta''(\mathbf{U}_{\varepsilon})$ is strictly positive defined.

\begin{definition}\label{DefAdm}
 The  function $\mathbf{U}_{\eps}=(\rho_{\eps}, \rho_{\eps}\mathbf{v}_{\eps})$ is called an admissible weak
solution of a bounded energy of the Cauchy problem for the Navier-Stokes system~(\ref{NSi1})--(\ref{press}) if it satisfies the following
properties:
\begin{enumerate}
    \item The pair $(\rho_{\eps}, \mathbf{v}_{\eps})$ is a weak  solution of the Cauchy problem for the Navier-Stokes system~(\ref{NSi1})--(\ref{press}) (in the distributional sense).
\item The function $\mathbf{U}_{\eps}$ 
satisfies in the sense of distributions (see Ref.~\cite[p.52]{Dafermos})
\begin{equation}\label{NSentropyweak}
  \del_t \eta (\mathbf{U}_{\eps}) +\nabla.
\textbf{q}(\mathbf{U}_{\eps})-\eps\nu \mathbf{v}_{\eps}\triangle \mathbf{v}_{\eps}\le 0, 
\hbox{ where }
  \textbf{q}(\mathbf{U}_\eps)=\mathbf{v}_{\eps} (\eta(\mathbf{U}_{\eps}) + p(\rho_{\eps})),
 \end{equation}
or equivalently, for any positive
 test function $\psi$  in $\mathcal{D}(\R^n\times[0,\infty[)$
the function $\mathbf{U}_{\eps}$ 
satisfies
\begin{eqnarray*}&&\int_0^T
\int_{\R^n} \left(\del_t \psi \eta(\mathbf{U}_\eps) +\nabla \psi.
\textbf{q}(\mathbf{U}_\eps)+\eps\nu|\nabla.\mathbf{v}_\eps|^2
\psi +\eps \nu
\mathbf{v}_\eps.[\nabla.\mathbf{v}_\eps\nabla \psi]\right) dxdt\\
&&+\int_{\R^n} \psi(x,0) \eta(\mathbf{U}_\eps(0))dx\ge0.
\end{eqnarray*}
    \item The function $\mathbf{U}_{\eps}$  
  satisfies the equality (with the notation $\mathbf{v}_\eps=(v_1,\ldots,v_n)$)
\begin{eqnarray*}
&&-\int_{\R^n} \frac{\mathbf{U}^2_\epsilon(t)}{2}dx +\int_0^t
\int_{\R^n}\left(\sum_{i=1}^n  \mathbf{G}_i(\mathbf{U}_\varepsilon)\partial_{x_i} \mathbf{U}_\epsilon -\epsilon
\nu \nabla(\rho_\eps v_i).\nabla v_i\right) dxds\\
&&+\int_{\R^n}
 \frac{\mathbf{U}^2_\epsilon(0)}{2}dx=0.
\end{eqnarray*}
\end{enumerate}
\end{definition}

Let us notice that any classical solution of~(\ref{NSmatr}), for instance the solution defined in Theorem~\ref{ThExistUUbarNSK}, satisfies the entropy condition~(\ref{NSentropyweak}) by the equality and obviously it is sufficient regular to perform the integration by parts resulting in the relation of point~3.
For  existence results of global weak solutions of the Cauchy problem for the Navier-Stokes system~(\ref{NSmatr}) with sufficiently small initial data around the constant state $(\rho_0,0)$ (actually, $\rho_0-\rho(0)$ is small in $L^\infty$, $\mathbf{v}(0)$ is small in $L^2$ and bounded in $L^{2^n}$) and with the pressure $p(\rho)=K\rho^\gamma$ with $\gamma\ge 1$, we refer to results of D. Hoff\cite{Hoff1,Hoff2}.
For fixing the idea of the regularity of a global weak solution we summarize the results of Hoff in the following theorem:
\begin{theorem}\cite{Hoff1}\label{ThHoff1}
Let for $n=3$ $\beta=0$ and for $n=2$ $\beta $ be arbitrary small, $N$ be a given arbitrary large constant.
There exists a constant $C_0>0$ such that  if  the initial data of~(\ref{NSmatr}) with $p(\rho)=K\rho^\gamma$ ($\gamma\ge 1$) satisfies the following smallness condition
\begin{align*}
 &\|\rho_0-\rho(0)\|^2_{L^\infty(\R^n)}+\int_{\R^n}\left[ (\rho_0-\rho(0))^2+|\mathbf{v}(0)|^2\right](1+|x|^2)^\beta\dx\le C_0,\\
 & \|\mathbf{v}(0)\|_{L^{2^n}(\R^n)}\le N,
\end{align*}
 then there exists a global weak solution $(\rho,\mathbf{v})$ (in the distributional sense)
 such that
 \begin{enumerate}
  \item $\rho-\rho_0\in L^\infty(\R^n\times[0,\infty[)$,
  \item $\mathbf{v}\in H^1(\R^n)$ for all $t>0$,
  \item for all $t\ge \tau >0$ $\mathbf{v}(\cdot,t)\in L^\infty(\R^n)$, 
  \item for all $\tau>0$ $\mathbf{v}\in C^{\alpha,\frac{\alpha}{2\alpha+2}}(\R^n\times[\tau,\infty[)$ for all $\alpha\in ]0,1[$ when $n=2$ and $\mathbf{v}\in C^{\frac{1}{2},\frac{1}{8}}(\R^n\times[\tau,\infty[)$ when $n=3$,
  \item $\eps\nu\operatorname{div} \mathbf{v}+p(\rho)-p(\rho_0)\in H^1(\R^n)\cap C^\alpha(\R^n)$ for almost all $t>0$ with $\alpha=\frac{1}{2}$ for $n=2$ and $\alpha=\frac{1}{10}$ when $n=3$.
 \end{enumerate}
  In addition, $(\rho, \mathbf{v})\to (\rho_0,0)$ as $t\to+\infty$ in the sense that for all $q\in ]2,+\infty[$
 $$\lim_{T\to\infty}\left(\|\rho-\rho_0\|_{L^\infty(\R^n\times[T,\infty[)}+\|\mathbf{v}(\cdot,T)\|_{L^q(\R^n)} \right)=0.$$
\end{theorem}
Therefore, from Theorem~\ref{ThHoff1} it follows that a weak solution of the isentropic compressible Navier-Stokes system~(\ref{NSi1})--(\ref{press}) is also admissible weak solution of a bounded energy in the sense of Definition~\ref{DefAdm}. 
But in the following we only consider the question of the validity of
the stability estimate~(\ref{validaproxintro}) for initial data closed to $\overline{\mathbf{U}}_{\varepsilon}(0)$ in $L^2$ norm (thus for initial data not necessarily satisfying Theorem~\ref{ThHoff1}) and we don't consider  the existence question of an admissible weak
solution of a bounded energy of the Cauchy problem for the Navier-Stokes system. 
Thanks to Theorem~\ref{ThExistUUbarNSK} for classical solutions of two models and to Definition~\ref{DefAdm} containing the minimal conditions on $\mathbf{U}_{\varepsilon}$ necessary for saying that it is in a small $L^2$-neighborhood of the regular solution of the Kuznetsov equation, we validate the approximation of $\mathbf{U}_{\varepsilon}$ by $\overline{\mathbf{U}}_{\varepsilon}$ following the ideas of Ref.~\cite{Roz3}. 
\begin{theorem}\label{ThapproxNSKuz}
Let 
$\nu>0$ and $\eps>0$ be fixed and all assumptions of Theorem~\ref{ThExistUUbarNSK} hold. 
Then  there exist constant $C>0$ and $K>0$, independent on $\eps$ and the time $t$, such that  
 \begin{enumerate}
  \item 
  for all $t\leq\frac{C}{\varepsilon}$
$$\Vert (\mathbf{U}_{\varepsilon}-\overline{\mathbf{U}}_{\varepsilon})(t)\Vert_{L^2(\mathbb{R}^3)}^2\leq K\varepsilon^3 t  e^{K\varepsilon t}\leq 4 \varepsilon^2;$$
\item 
for all $b\in ]0,1[$ during  all time $t\leq\frac{C}{\varepsilon} \ln(\frac{1}{\varepsilon})$ it holds
$$\Vert (\mathbf{U}_{\varepsilon}-\overline{\mathbf{U}}_{\varepsilon})(t)\Vert_{L^2(\mathbb{R}^3)}\leq 2\varepsilon^{b}.$$
 \end{enumerate}

 Moreover, if the initial conditions for the Kuznetsov equation are such that
 $$u_0\in H^{s+2}(\R^n), \quad u_1\in H^{s+1}(\R^n) \hbox{ for } s>\frac{n}{2},\; n\ge 2$$
 and sufficiently small (in the sense of Ref.~\cite{Perso}~Theorem~1.2),
 then there exists the unique global in time solution of the Cauchy problem for the Kuznetsov equation 
 \begin{align}
\overline{\rho}_\varepsilon-\rho_0 &\in C([0,+\infty[;H^s(\mathbb{R}^n))\cap C^1([0,+\infty[;H^{s-1}(\mathbb{R}^n)),\label{EqRegMinRhoNS}\\
\overline{\textbf{v}}_{\varepsilon} &\in C([0,+\infty[;H^{s+1}(\mathbb{R}^n))\cap C^1([0,+\infty[;H^{s}(\mathbb{R}^n))\label{EqRegMinUNS}
\end{align}
 and the remainder terms $(R_1^{NS-Kuz},\mathbf{R
}_2^{NS-Kuz})$, defined in Eqs.~(\ref{rNS1})--(\ref{rNS2}), belong to $C([0,+\infty[,H^{s-1}(\R^n))$.

If in addition there exists an admissible weak
solution of a bounded energy of the Cauchy problem for the Navier-Stokes system~(\ref{NSmatr}) (for instance if $\mathbf{U}_{\varepsilon}(0)$ satisfies conditions of Theorem~\ref{ThHoff1} there is a global such weak solution) on a time interval $[0,T_{NS}[$ for the initial data
\begin{equation*}
   \|\mathbf{U}_{\varepsilon}(0)-\overline{\mathbf{U}}_{\varepsilon}(0)\|_{L^2(\R^n)}\le \delta\le \eps,                                                                                                                                                                                                                                     \end{equation*}
then it holds for all $t< \min\{\frac{C}{\varepsilon}, T_{NS}\}$ the stability estimate~(\ref{validaproxintro}):
$$\Vert (\mathbf{U}_{\varepsilon}-\overline{\mathbf{U}}_{\varepsilon})(t)\Vert_{L^2(\mathbb{R}^n)}^2\leq K(\varepsilon^3 t +\delta^2) e^{K\varepsilon t}\leq 9 \varepsilon^2.$$
\end{theorem}
\begin{proof}
In terms of entropy system~(\ref{NSmatrA}), having by the assumption the unique classical solution $\overline{\mathbf{U}}_{\varepsilon}$, can be rewritten as follows

\begin{equation}\label{approxNs+entr}
\partial_t \eta(\overline{\mathbf{U}}_{\varepsilon})+\nabla.\mathbf{q}(\overline{\mathbf{U}}_{\varepsilon})-\varepsilon \nu \overline{ \mathbf{v}}_{\varepsilon}.\Delta \overline{ \mathbf{v}}_{\varepsilon}=\varepsilon^3\left(\frac{\eta(\overline{\mathbf{U}}_{\varepsilon})+p(\overline{\rho}_{\varepsilon})}{\overline{\rho}_{\varepsilon}}R_1^{NS-Kuz}+\overline{\mathbf{v}}_{\varepsilon}.\mathbf{R}_2^{NS-Kuz} \right),
\end{equation}
with $\mathbf{R}^{NS-Kuz}=(R_1^{NS-Kuz},\mathbf{R}_2^{NS-Kuz})$ defined in Eq.~(\ref{rNS1})-(\ref{rNS2}). To abbreviate the notations, we denote the remainder term of the entropy equation in system~(\ref{approxNs+entr}) by 
$$\overline{R}^{NS-Kuz}=\left(\frac{\eta(\overline{\mathbf{U}}_{\varepsilon})+p(\overline{\rho}_{\varepsilon})}{\overline{\rho}_{\varepsilon}}R_1^{NS-Kuz}+\overline{\mathbf{v}}_{\varepsilon}.\mathbf{R}_2^{NS-Kuz} \right).$$
In the same time, it is assumed that for $\mathbf{U}_{\varepsilon}$ it holds~(\ref{NSentropyweak}) in the sense of distributions.

Let us estimate in the sense of distributions
\begin{equation}\label{Eqexpression}
 \frac{\partial}{\partial t}\Big(\eta(\mathbf{U}_{\varepsilon})-\eta(\overline{\mathbf{U}}_{\varepsilon})-\eta'(\overline{\mathbf{U}}_{\varepsilon})(\mathbf{U}_{\varepsilon}-\overline{\mathbf{U}}_{\varepsilon})\Big).
\end{equation}
First we find from systems~(\ref{NSentropyweak}) and~(\ref{approxNs+entr}) that in the sense of distributions
\begin{align*}
\frac{\partial}{\partial t}(\eta(\mathbf{U}_{\varepsilon})-\eta(\overline{\mathbf{U}}_{\varepsilon}))\le
 &-\nabla.(\mathbf{q}(\mathbf{U}_{\varepsilon})-\mathbf{q}(\overline{\mathbf{U}}_{\varepsilon}))+\varepsilon\nu(\mathbf{v}_{\varepsilon}.\Delta \mathbf{v}_{\varepsilon}- \overline{ \mathbf{v}}_{\varepsilon}.\Delta \overline{ \mathbf{v}}_{\varepsilon})-\varepsilon^3 \overline{R}^{NS-Kuz}\\
 =&-\nabla.(\mathbf{q}(\mathbf{U}_{\varepsilon})-\mathbf{q}(\overline{\mathbf{U}}_{\varepsilon}))+\varepsilon\nu \sum_{i=1}^n \partial_{x_i}(\mathbf{v}_{\varepsilon}\partial_{x_i}\mathbf{v}_{\varepsilon}- \overline{ \mathbf{v}}_{\varepsilon}\partial_{x_i} \overline{ \mathbf{v}}_{\varepsilon})\\
 &- \varepsilon\nu \sum_{i=1}^n (\partial_{x_i}\mathbf{v}_{\varepsilon}\partial_{x_i}\mathbf{v}_{\varepsilon}-\partial_{x_i} \overline{ \mathbf{v}}_{\varepsilon}\partial_{x_i} \overline{ \mathbf{v}}_{\varepsilon})-\varepsilon^3 \overline{R}^{NS-Kuz}.
\end{align*}
Then we notice that
$$-\frac{\partial}{\partial t}(\eta'(\overline{\mathbf{U}}_{\varepsilon})(\mathbf{U}_{\varepsilon}-\overline{\mathbf{U}}_{\varepsilon}))=- \partial_t \overline{\mathbf{U}}_{\varepsilon}^t \eta''(\overline{\mathbf{U}}_{\varepsilon})(\mathbf{U}_{\varepsilon}-\overline{\mathbf{U}}_{\varepsilon})-\eta'(\overline{\mathbf{U}}_{\varepsilon}) (\partial_t \mathbf{U}_{\varepsilon}-\partial_t \overline{\mathbf{U}}_{\varepsilon}),$$
where in the sense of distributions
\begin{align*}
- \partial_t \overline{\mathbf{U}}_{\varepsilon}^t \eta''(\overline{\mathbf{U}}_{\varepsilon})(\mathbf{U}_{\varepsilon}-\overline{\mathbf{U}}_{\varepsilon})=& -\left[-\sum_{i=1}^n  D\mathbf{G}_i(\overline{\mathbf{U}}_\varepsilon)\partial_{x_i} \overline{\mathbf{U}}_{\varepsilon}\right]^t\eta''(\overline{\mathbf{U}}_{\varepsilon})(\mathbf{U}_{\varepsilon}-\overline{\mathbf{U}}_{\varepsilon})\\
&-\left(\begin{bmatrix}
0 \\
\varepsilon\nu\Delta \overline{ \mathbf{v}}_{\varepsilon}
\end{bmatrix} +\varepsilon^3 \mathbf{R}^{NS-Kuz}\right)^t\eta''(\overline{\mathbf{U}}_{\varepsilon})(\mathbf{U}_{\varepsilon}-\overline{\mathbf{U}}_{\varepsilon}),
\end{align*}
and 
\begin{align*}
-\eta'(\overline{\mathbf{U}}_{\varepsilon}) (\partial_t \mathbf{U}_{\varepsilon}-\partial_t \overline{\mathbf{U}}_{\varepsilon})= &- \eta'(\overline{\mathbf{U}}_{\varepsilon})(-\sum_{i=1}^n \partial_{x_i} (\mathbf{G}_i(\mathbf{U}_\varepsilon)-\mathbf{G}_i(\overline{\mathbf{U}}_{\varepsilon})))\\
&- \eta'(\overline{\mathbf{U}}_{\varepsilon}) \varepsilon \nu \begin{bmatrix}
0\\
\Delta \mathbf{v}_{\varepsilon}-\Delta \overline{ \mathbf{v}}_{\varepsilon}
\end{bmatrix}
+\varepsilon^3 \eta'(\overline{\mathbf{U}}_{\varepsilon}) \mathbf{R}^{NS-Kuz}\\
=& \sum_{i=1}^n \partial_{x_i}(\eta'(\overline{\mathbf{U}}_{\varepsilon}) (\mathbf{G}_i(\mathbf{U}_\varepsilon)-\mathbf{G}_i(\overline{\mathbf{U}}_{\varepsilon}))\\
&-\sum_{i=1}^n \partial_{x_i}\overline{U}^t \eta''(\overline{\mathbf{U}}_{\varepsilon})(\mathbf{G}_i(\mathbf{U}_{\varepsilon})-\mathbf{G}_i(\overline{\mathbf{U}}_{\varepsilon}))\\
&- \eta'(\overline{\mathbf{U}}_{\varepsilon}) \varepsilon \nu \begin{bmatrix}
0\\
\Delta \mathbf{v}_{\varepsilon}-\Delta \overline{ \mathbf{v}}_{\varepsilon}
\end{bmatrix}
+\varepsilon^3 \eta'(\overline{\mathbf{U}}_{\varepsilon}) \mathbf{R}^{NS-Kuz}.
\end{align*}
Thanks to the convex property of the entropy we have 
$$\eta''(\mathbf{U})D\mathbf{G}_i(\mathbf{U})=(D\mathbf{G}_i(\mathbf{U}))^t\eta''(\mathbf{U}),$$ and consequently
\begin{align*}
(D\mathbf{G}_i(\overline{\mathbf{U}}_\varepsilon)\partial_{x_i} \overline{\mathbf{U}}_{\varepsilon})^t\eta''(\overline{\mathbf{U}}_{\varepsilon})(\mathbf{U}_{\varepsilon}-\overline{\mathbf{U}}_{\varepsilon})=& \partial_{x_i} \overline{\mathbf{U}}_{\varepsilon}^t(D\mathbf{G}_i(\overline{\mathbf{U}}_\varepsilon))^t\eta''(\overline{\mathbf{U}}_{\varepsilon})(\mathbf{U}_{\varepsilon}-\overline{\mathbf{U}}_{\varepsilon})\\
=&
\partial_{x_i} \overline{\mathbf{U}}_{\varepsilon}^t \eta''(\overline{\mathbf{U}}_{\varepsilon})D\mathbf{G}_i(\overline{\mathbf{U}}_\varepsilon)  (\mathbf{U}_{\varepsilon}-\overline{\mathbf{U}}_{\varepsilon}).
\end{align*}
Moreover, we compute in the sense of distributions
\begin{align*}
-&\begin{bmatrix}
0 \\
\varepsilon\nu\Delta \overline{ \mathbf{v}}_{\varepsilon}
\end{bmatrix} ^t\eta''(\overline{\mathbf{U}}_{\varepsilon})(\mathbf{U}_{\varepsilon}-\overline{\mathbf{U}}_{\varepsilon})=  -\varepsilon \nu \Delta \overline{ \mathbf{v}}_{\varepsilon}(\mathbf{v}_{\varepsilon}- \overline{ \mathbf{v}}_{\varepsilon})-\varepsilon \nu \Delta \overline{ \mathbf{v}}_{\varepsilon} \frac{\rho_{\varepsilon}-\overline{\rho}_{\varepsilon}}{\overline{\rho}_{\varepsilon}}(\mathbf{v}_{\varepsilon}- \overline{ \mathbf{v}}_{\varepsilon})\\
= & -\varepsilon \nu \sum_{i=1}^n \partial_{x_i}(\partial_{x_i} \overline{ \mathbf{v}}_{\varepsilon}(\mathbf{v}_{\varepsilon}- \overline{ \mathbf{v}}_{\varepsilon} ))
+\varepsilon \nu \sum_{i=1}^n \partial_{x_i} \overline{ \mathbf{v}}_{\varepsilon}\partial_{x_i}(\mathbf{v}_{\varepsilon}- \overline{ \mathbf{v}}_{\varepsilon} )
-\varepsilon \nu \Delta \overline{ \mathbf{v}}_{\varepsilon} \frac{\rho_{\varepsilon}-\overline{\rho}_{\varepsilon}}{\overline{\rho}_{\varepsilon}}(\mathbf{v}_{\varepsilon}- \overline{ \mathbf{v}}_{\varepsilon}),
\end{align*}
and
\begin{align*}
- \eta'(\overline{\mathbf{U}}_{\varepsilon}) \varepsilon \nu \begin{bmatrix}
0\\
\Delta \mathbf{v}_{\varepsilon}-\Delta \overline{ \mathbf{v}}_{\varepsilon}
\end{bmatrix}
= & -\varepsilon \nu \overline{ \mathbf{v}}_{\varepsilon}.(\Delta \mathbf{v}_{\varepsilon}-\Delta \overline{ \mathbf{v}}_{\varepsilon})\\
=& - \varepsilon \nu \sum_{i=1}^n \partial_{x_i}(   \overline{ \mathbf{v}}_{\varepsilon} \partial_{x_i}(\mathbf{v}_{\varepsilon}- \overline{ \mathbf{v}}_{\varepsilon} ))+ \varepsilon \nu \sum_{i=1}^n \partial_{x_i} \overline{ \mathbf{v}}_{\varepsilon}\partial_{x_i}(\mathbf{v}_{\varepsilon}- \overline{ \mathbf{v}}_{\varepsilon} ).
\end{align*}
We integrate over $\mathbb{R}^n$ expression~(\ref{Eqexpression}) and notice that
the integrals of the terms in divergence form in the development of~(\ref{Eqexpression}) are equal to zero. For the regular case in the framework of Theorem~\ref{ThExistUUbarNSK} it is due to the regularity given by~(\ref{EqregRhoNS}) and~(\ref{EqregvNS}) and the following Sobolev embedding\cite{Adams}
\begin{equation}
H^s(\mathbb{R}^n)\hookrightarrow C_0(\mathbb{R}^n):=\lbrace f\in C(\mathbb{R}^n)\vert\;\; \vert f(x)\vert\rightarrow 0\hbox{ as } \Vert x\Vert\rightarrow +\infty\rbrace\hbox{ for }s>\frac{n}{2},
\end{equation}
which allows us to use the fact that 
 $$\forall f\in C_0(\mathbb{R}^n),\hbox{ }\int_{\mathbb{R}^n} \nabla.f(x) \;dx=0.$$
 In the case of a weak admissible solution $\mathbf{U}_{\varepsilon}$ it follows from its bounded energy property (see Definition~\ref{DefAdm} point~3) which implies that $\rho_\eps-\rho_0$ and $\mathbf{v}_\eps$ tend to $0$ for $|x|\to+\infty$ and also implies the existence of the integrals over $\R^n$. 
Therefore, 
we obtain the following estimate in which each term is well-defined in the sense of distributions on $[0,+\infty[\cap [0,T_{NS}]$
\begin{align}
&\frac{d}{dt}\int_{\mathbb{R}^3} \eta(\mathbf{U}_{\varepsilon})-\eta(\overline{\mathbf{U}}_{\varepsilon})-\eta'(\overline{\mathbf{U}}_{\varepsilon})(\mathbf{U}_{\varepsilon}-\overline{\mathbf{U}}_{\varepsilon})\dx\leq\nonumber\\
&-\sum_{i=1}^3\int_{\mathbb{R}^3} \partial_{x_i}\overline{U}^t \eta''(\overline{\mathbf{U}}_{\varepsilon})(\mathbf{G}_i(\mathbf{U}_{\varepsilon})-\mathbf{G}_i(\overline{\mathbf{U}}_{\varepsilon})-D\mathbf{G}_i(\overline{\mathbf{U}}_{\varepsilon})(\mathbf{U}_{\varepsilon}-\overline{\mathbf{U}}_{\varepsilon}))\dx\nonumber\\
&-\varepsilon \nu \int_{\mathbb{R}^3} \sum_{i=1}^3 (\partial_{x_i}\mathbf{v}_{\varepsilon}\partial_{x_i}\mathbf{v}_{\varepsilon}-\partial_{x_i} \overline{ \mathbf{v}}_{\varepsilon}\partial_{x_i} \overline{ \mathbf{v}}_{\varepsilon})\dx\label{dvliment}\\
&+2\varepsilon \nu \int_{\mathbb{R}^3}\sum_{i=1}^3 \partial_{x_i} \overline{ \mathbf{v}}_{\varepsilon}\partial_{x_i}(\mathbf{v}_{\varepsilon}- \overline{ \mathbf{v}}_{\varepsilon} )\dx
+\varepsilon \nu \int_{\mathbb{R}^3} \Delta \overline{ \mathbf{v}}_{\varepsilon} \frac{\rho_{\varepsilon}-\overline{\rho}_{\varepsilon}}{\overline{\rho}_{\varepsilon}}(\mathbf{v}_{\varepsilon}- \overline{ \mathbf{v}}_{\varepsilon})\dx\nonumber\\
&-\varepsilon^3 \int_{\mathbb{R}^3} (\overline{R}^{NS-Kuz}-\eta'(\overline{\mathbf{U}}_{\varepsilon})\mathbf{R}^{NS-Kuz})\dx
-\varepsilon^3 \int_{\mathbb{R}^3}[\mathbf{R}^{NS-Kuz}]^t \eta''(\overline{\mathbf{U}}_{\varepsilon})(\mathbf{U}_{\varepsilon}-\overline{\mathbf{U}}_{\varepsilon})\dx.\nonumber
\end{align}
Now we study lower bounds of the left hand side and upper bounds of the right hand side of~(\ref{dvliment}) in order to obtain a suitable estimate.
For the right hand side of Eq.~(\ref{dvliment}) we notice that
\begin{align*}
-\varepsilon \nu \int_{\mathbb{R}^3} \sum_{i=1}^3 (\partial_{x_i}\mathbf{v}_{\varepsilon}\partial_{x_i}\mathbf{v}_{\varepsilon}-\partial_{x_i} \overline{ \mathbf{v}}_{\varepsilon}\partial_{x_i} \overline{ \mathbf{v}}_{\varepsilon})\dx
&+2\varepsilon \nu \int_{\mathbb{R}^3}\sum_{i=1}^3 \partial_{x_i} \overline{ \mathbf{v}}_{\varepsilon}\partial_{x_i}(\mathbf{v}_{\varepsilon}- \overline{ \mathbf{v}}_{\varepsilon} )\dx\\
=& -\varepsilon \nu \int_{\mathbb{R}^3}\sum_{i=1}^3(\partial_{x_i}(\mathbf{v}_{\varepsilon}- \overline{ \mathbf{v}}_{\varepsilon}))^2\dx\leq 0,
\end{align*}
hence this term can be passed in the left hand side of Eq.(\ref{dvliment}) and omitted in the estimation.
As the entropy is convex it holds
$$\exists \delta_0>0:\;\;\;\;\eta(\mathbf{U}_{\varepsilon})-\eta(\overline{\mathbf{U}}_{\varepsilon})-\eta'(\overline{\mathbf{U}}_{\varepsilon})(\mathbf{U}_{\varepsilon}-\overline{\mathbf{U}}_{\varepsilon})\geq \delta_0 \vert \mathbf{U}_{\varepsilon}-\overline{\mathbf{U}}_{\varepsilon} \vert^2.$$
Then using also its continuity, we find 
\begin{align*}
 \delta_0 \int_{\mathbb{R}^3} \vert \mathbf{U}_{\varepsilon}-\overline{\mathbf{U}}_{\varepsilon} \vert^2(t)\dx \leq \int_0^t \frac{d}{ds}\left(\int_{\mathbb{R}^3} \eta(\mathbf{U}_{\varepsilon})-\eta(\overline{\mathbf{U}}_{\varepsilon})-\eta'(\overline{\mathbf{U}}_{\varepsilon})(\mathbf{U}_{\varepsilon}-\overline{\mathbf{U}}_{\varepsilon})\dx\right)ds\\
 + C_0\int_{\mathbb{R}^3} \vert \mathbf{U}_{\varepsilon}-\overline{\mathbf{U}}_{\varepsilon} \vert^2(0)\dx.
\end{align*}
On the right hand side of (\ref{dvliment}), by the Taylor expansion we  also have
$$\mathbf{G}_i(\mathbf{U}_{\varepsilon})-\mathbf{G}_i(\overline{\mathbf{U}}_{\varepsilon})-D\mathbf{G}_i(\overline{\mathbf{U}}_{\varepsilon})(\mathbf{U}_{\varepsilon}-\overline{\mathbf{U}}_{\varepsilon})\leq C \vert \mathbf{U}_{\varepsilon}-\overline{\mathbf{U}}_{\varepsilon} \vert^2.$$
With the boundness on $[0;+\infty[$ of $R_1(t)$ and $\mathbf{R}_2(t)$ in the $L^2$ and $L^{\infty}$ norms, and thanks to the regularity of $\overline{\mathbf{U}}_{\varepsilon}$ defined in~(\ref{EqRegMinRhoNS}) and~(\ref{EqRegMinUNS}) (see also~(\ref{EqregRhoNS}) and~(\ref{EqregvNS}) for the case $\mathbf{U}_{\varepsilon}(0)=\overline{\mathbf{U}}_{\varepsilon}(0)$)  and the energy boundedness of $\mathbf{U}_{\varepsilon}$, we estimate the other terms in Eq.~(\ref{dvliment}) in the following way
\begin{align*}
\varepsilon \nu \int_{\mathbb{R}^3}& \Delta \overline{ \mathbf{v}}_{\varepsilon} \frac{\rho_{\varepsilon}-\overline{\rho}_{\varepsilon}}{\overline{\rho}_{\varepsilon}}(\mathbf{v}_{\varepsilon}- \overline{ \mathbf{v}}_{\varepsilon})\dx \leq K \varepsilon \Vert \mathbf{U}_{\varepsilon}-\overline{\mathbf{U}}_{\varepsilon} \Vert_{L^2(\mathbb{R}^3)}^2 ,\\
-\varepsilon^3 \int_{\mathbb{R}^3}& (\overline{R}^{NS-Kuz}-\eta'(\overline{\mathbf{U}}_{\varepsilon})\mathbf{R}^{NS-Kuz})\dx\leq K \varepsilon^3, \\
-\varepsilon^3 \int_{\mathbb{R}^3}&[\mathbf{R}^{NS-Kuz}]^t \eta''(\overline{\mathbf{U}}_{\varepsilon})(\mathbf{U}_{\varepsilon}-\overline{\mathbf{U}}_{\varepsilon})\dx\\
&\leq  \varepsilon^3 \Vert \eta''(\overline{\mathbf{U}}_{\varepsilon})\Vert_{L^{\infty}(\mathbb{R}^3)} \Vert \mathbf{R}^{NS-Kuz}\Vert_{L^2(\mathbb{R}^3)} \Vert \mathbf{U}_{\varepsilon}-\overline{\mathbf{U}}_{\varepsilon} \Vert_{L^2(\mathbb{R}^3)}\\
&\leq  K \varepsilon^3 \Vert \mathbf{U}_{\varepsilon}-\overline{\mathbf{U}}_{\varepsilon} \Vert_{L^2(\mathbb{R}^3)}.
\end{align*}

Now, by integrating on $[0,t]$, we obtain from~(\ref{dvliment}) the following inequality
\begin{align*}
\int_{\mathbb{R}^3} \vert \mathbf{U}_{\varepsilon}-\overline{\mathbf{U}}_{\varepsilon} \vert^2(t)&\dx\leq  \int_0^t\Big[  (C\Vert \nabla \overline{\mathbf{U}}_{\varepsilon}\Vert_{L^{\infty}}+K \varepsilon)  \Vert \mathbf{U}_{\varepsilon}-\overline{\mathbf{U}}_{\varepsilon} \Vert_{L^2(\mathbb{R}^3)}^2\\
& +K\varepsilon^3+K\varepsilon^3\Vert \mathbf{U}_{\varepsilon}-\overline{\mathbf{U}}_{\varepsilon} \Vert_{L^2(\mathbb{R}^3)} \Big]ds+ C_1\int_{\mathbb{R}^3} \vert \mathbf{U}_{\varepsilon}-\overline{\mathbf{U}}_{\varepsilon} \vert^2(0)\dx.
\end{align*}
Here $K$, $C$ and $C_1$ are generic constants of order $O(\eps^0)$ which do not depend on time.
Using once more the regularity properties~(\ref{EqregRhoNS}) and ~(\ref{EqregvNS}), we have the boundness of~$\Vert \nabla \overline{\mathbf{U}}_{\varepsilon}\Vert_{L^{\infty}} $. 
But knowing that $\overline{\rho}_{\varepsilon}$ and $\overline{\textbf{v}}_{\varepsilon}$ are defined by \textit{ansatz} ~(\ref{rhoKuz})--(\ref{uKuz}), we deduce that  
 $\Vert \nabla \overline{\mathbf{U}}_{\varepsilon}\Vert_{L^{\infty}}\leq C \varepsilon.$
Therefore,
\begin{align*}
\Vert \mathbf{U}_{\varepsilon}-\overline{\mathbf{U}}_{\varepsilon} \Vert_{L^2}^2 & 
\leq  \int_0^t K \Big(  \varepsilon \Vert \mathbf{U}_{\varepsilon}-\overline{\mathbf{U}}_{\varepsilon} \Vert_{L^2(\mathbb{R}^3)}^2 +\varepsilon^3+\varepsilon^3\Vert \mathbf{U}_{\varepsilon}-\overline{\mathbf{U}}_{\varepsilon} \Vert_{L^2(\mathbb{R}^3)} \Big)ds\\
&+ C_1\int_{\mathbb{R}^3} \vert \mathbf{U}_{\varepsilon}-\overline{\mathbf{U}}_{\varepsilon} \vert^2(0)\dx.
\end{align*}
Then applying the Gronwall Lemma we  have directly $$\Vert (\mathbf{U}_{\varepsilon}-\overline{\mathbf{U}}_{\varepsilon})(t)\Vert_{L^2(\mathbb{R}^3)}^2\leq K(\varepsilon^3 t +\delta^2) e^{K\varepsilon t},$$
since $K\eps t$ is a non-decreasing in time function and $\varepsilon^3\sqrt{v} < K\eps v$ for all $v\in \R^+$. In addition, to find the estimate of  Point~2 for the regular case $\mathbf{U}_{\varepsilon}(0)=\overline{\mathbf{U}}_{\varepsilon}(0)$, we notice that
$$\Vert \mathbf{U}_{\varepsilon}-\overline{\mathbf{U}}_{\varepsilon} \Vert_{L^2(\mathbb{R}^3)}\leq v,$$
where $v$ is the solution of the following Cauchy problem 
$$\left\lbrace
\begin{array}{l}
(v^2)'= K(\varepsilon^3+ \varepsilon^3 v+ \varepsilon v^2),\\
v(0)=0.
\end{array}\right. 
$$
The study of this problem gives us
\begin{align*}
&\frac{1}{K\varepsilon}\ln\left(1+v(t)+\frac{1}{\varepsilon^2}v(t)^2\right)\\
&-\frac{1}{K}\frac{2}{\sqrt{4-\varepsilon^2}}\left[\arctan\left(\frac{2}{\sqrt{4\varepsilon^2-\varepsilon^4}}\left[v(t)+\frac{\varepsilon^2}{2}\right]\right)-\arctan\left(\frac{\varepsilon}{\sqrt{4-\varepsilon^2}} \right)\right]=t.
\end{align*}
The boundness of the function $\arctan x$ implies
\begin{align*}
1+v(t)+\frac{1}{\varepsilon^2}v(t)^2\leq & e^{\frac{2\varepsilon}{\sqrt{4-\varepsilon^2}}} e^{\arctan\left[\frac{2}{\sqrt{4\varepsilon^2-\varepsilon^4}}\left(v(t)+\frac{\varepsilon^2}{2}\right)\right]-\arctan\left(\frac{\varepsilon}{\sqrt{4-\varepsilon^2}} \right) } e^{K\varepsilon t}\\
\leq & e^{\frac{2\varepsilon}{\sqrt{4-\varepsilon^2}}} e^{\frac{\pi}{2}} e^{K\varepsilon t}
\leq c_0^2 \;e^{K\varepsilon t}
\end{align*}
with $c_0^2=e^{\frac{2}{\sqrt{3}}} e^{\frac{\pi}{2}}$ which for instance is less than $3.5\;$.
Therefore, the estimate
$$\Vert \mathbf{U}_{\varepsilon}-\overline{\mathbf{U}}_{\varepsilon} \Vert_{L^2(\mathbb{R}^3)}\leq c_0 \varepsilon e^{K \varepsilon  t}$$
gives the result as soon as $c_0 \varepsilon e^{\varepsilon K t}\leq 2 \varepsilon^{b},$ with $b\leq 1$, \textit{i.e.} for
$t\leq \frac{C}{\varepsilon}\;\; \text{when}\;\; b=1,$
and  for
$t\leq \frac{C}{\varepsilon}\ln(\frac{1}{\varepsilon})$ in the case $b<1$. 

We finish the proof with the remark on the minimal regularity of the initial data for the Kuznetsov equation such that the approximation is possible, $i.e.$ the remainder terms $R_1^{NS-Kuz}$ and $\mathbf{R}_2^{NS-Kuz}$ keep bounded for a finite time interval. Indeed, if $u_0\in H^{s+2}(\mathbb{R}^n)$ and $u_1\in H^{s+1}(\mathbb{R}^n)$ with $s>\frac{n}{2}$ then $u\in C([0,+\infty[;H^{s+2}(\mathbb{R}^n))$ and
\begin{align*}
u_t\in  C([0,+\infty[;H^{s+1}(\mathbb{R}^n)),\quad
u_{tt}\in C([0,+\infty[;H^{s-1}(\mathbb{R}^n)).
\end{align*}
Since $\overline{\rho}_{\varepsilon}$ is defined by~(\ref{rhoKuz}) with~(\ref{rho1K}) and~(\ref{rho2K}) and $\overline{\textbf{v}}_{\varepsilon}$ by~(\ref{uKuz}) respectively,  we exactly find the regularity~(\ref{EqRegMinRhoNS}) and~(\ref{EqRegMinUNS}).  
Thus by the regularity of the left-hand side part for the approximated Navier-Stokes system~(\ref{NSAp1})--(\ref{NSAp2}) we obtain the desired regularity for the right-hand side.
\end{proof}

\subsection{Navier-Stokes system and the KZK equation.}\label{secNSKZK}

\subsubsection{Derivation of the KZK equation from an isentropic Navier-Stokes system.}\label{secderNSKZK}
In the present section we focus on the derivation from the isentropic Navier-Stokes system of 
 the Khoklov-Zabolotskaya-Kuznetsov equation (KZK)
     in non-linear media using the following acoustical properties of beam's propagation
\begin{enumerate}
  \item The beams are concentrated near the $x_1$-axis ;
  \item The beams propagate along the $x_1$-direction;
  \item The beams are generated either by an initial condition or by a forcing term on the boundary $x_1=0$.
  \end{enumerate}
The different type of derivations of the KZK equation are discussed in Ref.~\cite{Roz3}. 

This time we perform it in two steps:
\begin{enumerate}
\item We introduce small perturbations around a constant state of the compressible isentropic Navier-Stokes system according to the Kuznetsov \textit{ansatz} ~(\ref{rhoKuz})--(\ref{uKuz}):
\begin{align}
\partial_t \rho_{\varepsilon}+\nabla.(\rho_{\varepsilon} \mathbf{v}_{\varepsilon})=&\varepsilon[\partial_t\rho_1-\rho_0\Delta u]\nonumber\\
     &+\varepsilon^2[\partial_t \rho_2-\nabla \rho_1 \nabla u-\rho_1\Delta u]+O(\varepsilon^3),\label{approxNSKuzR1}
\end{align}
and we have again~(\ref{EqAppNS1p}) for the conservation of momentum.
\item We perform the paraxial change of variable\cite{Roz3} (see Fig.~\ref{fig2}):
\begin{equation}\label{chvarkzk}
\tau= t-\frac{x_1}{c},\;\;\;z=\varepsilon x_1,\;\;\; y=\sqrt{\varepsilon}x'.
\end{equation}
\end{enumerate} 
\begin{figure}[h!]
\begin{center}
\psfrag{a}{$x_1$}\psfrag{b}{$\mathbf{x'}$}\psfrag{c}{$t$}\psfrag{NS}{\small{Navier-Stokes/}}\psfrag{E}{\small{
Euler $(x_1,\mathbf{x'},t)$}}\psfrag{a1}{$z=\epsilon
x_1$}\psfrag{b1}{$y=\sqrt{\epsilon}
\mathbf{x'}$}\psfrag{c1}{$\tau=t-\frac{x_1}{c}$}\psfrag{KZK}{\small{KZK$(\tau,z,\mathbf{y})$}}
  \includegraphics[width=.7\textwidth]{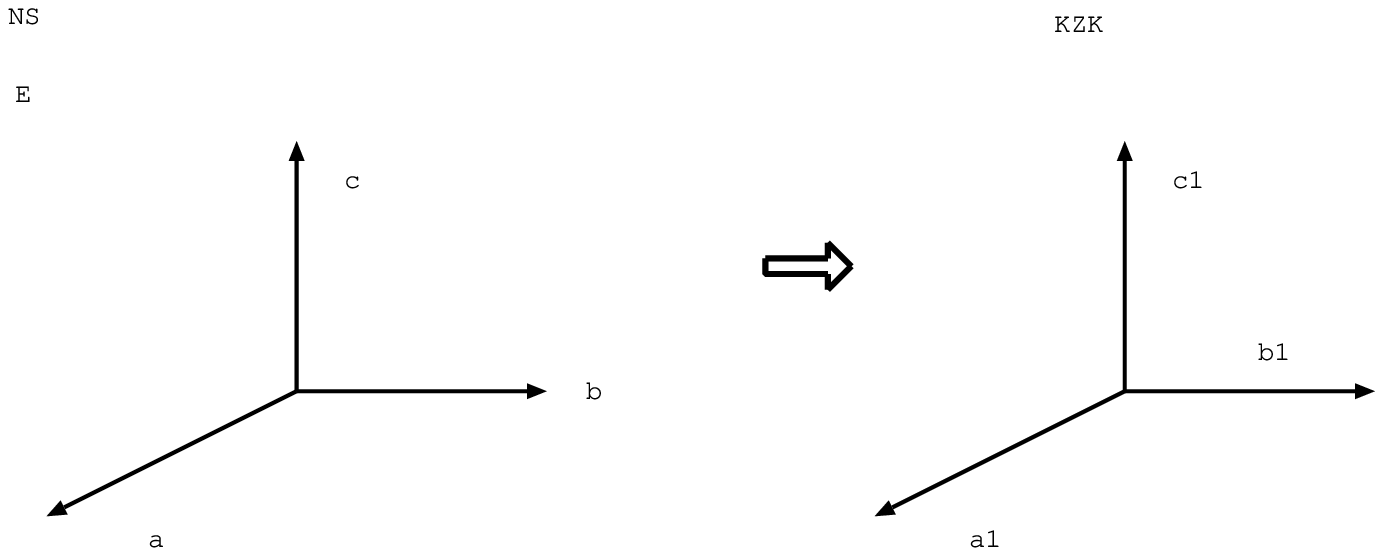}
                  \end{center}
   \caption{Paraxial change of variables for the profiles
$U(t-x_1/c,\epsilon x_1,\sqrt{\epsilon}\mathbf{x'})$.}\label{fig2}
\end{figure}
Since the gradient $\nabla$ in the coordinates $(\tau,z,y)$ becomes depending on $\eps$
$$\tilde{\nabla}=\left(\varepsilon \partial_z -\frac{1}{c}\partial_{\tau},\sqrt{\varepsilon} \nabla_y \right)^t,$$
if we denote 
\begin{equation}\label{paraxpot}
u(x,t)=\Phi(t-x_1/c,\epsilon x_1,\sqrt{\epsilon}x')=\Phi(\tau,z,y),
\end{equation}
we need to take attention to have the paraxial correctors of the order $O(1)$:
$$\rho_1(x,t)=I(\tau,z,y),\hbox{  }\rho_2(x,t)=H(\tau,z,y)=J(\tau,z,y)+O(\eps) ,$$
where actually $H(\tau,z,y)$ is the profile function obtained from $\rho_2$ (see~\ref{Apen1} Eq.~(\ref{rho2kzkH})) containing not only the terms of the order $O(1)$ but also terms up to $\eps^2$. Hence,  we denote by $J$ all terms of $H$ of order $0$ on $\eps$ which are significant in order to have an approximation up to the terms $O(\varepsilon^3)$. 

In new variables $(\tau,z,y)$ Eq.~(\ref{EqAppNS1p}) becomes 
\begin{align}
&\rho_{\varepsilon} [\partial_t \mathbf{v}_{\varepsilon}+(\mathbf{v}_{\varepsilon}.\nabla)\mathbf{v}_{\varepsilon}]+\nabla p(\rho_{\varepsilon})-\varepsilon \nu \Delta \mathbf{v}_{\varepsilon}= \varepsilon \tilde{\nabla}[-\rho_0 \partial_{\tau}\Phi+c^2 I]\label{EqAppNSkzk2}\\
&+ \varepsilon^2 \left[-I \tilde{\nabla}(\partial_{\tau}\Phi)+\frac{\rho_0}{2}\tilde{\nabla}\left(\frac{1}{c^2}(\partial_{\tau}\Phi)^2\right)\right.\nonumber\\
&\;\;\;\;\;\;\;\;\;\left.+c^2\tilde{\nabla} J+\frac{\gamma-1}{2\rho_0}c^2\tilde{\nabla}(I^2)
+\nu \tilde{\nabla}\left( \frac{1}{c^2}\partial^2_{\tau}\Phi\right) \right]+O(\varepsilon^3).\nonumber
\end{align}
Consequently, we find the correctors of the density as functions of $\Phi$:
\begin{align}
I(\tau,z,y) =&\frac{\rho_0}{c^2}\partial_\tau \Phi(\tau,z,\mathbf{y})\label{Ikzk},\\
J (\tau,z,y)=&-\frac{\rho_0 (\gamma-1)}{2c^4}(\partial_\tau \Phi)^2-\frac{\nu}{c^4} \partial^2_\tau \Phi. \label{Jkzk}
\end{align}
Indeed, we start by making $\varepsilon \tilde{\nabla}[-\rho_0 \partial_{\tau}\Phi+c^2 I]=0$ and find the first order perturbation of the density $I$ given by Eq.~(\ref{Ikzk}). Moreover, if $\rho_1$ satisfies~(\ref{Ikzk}), then Eq.~(\ref{EqAppNSkzk2}) becomes
\begin{multline}
\rho_\varepsilon[\partial_t
\mathbf{v}_\varepsilon+(\mathbf{v}_\varepsilon \cdot\nabla) \,
\mathbf{v}_\varepsilon] +\nabla p(\rho_\varepsilon)
-\varepsilon \nu \Delta \mathbf{v}_\varepsilon= \varepsilon \tilde{\nabla}[-\rho_0 \partial_{\tau}\Phi+c^2 I]\\
 \varepsilon^2 \tilde{\nabla}\left[-\frac{\rho_0}{2c^2} (\partial_{\tau}\Phi)^2+\frac{\rho_0}{2c^2} (\partial_{\tau}\Phi)^2
+c^2 J+\frac{(\gamma-1)\rho_0}{2c^2 }(\partial_{\tau} \Phi)^2+ \frac{\nu}{c^2}\partial^2_{\tau}\Phi\right]
+O(\varepsilon^3). 
\end{multline}
Thus, taking the corrector $J$ in the expansion of $\rho_\eps$
\begin{equation}
 \rho_{\varepsilon}(\mathbf{x},t)=\rho_0+\varepsilon I(\tau,z,\mathbf{y})+\varepsilon^2 J(\tau,z,\mathbf{y})\label{rhokzk},
\end{equation}
by formula~(\ref{Jkzk}), we ensure that 
\begin{equation}
 \rho_\varepsilon[\partial_t
\mathbf{v}_\varepsilon+(\mathbf{v}_\varepsilon \cdot\nabla) \,
\mathbf{v}_\varepsilon] +\nabla p(\rho_\varepsilon)
-\varepsilon \nu \Delta \mathbf{v}_\varepsilon= O(\varepsilon^3). 
\end{equation}
Now we put these expressions of $I$ from~(\ref{Ikzk}) and $J$ from~(\ref{Jkzk}) with  the paraxial approximation in  Eq.~(\ref{approxNSKuzR1}) of the mass conservation to obtain
\begin{align}
\partial_t \rho_{\varepsilon}+\nabla.(\rho_{\varepsilon} \mathbf{v}_{\varepsilon})=&\varepsilon^2\left[\frac{\rho_0}{c^2}(2c\partial^2_{z\tau}\Phi-c^2 \Delta_y\Phi)-\frac{\rho_0}{2c^4}(\gamma+1)\partial_{\tau}[(\partial_\tau \Phi)^2]-\frac{\nu}{c^4}\partial^3_\tau \Phi\right]\nonumber\\
&+O(\varepsilon^3).\label{massKZKsansreste}
\end{align}
All terms of the second order on $\eps$ in relation~(\ref{massKZKsansreste}) give us the equation on $\Phi$, which is the KZK equation.
If we use relation~(\ref{Ikzk}), we obtain the usual form of the KZK equation often written\cite{KhZabKuzn,Roz3} for the first perturbation $I$ of the density $\rho_\epsilon$:
\begin{equation}\label{KZKI}
c\partial^2_{\tau z} I -\frac{(\gamma+1)}{4\rho_0}\partial_\tau^2
I^2-\frac{\nu}{2 c^2\rho_0}\partial^3_\tau I-\frac{c^2}2 \Delta_y
I=0.
\end{equation}
We notice that, as the Kuznetsov equation, this model still contains terms describing the wave propagation $\partial^2_{\tau z} I$,  the non-linearity  $ \partial_\tau^2
I^2$ and the viscosity effects $\partial^3_\tau I$ of the medium but also adds a diffraction effects by the traversal Laplacian $ \Delta_yI$. This corresponds to the description of the quasi-one-dimensional propagation of a signal in a homogeneous nonlinear isentropic medium.
By our derivation (see also~(\ref{massKZK})--(\ref{momentKZK}))  we obtain that the KZK equation is the second order approximation of the isentropic Navier-Stokes system up to term of $O(\varepsilon^3)$. 
In this sense, since the entropy and the pressure in Section~\ref{secderisNS} are approximated up to terms of the order of $\varepsilon^3$, the \emph{ansatz}~(\ref{rhokzk})-(\ref{vkzk}) (for the KZK equations) is optimal, as the equations of the Navier-Stokes system are approximated up to $O(\varepsilon^3)$-terms. 

Let us compare our \emph{ansatz}
\begin{align}
 &u(x_1,\mathbf{x'},t)=\Phi(t-x_1/c,\epsilon x_1,\sqrt{\epsilon}x'),\label{EqAnsKZK1}\\
 &\rho_{\varepsilon}(x_1,\mathbf{x'},t)=\rho_0+\varepsilon I(t-x_1/c,\epsilon x_1,\sqrt{\epsilon}x')+\varepsilon^2 J(t-x_1/c,\epsilon x_1,\sqrt{\epsilon}x')\label{EqAnsKZK2}
\end{align}
to the \emph{ansatz} introduced in Ref.~\cite{Roz3} by definning a corrector $\epsilon^2 v_2$ for the velocity perturbation along the propagation axis in the initial \emph{ansatz}, proposed by Khokhlov and Zabolotskaya~\cite{KhZabKuzn}:
 \begin{align}
   &  \rho_\epsilon(x_1,\mathbf{x'},t)= \rho_0+\epsilon I(t-\frac{x_1}{c},\epsilon x_1, \sqrt{\epsilon} \mathbf{x'})\,,\label{AnsaKZK1}\\
   &\mathbf{v_\epsilon}(x_1,\mathbf{x'},t)=\epsilon(v_1+\epsilon v_2 ;\sqrt{\epsilon}
   \mathbf{w})(t-\frac{x_1}{c},\epsilon x_1, \sqrt{\epsilon} \mathbf{x'}).\label{AnsaKZK2}
  \end{align}

This time,
the assumption to  work directly with the velocity potential~(\ref{EqAnsKZK1}) imediately implies  the following velocity expansion
\begin{align}
\mathbf{v}_{\varepsilon}(\mathbf{x},t)=&-\varepsilon\left(-\frac{1}{c}\partial_\tau\Phi+\varepsilon \partial_z \Phi; \sqrt{\varepsilon}\nabla_y \Phi\right)(\tau,z,\mathbf{y}),\label{vkzk}
\end{align}
where we recognize the velocity \emph{ansatz} of Ref.~\cite{Roz3} %
   with $$v_1=\frac{1}{c}\partial_\tau\Phi= \frac{c}{\rho_0} I, \quad   \mathbf{w}=\nabla_y \Phi=\frac{c^2}{\rho_0}\partial^{-1}_{\tau}\nabla_y
I,$$
but for the corrector $v_2$ this time
$$v_2=-\partial_z \Phi=-\frac{c^2}{\rho_0}\partial^{-1}_{\tau}\partial_zI$$ 
instead of (see Ref.~\cite{Roz3} and formula~(\ref{invdtau}) for  definition of the operator $\del_{\tau}^{-1}$)
$$v_2^{Rozanova}=-\frac{c^2}{\rho_0}\partial^{-1}_{\tau}\partial_zI
+\frac{(\gamma-1)}{2\rho_0^2}cI^2+\frac{\nu}{c\rho_0^2}\del_\tau I.$$
If we add the second order correctors $v_2$ for the velocity to $J$ for the density, we obtain exactly all terms of the corrector $v_2^{Rozanova}$. But the \emph{ansatz}~(\ref{AnsaKZK1})--(\ref{AnsaKZK2}) is not optimal since the equation of momentum in transverse direction keeps the non-zero terms~\cite{Roz3} of the order of $\epsilon^\frac{5}{2}$.

\subsubsection{Well posedness of the KZK equation.}\label{secWPKZKNPE}
We use Ref.~\cite{Roz2} to give results on well posedness of the Cauchy problem:
\begin{equation}\label{NPEcau}
\left\lbrace
\begin{array}{c}
c\partial^2_{\tau z} I -\frac{(\gamma+1)}{4\rho_0}\partial_\tau^2
I^2-\frac{\nu}{2 c^2\rho_0}\partial^3_\tau I-\frac{c^2}2 \Delta_y
I=0\hbox{ on }\mathbb{T}_{\tau}\times\mathbb{R}_+\times\mathbb{R}^{n-1},\\
I(\tau,0,y)=I_0(\tau,y)\hbox{ on }\mathbb{T}_{\tau}\times\mathbb{R}^{n-1}
\end{array}\right.
\end{equation}
in the class of $L-$periodic functions with respect to the variable $\tau$ and with mean value zero
\begin{equation}\label{zeromeanval}
\int_0^L I(\ell,z,y)d\ell=0 .
\end{equation}
The introduction of the operator $\partial_\tau^{-1}$, defined by formula\begin{equation}\label{invdtau}
\partial_{\tau}^{-1} I(\tau,z,y):=\int_0^{\tau} I(\ell,z,y) d\ell+\int_0^L \frac{\ell}{L}I(\ell,z,y)d\ell,
\end{equation}
allows us to consider instead of Eq.~(\ref{KZKI}) the following equivalent equation
$$c\partial_{ z} I -\frac{(\gamma+1)}{4\rho_0}\partial_\tau
I^2-\frac{\nu}{2 c^2\rho_0}\partial^2_\tau I-\frac{c^2}2 \partial_{\tau}^{-1}\Delta_y
I=0\hbox{ on }\mathbb{T}_{\tau}\times\mathbb{R}_+\times\mathbb{R}^{n-1} ,$$
for which it holds the following theorem\cite{Ito,Roz2}:
\begin{theorem}\label{wpglopkzknpe}
\cite{Roz2}
Consider the Cauchy problem for the KZK equation:
\begin{equation}\label{NPEcau2}
\left\lbrace
\begin{array}{c}
c\partial_{ z} I -\frac{(\gamma+1)}{4\rho_0}\partial_\tau
I^2-\frac{\nu}{2 c^2\rho_0}\partial^2_\tau I-\frac{c^2}2 \partial_{\tau}^{-1}\Delta_y
I=0\hbox{ on }\mathbb{T}_{\tau}\times\mathbb{R}_+\times\mathbb{R}^{n-1},\\
I(\tau,0,y)=I_0(\tau,y)\hbox{ on }\mathbb{T}_{\tau}\times\mathbb{R}^{n-1},
\end{array}\right.
\end{equation}
with the operator $\partial_\tau^{-1}$ defined by formula~(\ref{invdtau}), $\nu\geq 0$, and $\int_0^L I_0(\ell,y)d\ell=0 $, the following results hold true
\begin{enumerate}
\item (Local existence)  For $s>\left[\frac{n}{2}\right]+1$ there exists a constant $C(s,L)$ such that for any initial data $I_0\in H^s(\mathbb{T}_\tau\times \mathbb{R}^{n-1})$ on an interval $[0,T[$ with
$$T\geq \frac{1}{C(s,L)\Vert I_0\Vert_{H^s(\mathbb{T}_\tau\times \mathbb{R}^{n-1})}} $$
problem~(\ref{NPEcau2}) has a unique solution $I$ such that  
$$I\in C([0,T[,H^s(\mathbb{T}_\tau\times \mathbb{R}^{n-1}))\cap C^1([0,T[,H^{s-2}(\mathbb{T}_\tau\times \mathbb{R}^{n-1})),$$ which satisfies the zero mean value condition~(\ref{zeromeanval}).
\item  (Shock formation) Let $T^*$ be the largest time on which such a solution is defined, then we have
$$\int_0^{T^*} \sup_{\tau,y}(\vert \partial_\tau I(\tau,t,y)\vert+\vert \nabla_y I(\tau,t,y)\vert)\;dt=+\infty.$$
\item(Global existence) If $\nu>0$ we have the global existence for small enough data: there exists a constant $C_1>0$ such that
$$\Vert I_0\Vert_{H^s(\mathbb{T}_\tau\times \mathbb{R}^{n-1})}\leq C_1 \Rightarrow T^*=+\infty.$$
\item (exponential decay)\cite{Ito}  If $\nu>0$, $s\in \mathbb{N}$ and $s\geq \left[\frac{n+1}{2}\right]$, then  there exists a constant $C_2>0$ such that $\Vert I_0\Vert_{H^s(\mathbb{T}_\tau\times \mathbb{R}^{n-1})}\leq C_2$ implies for all $z\geq0$
$$\Vert I(z)\Vert_{H^s(\mathbb{T}_\tau\times \mathbb{R}^{n-1})}\leq C \Vert I_0\Vert_{H^s(\mathbb{T}_\tau\times \mathbb{R}^{n-1})} e^{-\ell z}, $$where $C>0$ and $\ell\in]0,1[$ are constants.
\end{enumerate}
\end{theorem}
\begin{remark}
\cite{Roz2} We note that when $\nu=0$, all the corresponding statements of Theorem~\ref{wpglopkzknpe} remain valid for $0>t>-C$ with a suitable $C$.
\end{remark}
\begin{remark}\label{RemHalfSpace}
In the study of the well-posedness of the KZK equation we inverse the usual role of the time with the main space variable along the propagation axis $z$:
 for $\nu>0$ the solution $I(\tau,z,y)=I(t-\frac{x_1}{c},\eps x_1, \sqrt{\eps}x')$ is defined
for $x_1>0$, as it is global on $z\in \R^+$. Hence if we want to compare the KZK equation to other models such as the Kuznetsov equation or the Navier-Stokes system we need the well posedness results for these models on the half space
\begin{equation}\label{hs}
 \{x_1>0, \quad t>0, \quad x'\in
\R^{n-1}\},
\end{equation}
taking into account the fact that  the boundary conditions for the exact system come from the initial condition $I_0$ of the Cauchy problem~(\ref{NPEcau2}) associated to the KZK equation.
\end{remark} 

\subsubsection{Entropy estimate for the isentropic Navier-Stokes equation on the half space and the associated existence result.}\label{sNSe}

We follow now Section 5.2 in~Ref.~\cite{Roz3} updating it for the new \textit{ansatz} and correct the proof of Theorem~5.5. See~Ref.~\cite{Roz3} for more details. 

We consider 
the Cauchy problem for the KZK equation~(\ref{NPEcau2})
for an initial data $$I(t,0,y)=I_0(t,y) \quad (\tau=t \; \hbox{ for }x_1=0)$$ $L$-periodic in $t$ with
mean value zero.
Theorem~\ref{wpglopkzknpe} ensures that for any
initial data $I_0$, defined in
$\mathbb{T}_t\times \mathbb{R}^{n-1}$ with small enough $H^s$ ($s>
[\frac n2]+1$) norm (with respect to $\nu$), there exists a unique
solution of the KZK equation~(\ref{KZKI}) $I$, which
as a function of $(\tau , z, y)$ is global on $z\in \R^+$, periodic in $\tau$ of period
$L$ and mean value zero, and decays for
$z\rightarrow \infty$~\cite{Roz2}.

Therefore, see Remark~\ref{RemHalfSpace}, we consider our approximation problem  between the isentropic Navier-Stokes system~(\ref{NSi1})--(\ref{NSi2}) and the KZK equation in the half space~(\ref{hs}).

By $I_0$ we find $I$ and thus also $\Phi$ and $J$, using Eq.~(\ref{Ikzk})--(\ref{Jkzk}). This allows us to construct the density and velocities $\overline{\rho}_{\varepsilon}$ and $\overline{\textbf{v}}_{\varepsilon}$ in accordance with the \textit{ansatz} ~(\ref{rhokzk}) and~(\ref{vkzk}). Thus, by $I$ we  construct the function $\overline{\mathbf{U}}_{\varepsilon}=(\overline{\rho}_{\varepsilon},\overline{\rho}_{\varepsilon}\overline{\textbf{v}}_{\varepsilon})^t$.

In particular, for $t=0$ we have functions defined for $x_1>0$
because $I$ is well-defined for any $z>0$
\begin{eqnarray*}
&&\overline{\rho}_\eps(0, x_1,x')=\rho_0+\eps I(-\frac{x_1}c,\eps
x_1, \sqrt{\eps} x')+\varepsilon^2 J(-\frac{x_1}c,\eps
x_1, \sqrt{\eps} x'),
\\
&&\overline{\textbf{v}}_\eps(0,
x_1,x')=(\overline{v}_{1},\overline{\textbf{v}}'_\eps)(-\frac{x_1}c,\eps
x_1, \sqrt{\eps} x'),
\end{eqnarray*}
where
\begin{align*}
\overline{v}_{1}=\eps\frac{
c}{\rho_0}I+ \eps^2\frac{c^2}{\rho_0} \partial_z\partial_\tau^{-1}I,\quad\overline{\textbf{v}}'_\eps=\sqrt{\varepsilon}\frac{c^2}{\rho_0}\nabla_y\partial_\tau^{-1}I
\end{align*}
and for $x_1=0$ we have $L$-periodic functions with mean value zero
\begin{eqnarray}
&&\overline{\rho}_\eps(t,0,x')=\rho_0+\eps I(t,0, \sqrt{\eps}
x')+\varepsilon^2 J(t,0, \sqrt{\eps}
x'),\label{bns2}
\\
&&\overline{\textbf{v}}_\eps(t,0,x')=(\overline{v}_{1},\overline{\textbf{v}}'_\eps)(t,0,
\sqrt{\eps} x').\label{bns1}
\end{eqnarray}

It is important to notice that the solution $\overline{\textbf{v}}_\eps $ in
system~(\ref{NSi1})--(\ref{NSi2}) is small on the boundary:
$\overline{\textbf{v}}_\eps|_{x_1=0}=\eps \tilde{\textbf{v}}_\eps|_{x_1=0}$. Therefore, we have
$\left|\overline{\textbf{v}}_\eps|_{x_1=0}\right|<c$, which corresponds to
 the ``subsonic'' boundary case. More precisely, when the first
velocity component is positive $\overline{v}_{1}|_{x_1=0}>0$,
  we have a subsonic inflow boundary condition, and when it is
negative  $\overline{v}_{1}|_{x_1=0}<0$, we have a subsonic outflow boundary condition, see Fig.~\ref{figNShs}.
\begin{figure}[h!]
\begin{center}
\psfrag{0}{$0$}\psfrag{x}{$x_1>0$}\psfrag{y}{$x'$}\psfrag{B}{$v_1|_{x_1=0}<0$}
\psfrag{A}{$v_1|_{x_1=0}>0$}\psfrag{up}{$(\mathbf{v}-\overline{\mathbf{v}})|_{x_1=0}=0$}\psfrag{un}{$(\mathbf{v}-\overline{\mathbf{v}})|_{x_1=0}=0$}\psfrag{rp}{$(\rho-\overline{\rho})|_{x_1=0}=0$}\psfrag{t}{$t$}\includegraphics[width=6cm]{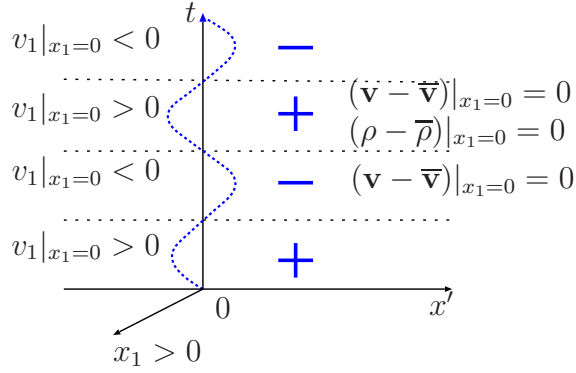}
\caption{Periodic subsonic inflow-outflow boundary conditions for the Navier-Stokes system.}\label{figNShs}
\end{center}
\end{figure}
We also notice that, due to Eq.~(\ref{vkzk}), the first component of the
velocity $\overline{\textbf{v}}_{1}$ on the boundary has the following form
\begin{eqnarray*}&&\overline{v}_{1}|_{x_1=0}=\left(\eps\frac{
c}{\rho_0}I+ \eps^2 G(I)\right)
(t,0,\sqrt{\eps}x')=\left.\left(\eps\frac{ c}{\rho_0}I+\eps^2
G(I)\right)\right|_{z=0}\\
&&=\eps\frac{ c}{\rho_0}I_0(t,y)+\eps^2 G(I_0)(t,y),
\end{eqnarray*}
where
\begin{equation}G(I)=\frac{c^2}{\rho_0} \partial_z\partial_\tau^{-1}I=\frac{c^2}{\rho_0}\partial_\tau^{-1}\left(\frac{(\gamma+1)}{4c\rho_0}\partial_\tau
I^2+\frac{\nu}{2 c^3\rho_0}\partial^2_\tau I +\frac{c}{2} \partial_{\tau}^{-1}\Delta_y I\right)
.\label{ch3F}
\end{equation}
 Therefore, the boundary
conditions for $\overline{\textbf{v}}_{1}$ are defined by the initial conditions for
KZK equation and
 are $L$-periodic in $t$ and have  mean value zero. In addition, the
sign of
 $\overline{\textbf{v}}_{1}|_{x_1=0}$ is the same as the sign of $I_0$
 (because the term $G(I_0)$ is of a higher order of smallness on $\eps$).\begin{remark}\label{remNST}
As  the viscosity term   $\eps \nu \overline{\textbf{v}}_{\varepsilon}$, where $\eps$ is
fixed small enough parameter, $\nu$ is a constant, and in our case
$\overline{\textbf{v}}_{\varepsilon}$ is of the order of $\eps$, the boundary layer phenomenon can be
excluded.
\end{remark}

\begin{theorem}\label{thENS}
Let $ n\leq 3$. Suppose that  the initial data of the KZK Cauchy problem
$I_0(t,y)=I_0(t,\sqrt{\epsilon} x')$ is such that \begin{enumerate}
    \item $I_0$ is $L$-periodic in $t$ and with  mean value zero,
\label{itT1}
    \item for fixed $t$, $I_0$ has the same sign for all $y\in
    \R^{n-1}$, and for $t\in ]0,L[$ the sign changes, i.e. $I_0=0$,
only for a finite number of times, \label{itT2}
    \item $I_0(t,y)\in H^{s}(\mathbb{T}_t\times \R^{n-1})$ for
    $s\ge 10$,\label{itT3}
    \item $I_0$ is sufficiently small in the sense of
Theorem~\ref{wpglopkzknpe} such that\cite[p.20]{Roz2} 
\begin{equation*}
 \|I_0\|_{H^{s}}< \frac{\nu}{2c^2\rho_0}\frac{C_1(L)}{C_2(s)}. 
\end{equation*}
       \label{itT4}
\end{enumerate}
Consequently, there exists  a unique global  solution in time $I(\tau,z,y)$
of~(\ref{NPEcau2}) for $z=\epsilon x_1 >0$, moreover, the functions
$\bar{\rho}_\epsilon$,
$\overline{\textbf{v}}_{\varepsilon}=(\overline{v}_1,\overline{\textbf{v}}'_\eps)$, defined
by the \textit{ansatz} ~(\ref{rhokzk})-(\ref{vkzk}) and Eq.~(\ref{Ikzk})--(\ref{Jkzk}) in the half space~(\ref{hs})
 are
smooth with $\Omega=\mathbb{T}_t\times
\R^{n-1}_y$:
\begin{equation}
\bar{\rho}_\epsilon\in C\left([0,\infty[,
H^{s-4}\left(\Omega\right)\right) \cap
C^1\left([0,\infty[;H^{s-6}\left(\Omega\right)\right),\label{barro}
\end{equation}
\begin{equation}\bar{\textbf{v}}_\epsilon\in C\left([0,\infty[;
H^{s-4}\left(\Omega\right)\right) \cap
C^1\left([0,\infty[;H^{s-6}\left(\Omega\right)\right).\label{baru}
\end{equation}
The Navier-Stokes system~(\ref{NSi1})--(\ref{NSi2}) in  the half space
 with
 initial data~(\ref{bban}) and following boundary conditions
$$    (\bar{\textbf{v}}_{\epsilon}-\mathbf{v}_{\epsilon})|_{x_1=0}=0,$$%
with positive  first component of the velocity
$v_1|_{x_1=0}
>0$ (i.e. at points where the fluid enters the domain) has the
additional boundary
condition
$$(\bar{\rho}_\epsilon-\rho_\epsilon)|_{x_1=0}=0.
$$
When $v_1|_{x_1=0} \le0$ there is no any boundary
condition for $\rho_\epsilon$.

Then, for all finite times $T>0$
    there exists a unique
    solution $U_\epsilon=(\rho_\epsilon,\rho_\varepsilon u_\epsilon)$ of the
Navier-Stokes system~(\ref{NSi1})--(\ref{NSi2}) with
    the following smoothness on $\left[0,T\right]$
     $$\rho_{\eps}\in
C\left(\left[0,T\right],H^{3}\left(\{x_1>0\}\times\R^{n-1}\right)\right)\cap
C^1\left(\left[0,T\right],H^{2}\left(\{x_1>0\}\times\R^{n-1}\right)\right)$$ and
$$u_{\eps}\in
C\left(\left[0,T\right],H^{3}\left(\{x_1>0\}\times\R^{n-1}\right)\right)\cap
C^1\left(\left[0,T\right],H^{1}\left(\{x_1>0\}\times\R^{n-1}\right)\right). $$%
\end{theorem}

\begin{remark}\label{NSREM1}\cite{Roz3}
The restriction to have the same sign for $I_0$ for all fixed time avoids a change
in the type of the boundary condition applied to the tangential
variables for the Navier-Stokes system. Moreover, Zabolotskaya~\cite{KhZabKuzn}  takes as the initial
conditions for the
KZK equation (which correspond to the boundary condition for $v_1$)
the expression $$I(\tau,0,y)=-F(y) \sin \tau$$  with an amplitude distribution $F(y)\geq 0$. Especially, for a Gaussian
beam~\cite{KhZabKuzn}
$$F(y)=e^{-y^2},$$
    while for a beam with a polynomial amplitude~\cite{KhZabKuzn}  $$F(y)=\left\{%
\begin{array}{ll}
    (1-y^2)^2, \quad y\leq 1, \\
    0, \quad y>1.
\end{array}%
\right.$$
\end{remark}

\begin{proof}
As previously, we use  the fact that the   entropy for the isentropic
Euler system  $\eta(\mathbf{U}_{\eps})$, defined by
Eq.~(\ref{entropy}) is a convex function\cite{Dafermos}.

%

Let us multiply the Navier-Stokes system~(\ref{NSmatr}), from the left, by
$2\mathbf{U}_\eps^T \eta''(\mathbf{U}_\eps)$
$$2\mathbf{U}_\eps^T
\eta''(\mathbf{U}_\eps)\del_t \mathbf{U}_{\eps}
 +\sum_{i=1}^n 2\mathbf{U}_\eps^T
\eta''(\mathbf{U}_\eps) DG_i(\mathbf{U}_\varepsilon)\partial_{x_i} \mathbf{U}_\varepsilon-
 \eps\nu 2\mathbf{U}_\eps^T
\eta''(\mathbf{U}_\eps)\left[\begin{array}{c}0\\ \triangle
\mathbf{v}_{\eps}\end{array}\right]=0.$$

We notice that
$$\mathbf{U}_\eps^T
\eta''(\mathbf{U}_\eps)\left[\begin{array}{c}0\\ \triangle
\mathbf{v}_{\eps}\end{array}\right]=0,$$
and, therefore, we have
$$2 \mathbf{U}_\eps^T \eta''(\mathbf{U}_\eps) \del_t
\mathbf{U}_\eps
 =\del_t[\mathbf{U}_\eps^T
\eta''(\mathbf{U}_\eps)\mathbf{U}_\eps]-\mathbf{U}_\eps^T \del_t
\eta''(\mathbf{U}_\eps) \mathbf{U}_\eps.$$
Moreover, by virtue of $\eta''(U)DG_i(U)=(DG_i(U))^T \eta''(U)$ we find$$2\mathbf{U}_\eps^T
\eta''(\mathbf{U}_\eps) DG_i(\mathbf{U}_{\eps})\partial_{x_i}
\mathbf{U}_\eps=\partial_{x_i} [\mathbf{U}_{\eps}^T
\eta''(\mathbf{U}_{\eps})DG_i(\mathbf{U}_{\eps})\mathbf{U}_{\eps}]
-  \mathbf{U}_{\eps}^T
\partial_{x_i}[\eta''(\mathbf{U}_{\eps})DG_i(\mathbf{U}_{\eps})]\mathbf{U}_{\eps}.$$
Integrating over $[0,t]\times\{x_1>0\}$ ($x'\in \R^{n-1}$), we obtain
$$\int_0^t\int_{x_1>0}\del_t[\mathbf{U}_\eps^T
\eta''(\mathbf{U}_\eps)\mathbf{U}_\eps]dxds+\int_0^t\int_{x_1>0}\sum_{i=1}^n\partial_{x_i}
[\mathbf{U}_{\eps}^T
\eta''(\mathbf{U}_{\eps})DG_i(\mathbf{U}_{\eps})\mathbf{U}_{\eps}]dxds$$
$$- \int_0^t\int_{x_1>0}\mathbf{U}_\eps^T \del_t
\eta''(\mathbf{U}_\eps) \mathbf{U}_\eps
dxds-\int_0^t\int_{x_1>0}\sum_{i=1}^n\mathbf{U}_{\eps}^T
\partial_{x_i}[\eta''(\mathbf{U}_{\eps})DG_i(\mathbf{U}_{\eps})]\mathbf{U}_{\eps}dxds=0.$$

 Integrating by parts we result in
\begin{eqnarray*}
&&\int_{x_1>0}\mathbf{U}_{\eps}^T\eta''(\mathbf{U}_{\eps})\mathbf{U}_{\eps}dx-
\int_{x_1>0}\mathbf{U}_{\eps}^T\eta''(\mathbf{U}_{\eps})\mathbf{U}_{\eps}|_{t=0}dx\\
&&- \int_0^t\int_{x_1>0}\mathbf{U}_\eps^T\left[ \del_t
\eta''(\mathbf{U}_\eps) +
\sum_{i=1}^n \partial_{x_i}[\eta''(\mathbf{U}_{\eps})DG_i(\mathbf{U}_{\eps})]\right]\mathbf{U}_{\eps}dxds\\
&& -\int_0^t\int_{\R^{n-1}}\mathbf{U}_{\eps}^T
\eta''(\mathbf{U}_{\eps})DG_1(\mathbf{U}_{\eps})\mathbf{U}_{\eps}|_{x_1=0}dx'ds=0.
\end{eqnarray*}

 We recall that
$\eta''(\mathbf{U}_{\eps})$ is positive definite, consequently for some $C>0$ and $\delta_0>0$
$$C |\mathbf{U}_{\eps}|^2 \geq\mathbf{U}_{\eps}^T
\eta''(\mathbf{U}_{\eps})\mathbf{U}_{\eps}\ge \delta_0
|\mathbf{U}_{\eps}|^2.$$

Therefore, we obtain for the initial data
\begin{equation}\label{defU0NSKZK}
\textbf{U}_{0}=\left[\begin{array}{c}\rho_0+\eps I+\varepsilon^2 J\\
\eps\left(\rho_0+\eps I+\varepsilon^2 J\right)\left(\frac{c}{\rho_0}I+\eps
G(I),\sqrt{\eps}\vec{w}\right)\end{array}\right]\left(-\frac{x_1}{c},\eps x_1, \sqrt{\eps}x'\right)
\end{equation} 
and the relation
 \begin{align*}
 \delta_0\int_{x_1>0}\mathbf{U}_{\eps}^2 dx -C \int_{x_1>0}\mathbf{U}_{0}^2dx- & \int_0^t\int_{\R^{n-1}}\mathbf{U}_{\eps}^T
\eta''(\mathbf{U}_{\eps})DG_1(\mathbf{U}_{\eps})\mathbf{U}_{\eps}|_{x_1=0}dx'ds \\
\leq & C_1 \int_0^t\int_{x_1>0} \mathbf{U}_{\eps}^2 dx \;ds.
 \end{align*}
As in Ref.~\cite{Gustafsson}, $C_1$ is an upper bound for the eigenvalues of the symmetric matrix
 $$ \del_t
\eta''(\mathbf{U}_\eps) +
\sum_{i=1}^n \partial_{x_i}[\eta''(\mathbf{U}_{\eps})DG_i(\mathbf{U}_{\eps})].$$
Let us now consider the integral on the boundary. With notation
$\mathbf{v}_{\varepsilon}=(v_1,\mathbf{v}_{\varepsilon}')^t$ for the velocity and
$H''(\rho)=\frac{p'(\rho)}{\rho}$, we see with $DG_1(\mathbf{U}_{\varepsilon})$ coming from~(\ref{NSmatrfact}) that
\begin{eqnarray*}
&&\mathbf{U}_{\eps}^T
\eta''(\mathbf{U}_{\eps})DG_1(\mathbf{U}_{\eps})\mathbf{U}_{\eps}\\
&&=(\rho_\eps, \rho_\eps
\mathbf{v}_{\varepsilon})^T\left(\begin{array}{cc}H''(\rho_\eps)+\frac{\mathbf{v}_\eps^2}{\rho_\eps}&-\frac{\mathbf{v}_{\varepsilon}}{\rho_\eps}\\& \\
-\frac{\mathbf{v}_{\varepsilon}}{\rho_\eps}&\frac{1}{\rho_\eps}Id_{n} \end{array}\right)
\left(\begin{array}{ccc}
0&1&0\\
-v_1^2+p'(\rho_\eps)&2v_1&0\\
-v_1 \mathbf{v}_{\varepsilon}' & \mathbf{v}_{\varepsilon}'& v_1 Id_{n-1}
\end{array}\right)
\left(\begin{array}{c}\rho_\eps\\
\rho_\eps \mathbf{v}_{\varepsilon}\end{array}\right)\\
&&=(\rho_\eps,\rho_\eps v_1 ,\rho_\eps
\textbf{v}'_{\varepsilon})^T \left(\begin{array}{ccc}
v_1\left(\frac{\mathbf{v}_{\varepsilon}^2}{\rho_{\varepsilon}}-\frac{p'(\rho_{\varepsilon})}{\rho_{\varepsilon}}\right) & \frac{-v_1^2}{\rho_{\varepsilon}}+\frac{p'(\rho_{\varepsilon})}{\rho_{\varepsilon}} & -v_1\frac{\mathbf{v}_{\varepsilon}'}{\rho_{\varepsilon}} \\
\frac{-v_1^2}{\rho_{\varepsilon}}+\frac{p'(\rho_{\varepsilon})}{\rho_{\varepsilon}} & \frac{v_1}{\rho_{\varepsilon}} & 0 \\
-v_1\frac{\mathbf{v}_{\varepsilon}'}{\rho_{\varepsilon}} & 0 & \frac{v_1}{\rho_{\varepsilon}}Id_{n-1}
\end{array}\right) \left(\begin{array}{c}\rho_\eps\\
\rho_\eps v_1\\
\rho_\eps \mathbf{v}_{\varepsilon}'
\end{array}\right)\\
&&=\rho_{\varepsilon}p'(\rho_{\varepsilon})v_1.
\end{eqnarray*}

Let us consider the initial condition $I_0(t,y)$ for the KZK
equation of the type described in Remark~\ref{NSREM1}. We suppose (without
loss of generality) that $I_0=0$ for $t\in ]0,L[$ only once. 
More precisely, we suppose that the sign of $v_1$ is changing in the following way:

 \begin{itemize}
    \item $v_1\le0$ for $t\in [0+(k-1)L, \frac{L}{2}+(k-1)L]$
($k=1,2,3,...$),
    \item $v_1>0$  for $t\in ] \frac{L}{2}+(k-1)L,kL[$ ($k=1,2,3,...$). \end{itemize}

If $t\in [0, \frac{L}{2}]$ (for $k=1$), the first component of the
velocity 
$v_1|_{x_1=0}<0$ is negative, and thus we have
$$\rho_{\varepsilon}p'(\rho_{\varepsilon})v_1<0.$$
 If $t\in ] \frac{L}{2},L[$, the first component of velocity is positive
$v_1|_{x_1=0}>0,$
 then we also impose $\rho_\eps|_{x_1=0}=\rho_0+\eps I_0(t,y)+\varepsilon^2 J$, where
$I_0(t,y)$ is the initial
condition for the KZK equation and $J$ coming from Eq.~(\ref{Jkzk}).
For the term $$\rho_{\varepsilon}p'(\rho_{\varepsilon})v_1>0$$ we see that on the boundary   it has
the form
\begin{align*}
\rho_{\varepsilon}p'(\rho_{\varepsilon})v_1=& \varepsilon \left( \frac{c}{\rho_0} I_0+\frac{c^2}{\rho_0}\partial_z\partial_{\tau}^{-1}I_0\right) (\rho_0+\eps I_0(t,y)+\varepsilon^2 J)p'(\rho_0+\eps I_0(t,y)+\varepsilon^2 J)\\
\leq & C_0 \varepsilon I_0
\end{align*}

 for some constant $C_0>0$ independent on $\varepsilon$. Consequently, for $k\geq 1$
\begin{align*}
\int_0^{kL} \int_{\mathbb{R}^{n-1}}\rho_{\varepsilon}p'(\rho_{\varepsilon})v_1 \vert_{x_1=0} dx'\; ds \leq & \sum_{j=1}^k \int_{\left]\frac{L}{2}+(j-1)L,jL\right[ }\int_{\mathbb{R}^{n-1}}\rho_{\varepsilon}p'(\rho_{\varepsilon})v_1 \vert_{x_1=0} dx'\; ds\\
\leq & \sum_{j=1}^k \int_{\left]\frac{L}{2}+(j-1)L,jL\right[ }\int_{\mathbb{R}^{n-1}}  C_0 \varepsilon I_0
\leq  K k \varepsilon \Vert I_0\Vert_{H^{s}},
\end{align*}
where $K=O(1)$ is a positive constant independent of $k$.
\\However for $t>0$ we have $k\geq 1$ such that $t\in [(k-1)L,kL[$ and it implies on one hand if $t\in \left[(k-1)L,(k-1)L+\frac{L}{2}\right[ $
$$ \int_0^{t} \int_{\mathbb{R}^{n-1}}\rho_{\varepsilon}p'(\rho_{\varepsilon})v_1 \vert_{x_1=0} dx'\; ds\leq \int_0^{(k-1)L} \int_{\mathbb{R}^{n-1}}\rho_{\varepsilon}p'(\rho_{\varepsilon})v_1 \vert_{x_1=0} dx'\; ds$$and on the other hand if $t\in\left[(k-1)L+\frac{L}{2},kL\right[$
$$ \int_0^{t} \int_{\mathbb{R}^{n-1}}\rho_{\varepsilon}p'(\rho_{\varepsilon})v_1 \vert_{x_1=0} dx'\; ds\leq \int_0^{kL} \int_{\mathbb{R}^{n-1}}\rho_{\varepsilon}p'(\rho_{\varepsilon})v_1 \vert_{x_1=0} dx'\; ds.$$
As a consequence, we obtain for all $t>0$
$$\int_0^{t} \int_{\mathbb{R}^{n-1}}\rho_{\varepsilon}p'(\rho_{\varepsilon})v_1 \vert_{x_1=0} dx'\; ds\leq K \left(\left[\frac{t}{L}\right]+1\right)\varepsilon \Vert I_0\Vert_{H^{s}}.$$
Therefore we deduce the estimate, as $\delta_0>0$
$$\int_{x_1>0}\mathbf{U}_{\eps}^2 dx \le \frac{C}{\delta_0} \int_{x_1>0}\mathbf{U}_{0}^2
dx+\eps \frac{K}{\delta_0}\left(\left[\frac{t}{L}\right]+1\right) \Vert I_0\Vert_{H^{s}}+\frac{C_1}{\delta_0} \int_0^t\int_{x_1>0} \mathbf{U}_{\eps}^2 dx \;ds.$$
By Gronwall's lemma we find
$$\|\mathbf{U}_{\eps}\|^2_{L^{2}}(t) \le \frac{C}{\delta_0}\left(\|\mathbf{U}_{0}\|^2_{L^{2}}
+\eps \frac{K}{C}\left(\left[\frac{t}{L}\right]+1\right) \Vert I_0\Vert_{H^{s}} \right)e^{\frac{C_1}{\delta_0}t}$$
remainning bounded for all finite times.

Thus, for all $T<+\infty$ 
we obtain
 that $$\mathbf{U}_{\eps}\in
L^{\infty}\left([0,T],L^2\left(\{x_1>0\}\times\R^{n-1}\right)\right).$$
If $I_0=0$ for $t\in ]0,L[$ a finite number of times $m$, we obtain the same result for $\mathbf{U}_{\eps}$.

Hence, by  Ref.~\cite{Gustafsson} we have proved that the chosen boundary conditions ensure the local well-posedness for the Navier-Stokes system in the half space, which can be viewed as a symetrisable incompletely parabolic system. 
We apply now the theory of incompletely parabolic problems~\cite[p. 352]{Gustafsson}
with the result of global well-posedness of the Navier-Stokes system in the half space with the Dirichlet boundary conditions\cite{Matsumura2} for the velocity and with the initial data $\rho_{\varepsilon}(0)-\rho_0\in H^3(\{x_1>0\}\times\R^{n-1}))$ and $\mathbf{v}_{\varepsilon}(0)\in H^3(\{x_1>0\}\times\R^{n-1})$ small enough. Hence,  for sufficient regular initial data
$\mathbf{U}_{0}\in H^{3}(\{x_1>0\}\times\R^{n-1})$ ($n\leq 3$) for all finite time $T<\infty$, we obtain by the energy method that $\mathbf{U}_{\eps}\in
L^{\infty}([0,T],H^{3}(\{x_1>0\}\times\R^{n-1}))$.

To ensure that $\textbf{U}_0$ defined in Eq.~(\ref{defU0NSKZK}) belongs $H^3(\{x_1>0\}\times\R^{n-1})$ we need to take $I_0\in H^s(\mathbb{T}_t\times\mathbb{R}^{n-1})$ such that 
$$\overline{\rho}_{\varepsilon}\in C([0,+\infty[;H^3(\{x_1>0\}\times\R^{n-1}),\;
\overline{\textbf{v}}_{\varepsilon}\in C([0,+\infty[;H^3(\{x_1>0\}\times\R^{n-1}).$$
By Theorem~\ref{wpglopkzknpe}, 
$I_0\in H^s(\mathbb{T}_t\times\mathbb{R}^{n-1})$ implies while $s-2k\geq 0$ that
$$I(\tau,z,y)\in C^k(\{x_1>0\};H^{s-2k}(\mathbb{T}_{\tau}\times\mathbb{R}^{n-1})),$$
but we can also say\cite{Ito}, thanks to Point~4 of Theorem~\ref{wpglopkzknpe}, that
$$\partial_z^k I(\tau,z,y)\in L^2(\{x_1>0\};H^{s-2k}(\mathbb{T}_{\tau}\times\mathbb{R}^{n-1})).$$
Considering the expressions of $\overline{\rho}_{\varepsilon}$ and $\overline{\textbf{v}}_{\varepsilon}$ 
$$\overline{\rho}_{\varepsilon}=\rho_0+\eps I-\frac{\eps^2}{\rho_0}\left(\frac{\gamma-1}{2}I^2-\frac{\nu}{c^2}\del_\tau I\right),\; \overline{\textbf{v}}_{\varepsilon}=\frac{c^2}{\rho_0}\left(\frac{\eps}{c}I-\eps^2 \del_{\tau}^{-1}\partial_z I; \eps^\frac{3}{2} \del_{\tau}^{-1} \nabla_y I\right),$$
the least regular term  is $\partial_{\tau}^{-1}\partial_z I$.
 Thus we need to ensure
$$\partial_z I\in  C([0,+\infty[;H^3(\{x_1>0\}\times\R^{n-1}),$$ which leads us to take $s\geq 10$ in order to have
$$\partial_z^k I(\tau,z,y)\in L^2(\{x_1>0\};H^{s-2k}(\mathbb{T}_{\tau}\times\mathbb{R}^{n-1}))$$ for $k\leq 4$ with $s-2k\geq 2$ as we want to have the continuity on time.
This choice of the regularity for $I_0$ allows us to control the boundary terms appearing from the integration by parts in the energy method. 
 Indeed, 
we can perform analogous computations as in Ref.~\cite{Dafermos} p.103 to control the spatial derivative of $\mathbf{U}_{\eps}$ of the order less or equal to $3$ and directly verify that all boundary terms are controled by $t \|I_0\|_{H^s}$, what is actually is a consequence of the well-posedness\cite{Matsumura2} in $H^3$. 

Thus, we obtain the existence of the unique local solution of the Navier-Stokes
system with 
$$\rho_{\eps}\in
C\left([0,T],H^{3}\left(\{x_1>0\}\times\R^{n-1}\right)\right)\cap
C^1\left([0,T],H^{2}\left(\{x_1>0\}\times\R^{n-1}\right)\right)$$ and
$$u_{\eps}\in
C\left([0,T],H^{3}\left(\{x_1>0\}\times\R^{n-1}\right)\right)\cap
C^1\left([0,T],H^{1}\left(\{x_1>0\}\times\R^{n-1}\right)\right). $$
%
%
\end{proof}
\subsubsection{Approximation of the solutions of the isentropic Navier-Stokes system with the solutions of the KZK equation.}\label{secValNSKZK}
Knowing from Subsection~\ref{secderNSKZK} that the KZK equation can be derived from the
compressible isentropic Navier-Stokes system~(\ref{NSi1})--(\ref{NSi2})
using the \textit{ansatz}~(\ref{EqAnsKZK1})--(\ref{EqAnsKZK2}) with $I$ and $J$ given by~(\ref{Ikzk}) and~(\ref{Jkzk}) respectively,
 we obtain the following expansion of the Navier-Stokes equations
 \begin{align}
\partial_t \rho_{\varepsilon}+\nabla.(\rho_{\varepsilon} \mathbf{v}_{\varepsilon})=&\varepsilon^2[\frac{\rho_0}{c^2}(2c\partial^2_{z\tau}\Phi-c^2 \Delta_y\Phi)-\frac{\rho_0}{2c^4}(\gamma+1)\partial_{\tau}[(\partial_\tau \Phi)^2]-\frac{\nu}{c^4}\partial^3_\tau \Phi]\nonumber\\
&+\varepsilon^3 R_1^{NS-KZK}\label{massKZK}
\end{align}
and
\begin{align}
\rho_{\varepsilon} [\partial_t \mathbf{v}_{\varepsilon}&+(\mathbf{v}_{\varepsilon}.\nabla)\mathbf{v}_{\varepsilon}]+\nabla p(\rho_{\varepsilon})-\varepsilon \nu \Delta \mathbf{v}_{\varepsilon}
= \varepsilon \tilde{\nabla}[-\rho_0 \partial_{\tau}\Phi+c^2 I]\nonumber\\
&+ \varepsilon^2 \tilde{\nabla}\left[c^2 J+\frac{(\gamma-1)\rho_0}{2c^2 }(\partial_{\tau} \Phi)^2+ \frac{\nu}{c^2}\partial^2_{\tau}\Phi\right]
+\varepsilon^3 \mathbf{R}_2^{NS-KZK},\label{momentKZK}
\end{align}
where $R_1^{NS-KZK}$ and $\mathbf{R}_2^{NS-KZK}$ are the remainder terms given in~\ref{Apen1}.
So, as it was previously explained for the approximation of the Navier-Stokes by the Kuznetsov equation in Subsection~\ref{secapNSKuz}, if we consider a solution of the KZK equation $I$ and define by it the functions $\Phi$ and $J$, then we define according to  \textit{ansatz}~(\ref{EqAnsKZK1})--(\ref{EqAnsKZK2}) $\overline{\rho}_{\varepsilon}$ and  $\overline{ \mathbf{v}}_{\varepsilon}$ (see Eq.~(\ref{vkzk})), which solve the approximate  system~(\ref{NSAp1})--(\ref{NSAp2}) with the remainder terms $R_1^{NS-KZK}$ and $\mathbf{R}_2^{NS-KZK}$ and, as previously, with $p(\overline{\rho}_{\varepsilon})$ from the state law~(\ref{press})
:
\begin{align}
&\partial_t\overline{ \rho}_{\varepsilon}+\operatorname{div}(\overline{\rho}_{\varepsilon}\overline{ \mathbf{v}}_{\varepsilon})=\varepsilon^3 R_1^{NS-KZK},\label{massaproxKZK}\\
&\overline{\rho}_{\varepsilon} [\partial_t \overline{ \mathbf{v}}_{\varepsilon}+( \overline{ \mathbf{v}}_{\varepsilon}.\nabla) \overline{ \mathbf{v}}_{\varepsilon}]+\nabla p(\overline{\rho}_{\varepsilon})-\varepsilon \nu \Delta \overline{ \mathbf{v}}_{\varepsilon}=\varepsilon^3 \mathbf{R}_2^{NS-KZK}.\label{momaproxKZK}
\end{align}

As usual, we denote  by $\mathbf{U}_{\varepsilon}=(\rho_{\varepsilon}\;,\;\rho_{\varepsilon}\mathbf{v}_{\varepsilon})^t$  the solution of the exact Navier-Stokes system and by $\overline{\mathbf{U}}_{\varepsilon}=(\overline{\rho}_{\varepsilon} \;,\; \overline{\rho}_{\varepsilon} \overline{ \mathbf{v}}_{\varepsilon})^t $  the solution of~(\ref{massaproxKZK})--(\ref{momaproxKZK}). 

We  work on $\mathbb{R}_+\times \mathbb{R}^{n-1}$ ($n=2$ or 3) due to the domain  of the well-posedness for the KZK equation. In this case the Navier-Stokes system is locally well-posed with non homogeneous  boundary conditions of $\overline{\mathbf{U}}_{\varepsilon}$,  as they are directly determined by the initial condition $I_0$ of the KZK equation~(\ref{NPEcau2}) according to Theorem~\ref{thENS}. 
Knowing the existence results for  two problems, we validate the approximation of $\mathbf{U}_{\varepsilon}$ by $\overline{\mathbf{U}}_{\varepsilon}$ following Ref.~\cite{Roz3} and Subsection~\ref{secapNSKuz}:
\begin{theorem}\cite{Roz3}\label{ThAprKZKNS}
Let $n=2$ or $3$, $s\ge 10$ and   Theorem~\ref{thENS} hold. 
Then there exist constants $C>0$ and $K>0$ such that we have the following stability estimate 
$$0\leq t\leq\frac{C}{\varepsilon},\;\;\Vert \mathbf{U}_{\varepsilon}-\overline{\mathbf{U}}_{\varepsilon}\Vert^2_{L^2(\mathbb{R}_+\times\mathbb{R}^{n-1})}(t)\leq K\eps^3t  e^{K\varepsilon t}\leq 9\varepsilon^2. $$
\end{theorem}
\begin{remark}
The regularity of $I_0\in H^s(\mathbb{T}_t\times\mathbb{R}^{n-1})$ with $s>8$ (see Table~\ref{TABLE}) is minimal to ensure  that $R_1^{NS-KZK}$ and $\textbf{R}_2^{NS-KZK}$, see~\ref{Apen1}, belongs to $C([0,+\infty[;L^2(\mathbb{R}_+\times\mathbb{R}^{n-1}))$. 

 Indeed, if $I_0\in H^s(\mathbb{T}_t\times \mathbb{R}^{n-1})$ with $s>\max\{8,\frac{n}{2}\}$, then for $0\leq k \leq 4$
$$I(\tau,z,y)\in C^k(\lbrace z >0\rbrace;H^{s-2k}(\mathbb{T}_{\tau}\times \mathbb{R}^{n-1})).$$
Let us denote $\Omega=\mathbb{T}_{\tau}\times \mathbb{R}^{n-1}$.
Given the equations for $\overline{\rho}_{\varepsilon}$ by~(\ref{rhokzk}) with~(\ref{Ikzk}) and~(\ref{Jkzk}) and for $\overline{\textbf{v}}_{\varepsilon}$  by~(\ref{vkzk}) respectively,  we have for $0\leq k\leq 2$
\begin{align*}
&\partial^k_z\overline{\rho}_{\varepsilon}(\tau,z,y)\in  C(\lbrace z >0\rbrace;H^{s-1-2k}(\Omega)),\;
\partial^k_z\overline{\textbf{v}}_{\varepsilon}(\tau,z,y) \in  C(\lbrace z >0\rbrace;H^{s-2-2k}(\Omega)),
\end{align*}
but we can also say\cite{Ito} thanks to Point~4 of Theorem~\ref{wpglopkzknpe} that 
\begin{align*}
&\partial^k_z\overline{\rho}_{\varepsilon}(\tau,z,y)\in  L^2(\lbrace z >0\rbrace;H^{s-1-2k}(\Omega)),\;
\partial^k_z\overline{\textbf{v}}_{\varepsilon}(\tau,z,y) \in  L^2(\lbrace z >0\rbrace;H^{s-2-2k}(\Omega)).
\end{align*}
This implies for $0\leq k\leq 2$ (as $s>8$)  that $s-2-2k>2$ and
\begin{align*}
\partial^k_z\overline{\rho}_{\varepsilon}(\tau,z,y)\in &C(\mathbb{T}_{\tau}; L^2(\lbrace z >0\rbrace;H^{s-1-2k}( \mathbb{R}^{n-1}))),\\
\partial^k_z\overline{\textbf{v}}_{\varepsilon}(\tau,z,y) \in &C(\mathbb{T}_\tau; L^2(\lbrace z >0\rbrace;H^{s-2-2k}( \mathbb{R}^{n-1}))).
\end{align*}
Hence we find 
\begin{align*}
\overline{\rho}_{\varepsilon}(t,x_1,x'),\; \overline{\textbf{v}}_{\varepsilon}(t,x_1,x')\in & C([0,+\infty[;H^2(\lbrace x_1 >0\rbrace\times \mathbb{R}^{n-1}).
\end{align*}
As in addition for $0\leq k\leq 1$, considering $\overline{\rho}_{\varepsilon}$ and $\overline{\mathbf{v}}_{\varepsilon}$ as functions of $(\tau,z,y)$,
\begin{align*}
&\partial^k_z\partial_{\tau}\overline{\rho}_{\varepsilon}\in  C(\lbrace z >0\rbrace;H^{s-2-2k}(\Omega)),\;
\partial^k_z\partial_{\tau}\overline{\mathbf{v}}_{\varepsilon} \in  C(\lbrace z >0\rbrace;H^{s-3-2k}(\Omega)),
\end{align*}
 we deduce in the same way that
\begin{align*}
\partial_t \overline{\rho}_{\varepsilon}(t,x_1,x'),\; \partial_t \overline{\textbf{v}}_{\varepsilon}(t,x_1,x')\in & C([0,+\infty[;H^1(\lbrace x_1 >0\rbrace\times \mathbb{R}^{n-1})).
\end{align*}
These regularities of  $\overline{\rho}_{\varepsilon}$ and $\overline{\mathbf{v}}_{\varepsilon}$ viewed as functions of $(t,x_1,x')$ allow to have all left-hand terms in the approximated Navier-Stokes system (\ref{massaproxKZK})--(\ref{momaproxKZK}) of the regularity $C([0,T];L^2(\lbrace x_1 >0\rbrace\times \mathbb{R}^{n-1}))$ and the remainder terms in the right-hand side inherit it.
\end{remark}
\subsection{Navier-Stokes system and the NPE equation.}\label{secNSNPE}
\subsubsection{Derivation of the NPE equation}\label{secderNSNPE}

The NPE equation (Nonlinear Progressive wave Equation), initially derived by McDonald and Kuperman~\cite{McDonald}, is an example of a paraxial approximation in the aim to describe short-time pulses and a long-range propagation, for instance, in an ocean wave-guide, where the refraction phenomena are important.
To compare to the KZK equation we use the following paraxial change of variables
\begin{equation}\label{paraxpotenNPE}
u(t,x_1,x')=\Psi(\varepsilon t,x_1-ct,\sqrt{\varepsilon}x')=\Psi(\tau,z,y),
\end{equation}
with
\begin{equation}\label{chvarnpe}
\tau=\varepsilon t,\;\;z=x_1-ct,\;\;y=\sqrt{\varepsilon}x'.
\end{equation}

\begin{figure}[h!]
\begin{center}
\psfrag{a}{$x_1$}\psfrag{b}{$\mathbf{x'}$}\psfrag{c}{$t$}\psfrag{NS}{\small{Navier-Stokes/}}\psfrag{E}{\small{
Euler
$(x_1,\mathbf{x'},t)$}}\psfrag{a1}{$z=x_1-ct$}\psfrag{b1}{$\mathbf{y}=\sqrt{\epsilon}
\mathbf{x'}$}\psfrag{c1}{$\tau=\epsilon t$}\psfrag{NPE}{\small{NPE
$(\tau,z,\mathbf{y})$}}
  \includegraphics[width=.6\textwidth]{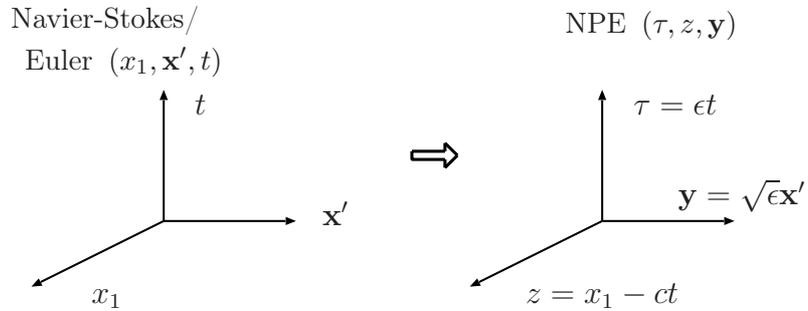}
                  \end{center}
   \caption{Paraxial change of variables for the profiles $U(\epsilon
t,x_1-ct,\sqrt{\epsilon}\mathbf{x'})$.}\label{fig3}
\end{figure}
For the velocity we have
\begin{equation}\label{vNPE}
\mathbf{v}_{\varepsilon}(t,x_1,x')= -\varepsilon \nabla u(t,x_1,x')=-\varepsilon(\partial_z \Psi,\sqrt{\varepsilon} \nabla_y \Psi)(\tau,z,y).
\end{equation}
If we compare the NPE equation to the isentropic Navier-Stokes system this method of approximation does not allow to keep the Kuznetsov \textit{ansatz} of perturbations~(\ref{rhoKuz})--(\ref{uKuz}) imposing~(\ref{rho1K})--(\ref{rho2K})
just by introducing the new paraxial profiles $\Psi$ for $u$, $\xi$ for $\rho_1$ and $\chi$ for $\rho_2$ and taking the term of order $0$ in $\varepsilon$ as it was done in the case of the KZK-approximation. This time 
the paraxial change of variables~(\ref{chvarnpe})
 for $\rho_1$ and $\rho_2$ defined in (\ref{rho1K})--(\ref{rho2K}) gives
\begin{align*}
\rho_1=&-\frac{\rho_0}{c}\partial_z\Psi+\varepsilon \frac{\rho_0}{c^2}\partial_{\tau}\Psi,\\
\rho_2=&-\frac{\rho_0(\gamma-2)}{2c^2}(\partial_z \Psi)^2-\frac{\rho_0}{2c^2}(\partial_z \Psi)^2-\frac{\nu}{\rho_0}\partial_z^2 \Psi\\
&+\varepsilon\left[\frac{\rho_0(\gamma-2)}{2c^3}\partial_{ z}\Psi \partial_{\tau}\Psi-\frac{\rho_0}{2c^2}(\nabla_y \Psi)^2-\frac{\nu}{c^2}\Delta_y \Psi\right]\\
&+\varepsilon^2 \left(-\frac{\rho_0(\gamma-2)}{2c^4}\right)(\partial_{\tau }\Psi)^2.
\end{align*}
Thus one of the terms in the $\rho_1$-extension takes part on the second order corrector of $\rho_{\varepsilon}$:
\begin{equation}\label{rhoNPE}
\rho_{\varepsilon}(t,x_1,x')= \rho_0+\varepsilon \xi(\tau,z,y)+\varepsilon^2 \chi(\tau,z,y),
\end{equation}
with
\begin{align}
\xi(\tau,z,y)=&-\frac{\rho_0}{c}\partial_z\Psi,\label{P1NPE}\\
\chi(\tau,z,y)=&\frac{\rho_0}{c^2}\partial_{\tau}\Psi-\frac{\rho_0(\gamma-1)}{2c^2}(\partial_z \Psi)^2-\frac{\nu}{c^2}\partial^2_z\Psi.\label{P2NPE}
\end{align}

The obtained \textit{ansatz}~(\ref{vNPE})--(\ref{rhoNPE})  applied to the Navier-Stokes system gives 
\begin{align*}
\partial_t\rho_{\varepsilon}+\operatorname{div}(\rho_{\varepsilon} \mathbf{v}_{\varepsilon})=&\varepsilon^2(-\frac{2\rho_0}{c})\left(\partial^2_{\tau z}\Psi-\frac{(\gamma+1)}{4}\partial_z(\partial_z\Psi)^2-\frac{\nu}{2\rho_0}\partial^3_z\Psi+\frac{c}{2}\Delta_y\Psi\right)\\
&+\varepsilon^3 R_1^{NS-NPE},
\end{align*}
\begin{align*}
\rho_{\varepsilon} [\partial_t \mathbf{v}_{\varepsilon}&+(\mathbf{v}_{\varepsilon}.\nabla)\mathbf{v}_{\varepsilon}]+\nabla p(\rho_{\varepsilon})-\varepsilon \nu \Delta \mathbf{v}_{\varepsilon}
=\varepsilon \nabla\left(\xi+\frac{\rho_0}{c}\partial_z\Psi\right)\\
&+c^2\varepsilon^2\nabla\big[\chi-\frac{\rho_0}{c^2}\partial_{\tau}\Psi+\frac{\rho_0(\gamma-1)}{2c^2}(\partial_z \Psi)^2
+\frac{\nu}{c^2}\partial^2_z\Psi\Big]
+\varepsilon^3 \mathbf{R}_2^{NS-NPE}.
\end{align*}
The remainder term in the conservation of mass is given by
\begin{align}
\varepsilon^3 R_1^{NS-NPE}=& \varepsilon^3\big( \partial_{\tau} \chi-\nabla_y \xi\;\nabla_y\Psi-\xi\;\Delta_y\Psi-\partial_z\chi\;\partial_z\Psi-\chi\;\partial^2_z\Psi)\nonumber\\
&+\varepsilon^4(-\nabla_y \chi\;\nabla_y\Psi-\chi \;\Delta_y\Psi),\label{R1NSNPE}
\end{align}
while in the conservation of momentum along the $x_1$ axis it is given by\begin{align}
&\varepsilon^3 \mathbf{R}_2^{NS-NPE}. \overrightarrow{e}_1=\varepsilon^3\Big[-\frac{\rho_0}{c}\partial_z\Psi \;\partial^2_{\tau z}\Psi+\frac{\rho_0}{2}\partial_z(\nabla_y\Psi)^2+\nu \partial_z \Delta_y\Psi+\frac{\xi}{2}\partial_z(\partial_z\Psi)^2\nonumber\\
&+c\chi\partial^2_z\Psi\Big]
+\varepsilon^4\left(\frac{\xi}{2}\partial_z(\nabla_y\Psi)^2-\chi\partial^2_{\tau z}\Psi+\frac{\chi}{2}\partial_z(\partial_z\Psi)^2\right)
+\varepsilon^5\frac{\chi}{2}\partial_z(\nabla_y\Psi)^2,\label{R2NSNPEe1}
\end{align}
and along all transversal direction $x_j$  to the propagation  $x_1$-axis\begin{align}
&\varepsilon^3 \mathbf{R}_2^{NS-NPE}. \overrightarrow{e}_j=\varepsilon^{\frac{7}{2}}\Big[-\frac{\rho_0}{c}\partial_z\Psi \;\partial^2_{\tau y_j}\Psi+\frac{\rho_0}{2}\partial_{y_j}(\nabla_y\Psi)^2+\nu \partial_{y_j} \Delta_y\Psi+\frac{\xi}{2}\partial_{y_j}(\partial_z\Psi)^2\nonumber\\
&+c\chi\partial^2_{z y_j}\Psi\Big]
+\varepsilon^{\frac{9}{2}}\left(\frac{\xi}{2}\partial_{y_j}(\nabla_y\Psi)^2-\chi\partial^2_{\tau y_j}\Psi+\frac{\chi}{2}\partial_{y_j}(\partial_z\Psi)^2\right)
+\varepsilon^{\frac{11}{2}}\frac{\chi}{2}\partial_{y_j}(\nabla_y\Psi)^2.\label{R2NSNPEej}
\end{align}
As all previous models, for this \textit{ansatz}, the NPE equation
\begin{equation}\label{NPE}
 \partial^2_{\tau z}\Psi-\frac{(\gamma+1)}{4}\partial_z(\partial_z\Psi)^2-\frac{\nu}{2\rho_0}\partial^3_z\Psi+\frac{c}{2}\Delta_y\Psi=0
\end{equation}
 appears  as the second order approximation of the isentropic Navier-Stokes system up to the terms of the order of $\mathcal{O}(\varepsilon^3)$.
In the sequel we work with the NPE equation satisfied by $\xi$ (see Eq.~(\ref{P1NPE}) for the definition) 
\begin{equation}\label{NPE2}
\partial^2_{\tau z} \xi+\frac{(\gamma+1)c}{4\rho_0}\partial_z^2[(\xi)^2]-\frac{\nu}{2\rho_0}\partial^3_z  \xi+\frac{c}{2}\Delta_y  \xi=0.
\end{equation}
Looking at Figs~\ref{fig2} and~\ref{fig3} together with~(\ref{KZKI}) and~(\ref{NPE}) we see that we have a bijection between the variables of the KZK and NPE equations
defined by the relations
\begin{equation}\label{bijKZKNPE}
z_{NPE}=-c\tau_{KZK} \hbox{ and }\tau_{NPE}=\varepsilon \tau_{KZK}+\frac{z_{KZK}}{c},
\end{equation}
which implies  for the derivatives
$$\partial_{\tau_{NPE}}=c\partial_{z_{KZK}}\hbox{ and } \partial_{z_{NPE}}=-\frac{1}{c}\partial_{\tau_{KZK}}.$$
Thus, as it was mentioned in Introduction, the known mathematical results for the KZK equation can be directely applied for the NPE equation. 
\subsubsection{Well posedness of the NPE equation.}
We consider the Cauchy problem:
\begin{equation}\label{npecau}
\left\lbrace
\begin{array}{c}
\partial^2_{\tau z} \xi+\frac{(\gamma+1)c}{4\rho_0}\partial_z^2[(\xi)^2]-\frac{\nu}{2\rho_0}\partial^3_z  \xi+\frac{c}{2}\Delta_y  \xi=0\hbox{ on }\mathbb{R}_+\times\mathbb{T}_{z}\times\mathbb{R}^{n-1},\\
\xi(0,z,y)=\xi_0(z,y)\hbox{ on }\mathbb{T}_{z}\times\mathbb{R}^{n-1},
\end{array}\right.
\end{equation}
in the class of $L-$periodic functions with respect to the variable $z$ and with mean value zero along $z$.
The introduction of the operator $\partial_z^{-1}$ defined similarly to $\partial_{\tau}^{-1}$ in Eq.~(\ref{invdtau})
allows us to consider instead of Eq.~(\ref{NPE2}) the following equivalent equation
$$\partial_{\tau } \xi+\frac{(\gamma+1)c}{4\rho_0}\partial_z[(\xi)^2]-\frac{\nu}{2\rho_0}\partial^2_z  \xi+\frac{c}{2}\Delta_y \partial_{z}^{-1} \xi=0\hbox{ on }\mathbb{R}_+\times\mathbb{T}_{z}\times\mathbb{R}^{n-1} .$$
As a consequence we can use the results of Subsection~\ref{secWPKZKNPE} if we replace $\tau$ by $z$.
In the same time 
for the viscous case it holds the following theorem: 
\begin{theorem}\label{ThWPNPE}
Let $n\geq 2$, $\nu>0$, $s>\max\left(4,\left[\frac{n}{2}\right]+1\right)$ and $\xi_0\in H^{s}(\mathbb{T}_{z}\times\mathbb{R}^{n-1})$ with zero mean value along $z$. Then there exists a constant $k_2>0$ such that if 
\begin{equation}\label{smallcondWPnpe}
\Vert \xi_0\Vert_{H^{s}(\mathbb{T}_{z}\times\mathbb{R}^{2})}<k_2,
\end{equation}
then the Cauchy problem for the NPE equation~(\ref{npecau}) has a unique global in time solution 
\begin{equation}\label{RegNPE3}
\xi\in \bigcap_{i=0}^2 C^i([0,+\infty[,H^{s-2i}(\mathbb{T}_{z}\times\mathbb{R}^{2}))
\end{equation}
satisfying the zero mean value condition along $z$.
Moreover for $\Psi$ according with Eq.~(\ref{P1NPE}) we have
$$\Psi:=-\frac{c}{\rho_0}\partial_{z}^{-1} \xi\in \bigcap_{i=0}^2 C^i([0,+\infty[,H^{s-2i}(\mathbb{T}_{z}\times\mathbb{R}^{2}))$$
and it also satisfies the zero mean value condition along $z$, $i.e.$
$\int_0^L \Psi(\tau,l,y)dl=0.$
\end{theorem}

\begin{proof}
For $\xi_0\in H^{s}(\mathbb{T}_{z}\times\mathbb{R}^{n-1})$ small enough, the existence of a global in time solution 
$$ \xi\in \bigcap_{i=0}^1 C^i([0,+\infty[,H^{s-2i}(\mathbb{T}_{z}\times\mathbb{R}^{n-1}))$$
of the Cauchy problem for the NPE equation~(\ref{npecau}) comes from Theorem~\ref{wpglopkzknpe}. We also have the desired regularity by a simple bootstrap argument. Moreover the formula for $\partial_z^{-1}$ (see the equivalent definition of $\partial_{\tau}^{-1}$ in Eq.~(\ref{invdtau})) implies for $s\geq 1$ by the Poincar\'e inequality
$$\Vert \partial_z^{-1} \xi\Vert_{H^s(\mathbb{T}_{z}\times\mathbb{R}^{n-1})}\leq C \Vert \partial_z \partial_z^{-1} \xi\Vert_{H^s(\mathbb{T}_{z}\times\mathbb{R}^{n-1})} \leq C \Vert \xi\Vert_{H^s(\mathbb{T}_{z}\times\mathbb{R}^{n-1})},$$
which gives us the same regularity for $\Psi$.
\end{proof}

\subsubsection{Approximation of the solutions of the isentropic Navier-Stokes system by the solutions of the NPE equation.}\label{secapNSNPE}
By Subsection~\ref{secWPKZKNPE}, this time the approximation domain is $\mathbb{T}_{x_1}\times \mathbb{R}^{n-1}$. 
Let $\xi$ be a sufficiently regular solution of the Cauchy problem~(\ref{npecau}) for the NPE equation in $\mathbb{T}_z\times\mathbb{R}^{n-1}$.
Then, taking $\xi$ and $\chi$ according to formulas~(\ref{P1NPE})-(\ref{P2NPE}), with $\Psi$ defined using the operator $\partial_z^{-1}$ equivalent to $\partial_{\tau}^{-1}$ (see Eq.~~(\ref{invdtau})), we define
$\overline{\rho}_{\varepsilon}$ and $\overline{ \mathbf{v}}_{\varepsilon}$ by formulas~(\ref{vNPE})--(\ref{rhoNPE}). For $\overline{\rho}_{\varepsilon}$ and $\overline{ \mathbf{v}}_{\varepsilon}$ we obtain a solution of the approximate system~(\ref{NSAp1})--(\ref{NSAp2}) defined on $\mathbb{R}_+\times\mathbb{T}_{x_1}\times\mathbb{R}^{n-1}$
with $p(\overline{\rho}_{\varepsilon})$ from the state law~(\ref{press}), but with the remainder terms  $R_1^{NS-NPE}$ and $\textbf{R}_2^{NS-NPE}$ defined respectively in Eqs.~(\ref{R1NSNPE})--(\ref{R2NSNPEej}) instead of $R_1^{NS-Kuz}$ and  $\textbf{R}_2^{NS-Kuz}$. 


In what following we consider the three dimensional case, knowing, thanks to the energy method used in Ref.~\cite{Matsumura} on $\mathbb{R}^3$, that  
the Cauchy problem for the Navier-Stokes system is globally well-posed in $\mathbb{T}_{x_1}\times\mathbb{R}^{2}$ for sufficiency small initial data (see Ref.~\cite{Matsumura} Theorem~7.1, p.~100 or Ref.~\cite{Cao}): 
\begin{theorem}\label{ThWPNSper}
There exists a constant $k_1>0$ such that if  the initial data 
\begin{equation}\label{NSinitdataper}
\rho_{\varepsilon}(0)-\rho_0\in H^3(\mathbb{T}_{x_1}\times\mathbb{R}^{2}) ,\;\;\;\mathbf{v}_{\varepsilon}(0)\in H^3(\mathbb{T}_{x_1}\times\mathbb{R}^{2})
\end{equation}
satisfy $$\Vert \rho_{\varepsilon}(0)-\rho_0\Vert_{H^3(\mathbb{T}_{x_1}\times\mathbb{R}^{2})}+  \Vert \mathbf{v}_{\varepsilon}(0)\Vert_{H^3(\mathbb{T}_{x_1}\times\mathbb{R}^{2})}<k_1,$$ 
and $\rho_{\varepsilon}(0)-\rho_0$ and $\mathbf{v}_\varepsilon(0)$ have a zero mean value among $x_1$ then the  Cauchy problem~(\ref{NSi1})-(\ref{press}) on $\mathbb{T}_{x_1}\times\mathbb{R}^{2}$ with the initial data~(\ref{NSinitdataper}) has a unique global in time solution $(\rho_{\varepsilon},$ $\mathbf{v}_{\varepsilon} )$  such that \begin{equation}\label{EqregRhoNSper}
 \rho_{\varepsilon}-\rho_0\in C([0,+\infty[;H^3(\mathbb{T}_{x_1}\times\mathbb{R}^{2}))\cap C^1([0,+\infty[;H^2(\mathbb{T}_{x_1}\times\mathbb{R}^{2})),
\end{equation} 
which implies
\begin{equation}\label{EqregRhoNSper2}
 \rho_{\varepsilon}-\rho_0\in C([0,+\infty[;H^3(\mathbb{T}_{x_1}\times\mathbb{R}^{2}))\cap C^1([0,+\infty[;H^1(\mathbb{T}_{x_1}\times\mathbb{R}^{2}))
\end{equation}
and 
\begin{equation}\label{EqregvNSper}
      \mathbf{v}_{\varepsilon}\in C([0,+\infty[;H^3(\mathbb{T}_{x_1}\times\mathbb{R}^{2}))\cap C^1([0,+\infty[;H^1(\mathbb{T}_{x_1}\times\mathbb{R}^{2})).
    \end{equation}
Moreover for all time for $\rho_{\varepsilon}-\rho_0$ and $\mathbf{v}_\varepsilon$ have a zero mean value along $x_1$.
\end{theorem}

The existence results for the Cauchy problems of the Navier-Stokes system~(\ref{NSi1})-(\ref{press}) and the NPE equation~(\ref{npecau}) 
allow us to establish the global existence of $\mathbf{U}_{\varepsilon}$ and $\overline{\mathbf{U}}_{\varepsilon}$, considered in the NPE approximation framework: 
 \begin{theorem}\label{ThExistUUbarNSNPE}
 Let $n=3$.
 There exists  a constant $k>0$ such that if the initial datum 
 $\xi_0\in H^5(\mathbb{T}_{z}\times\mathbb{R}^{2})$ for the Cauchy problem for the NPE equation~(\ref{npecau}) (necessarily $k\leq k_2$, see Theorem~\ref{ThWPNPE})
 is sufficiently small
 $$\Vert \xi_0\Vert_{H^5(\mathbb{T}_{z}\times\mathbb{R}^{n-1})}<k,$$
and has a zero mean value then there exist global in time solutions $\overline{\mathbf{U}}_\eps=(\overline{\rho}_{\varepsilon},\; \overline{\rho}_{\varepsilon} \overline{ \mathbf{v}}_{\varepsilon})^t$ of the approximate Navier-Stokes system~(\ref{NSmatrA})  and $\mathbf{U}_\eps=(\rho_{\varepsilon},\;\rho_{\varepsilon}\mathbf{v}_{\varepsilon})^t$ of the exact  Navier-Stokes system~(\ref{NSmatr}) respectively, with the same regularity corresponding to~(\ref{EqregRhoNSper2}) and~(\ref{EqregvNSper}) and a zero mean value in the $x_1$-direction, both considered with the state law~(\ref{press}) and with the same initial data 
 \begin{eqnarray}
    &&(\bar{\rho}_\eps-\rho_\eps)|_{t=0} =0,
\quad (\bar{\mathbf{v}}_\eps-\mathbf{v}_\eps)|_{t=0} = 0,\label{bbannpe}\end{eqnarray}
where $\bar{\rho}_\eps|_{t=0}$ and $\bar{\mathbf{v}}_\eps|_{t=0}$ are constructed as the functions of the initial datum for NPE equation $\xi_0$ according to formulas~(\ref{vNPE})--(\ref{P2NPE}).
\end{theorem}
\begin{proof}
The proof is essentially the same as for Theorem~\ref{ThExistUUbarNSK}. 
According to Theorem~\ref{ThWPNPE} with $s=5$, the datum $\xi_0$ is regular enough so that 
$$\rho_{\varepsilon}-\rho_0\vert_{t=0}\in H^3(\mathbb{T}_{x_1}\times\mathbb{R}^{2})\hbox{ and }\mathbf{v}_{\varepsilon}\vert_{t=0}\in [H^3(\mathbb{T}_{x_1}\times\mathbb{R}^{2})]^3$$
constructed with the help of formulas~(\ref{vNPE})--(\ref{P2NPE}) in order to apply Theorem~\ref{ThWPNSper}. These formulas together with Theorem~\ref{ThWPNPE} imply that $\overline{\rho}_{\varepsilon}$ and $\overline{\textbf{v}}_\varepsilon$ have the desired regularity.
\end{proof}

Thanks to Theorem~\ref{ThExistUUbarNSNPE} we validate the approximation of $\mathbf{U}_{\varepsilon}$ by $\overline{\mathbf{U}}_{\varepsilon}$ following Ref.~\cite{Roz3}: 
\begin{theorem}\label{ThapproxNSNPE}
 Let  
 $\nu>0$ and $\eps>0$ be fixed and all assumptions of Theorem~\ref{ThExistUUbarNSNPE} hold. Then estimates of Theorem~\ref{ThapproxNSKuz} hold  in $L^2(\mathbb{T}_{x_1}\times\mathbb{R}^{2})$.
 \end{theorem}
The proof being the same as in Theorem~\ref{ThapproxNSKuz} is omitted. In fact it is due to the same Eqs.~~(\ref{NSmatr}) and~(\ref{NSmatrA}) with just different remainders terms of the same order on $\varepsilon$.

It is also easy to see using the previous arguments that the minimum regularity of the initial data (see Table~\ref{TABLE}) to have the remainder terms 
$$
R_1^{NS-NPE}\hbox{ and }\textbf{R}_2^{NS-NPE}\in  C([0,+\infty[;L^2(\mathbb{T}_{x_1}\times \mathbb{R}^{2}))
$$
corresponds to
$\xi_0\in H^s(\mathbb{T}_{x_1}\times \mathbb{R}^{2})$ with $s \geq 4$ since then for $0\leq k \leq 2$
$$\xi(\tau,z,y)\in C^k([0,+\infty[\rbrace;H^{s-2k}(\mathbb{T}_{z}\times \mathbb{R}^{2})),$$
which finally implies with formulas~(\ref{vNPE})--(\ref{P2NPE}) that
\begin{align*}
\overline{\rho}_{\varepsilon}(t,x_1,x')\in C([0,+\infty[;H^2(\mathbb{T}_{x_1}\times \mathbb{R}^{2}))\cap C^1([0,+\infty[;L^2(\mathbb{T}_{x_1}\times \mathbb{R}^{2})),\\
\overline{\textbf{v}}_{\varepsilon}(t,x_1,x')\in C([0,+\infty[;H^3(\mathbb{T}_{x_1}\times \mathbb{R}^{2}))\cap C^1([0,+\infty[;H^1(\mathbb{T}_{x_1}\times \mathbb{R}^{2})).
\end{align*}

%
 
%
%

\section{Approximation of the Euler system.}\label{secEul}
Let us consider the following isentropic Euler system:
\begin{align}
&\partial_t\rho_{\varepsilon}+\operatorname{div}(\rho_{\varepsilon} \mathbf{v}_{\varepsilon})=0,\label{Euli1}\\
& \rho_{\varepsilon} [\partial_t \mathbf{v}_{\varepsilon}+(\mathbf{v}_{\varepsilon}.\nabla)\mathbf{v}_{\varepsilon}]+\nabla p(\rho_{\varepsilon})=0\label{Euli2}
\end{align}
with $p(\rho_{\varepsilon})$ given in Eq.~(\ref{press}). 
We use all notations of Section~\ref{secNS} just taking $\nu=0$. 

Let us consider two and three dimensional cases.
The entropy $\eta$ of the isentropic Euler system, defined in Eq.~(\ref{entropy}), is of class $C^3$ and in addition $\eta''(U_{\varepsilon})$ is positive definite for $\rho_{\varepsilon}>0$. Moreover, from~(\ref{NSmatr}) we see that $G_i\in C^\infty$ with respect to $U_{\varepsilon}$ for $\rho_{\varepsilon}>0$. Then we can apply Theorem~5.1.1 p. 98 in Ref.~\cite{Dafermos} which gives us the local well-posedness of the Euler system:
\begin{theorem}\label{ThWPEulDafermos}\cite{Dafermos}
In $\mathbb{R}^n$ for $n=2$ or $3$, suppose the initial data $\mathbf{U}_{\varepsilon}(0)$ be continuously differentiable on $\mathbb{R}^n$, take value in some compact set with $\rho_{\varepsilon}(0)>0$, and 
$$\hbox{for }i=1,...,n,\hbox{ } \partial_{x_i} \mathbf{U}_{\varepsilon}(0)\in [H^s(\mathbb{R}^{n})]^{n+1} \hbox{ with }  s>n/2.$$ 
Then there exists $0<T_{\infty}\leq +\infty,$   and a unique continuously differentiable function $\mathbf{U}_{\varepsilon}$ on $\mathbb{R}^3\times [0, T_{\infty}[$ taking value with $\rho_{\varepsilon}>0$, which is a classical solution of the Cauchy problem associated to~(\ref{NSmatr}) with $\nu=0$. Furthermore for $i=1,...,n$
$$\partial_{x_i} \mathbf{U}_{\varepsilon}(t)\in \bigcap_{k=0}^s C^k([0, T_{\infty}[;[H^{s-k}(\mathbb{R}^{n})]^{n+1}).$$
The interval $[0, T_{\infty}[$ is maximal in that if $T_{\infty}< +\infty$ then 
$$\int_0^{T_{\infty}} \sup_{i=1,...,n}\Vert \partial_{x_i}\mathbf{U}_{\varepsilon}\Vert_{[L^{\infty}(\mathbb{R}^n)]^{n+1}}dt=+\infty,$$
and/or the range of $\mathbf{U}_{\varepsilon}(t)$ escapes from every compact subsets of $\mathbb{R}^*_+\times \mathbb{R}^n$ as $t\rightarrow T_{\infty}$. 
\end{theorem}
\begin{remark}
A sufficient condition for the initial data to apply Theorem~\ref{ThWPEulDafermos} is to have $\rho_{\varepsilon}(0)-\rho_0\in H^3(\mathbb{R}^n)$ and $\mathbf{v}_{\varepsilon}(0)\in (H^3(\mathbb{R}^n))^n$ with $\rho_{\varepsilon}(0)>0$.
\end{remark}
To approximate the solutions of the Euler system and the Kuznetsov, the NPE or the KZK equations, we need to know for which time (how long) they exist. In the difference to the viscous case, the inviscid models 
can provide blow-up phenomena as indicated in Theorem~\ref{ThWPEulDafermos} for the Euler system, in Theorem~\ref{ThTimeExistNu0} for the Kuznetsov equation and for the KZK and the NPE equations see Theorem~1.3 in Ref.~\cite{Roz2}. Let us start by summarize what is known on the blow-up time for the Euler system\cite{Alinhac2,Sideris2,Sideris3,Sideris1,Sideris2d,Yin1}.

Due to our framework of the non-linear acoustic, it is important for us to have a potential motion (the irrotational case) and to consider the compressible isentropic Euler system~(\ref{Euli1})--(\ref{Euli2}) with initial data defining a perturbation of order $\varepsilon$ around the constant state $(\rho_0,0)$: 
%
\begin{theorem}\label{ThminTepsSideris}\textbf{(Existence time for the Euler system)}
\begin{enumerate}
\item In $\mathbb{R}^n$ for $n=2$ or $3$, suppose the initial data $$\mathbf{U}_{\varepsilon}(0)=(\rho_{\varepsilon,0},\rho_{\varepsilon,0}\mathbf{v}_{\varepsilon,0})^t $$ be a perturbation of order $\varepsilon$ around the constant state $(\rho_0,0)$ (see Eq.~(\ref{EulerinitAlinh})) and take value such that for $i=1,...,n$, $\partial_{x_i} \mathbf{U}_{\varepsilon}(0)\in [H^s(\mathbb{R}^{n})]^{n+1}$ with $s>n/2$.
Then according to Theorem \ref{ThWPEulDafermos}  there exists a unique classical solution of the Cauchy problem associated to~(\ref{NSmatr}) with $\nu=0$ with a regularity given in Theorem \ref{ThWPEulDafermos}.
Moreover considering a generic constant $C>0$  independent on $\varepsilon$, the existence time $T_\varepsilon$ is estimated by
$T_\varepsilon \geq \frac{C}{\varepsilon}.$
\item \cite{Sideris2,Sideris3,Sideris1,Sideris2d} If $\nabla\times \mathbf{v}_{\varepsilon,0}=0$ and if 
$$\left(\frac{\rho_{\varepsilon,0}}{\rho_0}\right)^{\frac{\gamma-1}{2}}-1\hbox{ and }\textbf{v}_{\varepsilon,0}\hbox{ belong to the energy space } X^m$$ a dense subspaces of $H^m(\mathbb{R}^n)$ with $m\geq 4$ (for instance they can belong to $\mathcal{D}(\mathbb{R}^n)$, see p.7-8 in Ref.~\cite{Sideris1} for the exact definition of $X^m$) then
$$T_\varepsilon\geq \frac{C}{\varepsilon^2}\hbox{ for }n=2, \hbox{ and }
T_\varepsilon\geq \exp\left(\frac{C}{\varepsilon}\right)-1 \hbox{ for }n=3.$$
The regularity is given by energy estimates on $X^m$ which implies at least the same regularity as in Theorem \ref{ThWPEulDafermos} if for $i=1,...,n$, $\partial_{x_i} \mathbf{U}_{\varepsilon}(0)\in [H^{m-1}(\mathbb{R}^{n})]^{n+1}$.
\end{enumerate} 
\end{theorem}
\begin{proof}
The first point is a direct consequence of the proof of Theorem~5.1.1 p. 98 in Ref.~\cite{Dafermos}.
For the second point we refer to Refs.~ \cite{Sideris2,Sideris3,Sideris1,Sideris2d} in order to have estimations of $T_{\varepsilon}$ with the help of energy estimates in the considered energy spaces which are dense subspaces of the usual Sobolev spaces.
\end{proof}

%
Let us pay attention on the optimality of the lifespan in the previous results for two~\cite{Alinhac2} and three dimensional cases~\cite{Yin1}. The following theorem tells us that the lowerbound for the lifespan of the compressible Euler system in the irrotational case found in Theorem~\ref{ThminTepsSideris} is optimal:
\begin{theorem}\label{ThWpEulaxisphesym}\textbf{(Blow-up for the Euler system)}
\begin{enumerate}
\item  \cite{Alinhac2}
In $\mathbb{R}^2$, we consider the initial data given by 
\begin{equation}\label{EulerinitAlinh}
\rho_\varepsilon(0)=\rho_0+\varepsilon \rho_{\varepsilon,0}\hbox{ and }\mathbf{v}_{\varepsilon}(0)=\varepsilon \mathbf{v}_{\varepsilon,0},
\end{equation}
with $\rho_{\varepsilon,0}$ and $\mathbf{v}_{\varepsilon,0}$ of regularity $C^{\infty}$ with a compact support. Moreover 
$$\mathbf{v}_{\varepsilon,0}(x)=v_r \vert x\vert_2 \overrightarrow{e}_r+v_{\theta} \vert x\vert_2 \overrightarrow{e}_{\theta},$$
with $\rho_{\varepsilon,0},\; v_r,\; v_{\theta}\in \mathcal{D}(\mathbb{R}^2)$ depending only on $r=\vert x\vert_2=\sqrt{x_1^2+x_2^2}$ for $x=(x_1,x_2)^t$.

Then the Euler system~(\ref{Euli1})--(\ref{Euli2}) with initial data~(\ref{EulerinitAlinh}) admits a $C^{\infty}$ solution for~$t\in[0,T_{\varepsilon}[$ with
$$\lim_{\varepsilon\rightarrow 0} \varepsilon^2 T_{\varepsilon}=C>0.$$

\item \cite{Yin1}
In $\mathbb{R}^3$, we consider the initial data given by~(\ref{EulerinitAlinh}) with $\rho_{\varepsilon,0}$ and $\mathbf{v}_{\varepsilon,0}$ of regularity $C^{\infty}$ with a compact support. Moreover 
$$\mathbf{v}_{\varepsilon,0}(x)=v_r \vert x\vert_3 \overrightarrow{e}_r,$$
with $\rho_{\varepsilon,0}\hbox{ and } v_r\in \mathcal{D}(\mathbb{R}^3)$ depending only on $r=\vert x\vert_3=\sqrt{x_1^2+x_2^2+x_3^2}$ 
\\for $x=(x_1,x_2,x_3)^t$.
Then the Euler system~(\ref{Euli1})--(\ref{Euli2}) with initial data~(\ref{EulerinitAlinh}) admits a $C^{\infty}$ solution for $t\in[0,T_{\varepsilon}[$ with
$$\lim_{\varepsilon\rightarrow 0} \varepsilon \ln( T_{\varepsilon})=C>0.$$
\end{enumerate}
\end{theorem}

Now let us consider the derivation of the Kuznetsov equation of Subsection~\ref{secderNSKuz} in the assumption $\nu=0$. 
Taking \textit{ansatz}~(\ref{rhoKuz})--(\ref{uKuz}) for $\rho_{\varepsilon}$ and $\mathbf{v}_{\varepsilon}$ 
and imposing~(\ref{rho1K})--(\ref{rho2K}) for $\rho_1$ and $\rho_2$ with $\nu=0$, 
we derive as in Subsection~\ref{secderNSKuz} the  inviscid Kuznetsov equation with the notation $\alpha=\frac{\gamma-1}{c^2}$ 
\begin{equation}\label{CauProbKuznonvisc}
\left\lbrace
\begin{array}{l}
 \partial^2_t u -c^2 \Delta u=\varepsilon\partial_t\left( (\nabla u)^2+\frac{\alpha}{2}(\partial_t u)^2\right),\\
u(0)=u_0,\;\;u_t(0)=u_1.
\end{array}
\right.
\end{equation}
 Thanks to Theorem~1.1 in Ref.~\cite{Perso}, we have the following local well posedness result:
\begin{theorem}\textbf{(Local well posedness for the inviscid Kuznetsov equation)}\cite{Perso}\label{ThMainWPnu0}
  Let $\nu=0$, $n\in \mathbb{N}^*$ and $s>\frac{n}{2}+1$. For all $u_0\in H^{s+1}(\mathbb{R}^n)$ and $u_1\in H^s(\mathbb{R}^n)$ such that 
  $$\Vert u_1\Vert_{L^{\infty}(\mathbb{R}^n)}< \frac{1}{2 \alpha\varepsilon},\hbox{ } \Vert u_0\Vert_{L^{\infty}(\mathbb{R}^n)}<M_1\hbox{ and }\Vert \nabla u_0\Vert_{L^{\infty}(\mathbb{R}^n)}<M_2,$$ with $M_1$ and $M_2$ in $\mathbb{R}^*_+$, the following results hold:
 \begin{enumerate}
  \item There exists $T^*>0$, finite or not, such that there exists a unique solution $u$ of the inviscid Kuznetsov system~(\ref{CauProbKuznonvisc}) with the following regularity 
\begin{align}
&u\in C^r([0,T^*[;H^{s+1-r}(\mathbb{R}^n))\;\;\text{for}\;\;0\leq r\leq s,\label{regkuznu0}\\
&\forall t\in [0,T^*[,\;\;\Vert u_t(t)\Vert_{L^{\infty}(\mathbb{R}^n)}< \frac{1}{2 \alpha\varepsilon},\;\;\Vert u\Vert_{L^{\infty}(\mathbb{R}^n)}<M_1,\;\;\Vert \nabla u\Vert_{L^{\infty}(\mathbb{R}^n)}<M_2.\label{conshyper}
\end{align}
\item The map $(u_0,u_1)\mapsto (u(t,.),\partial_t u(t,.))$ is continuous in the topology of $H^{s+1}\times H^s$ uniformly in $t\in[0,T^*[$.
                        \end{enumerate}
\end{theorem}
Ref.~\cite{Perso} allows us to give a result on the lower bound of the lifespan~$T_\varepsilon$ of the Kuznetsov equation. The method is similar to the case of the Euler system~(\ref{Euli1})--(\ref{Euli2}). It is based on the using of a group of linear transformations preserving the equation $u_{tt}-\Delta u=0$, initially proposed by John\cite{John}. We formulate the lifespan and blow-up time results for the inviscid Kuznetsov equation in the following theorem:
\begin{theorem}\label{ThTimeExistNu0}
\begin{enumerate}
\item \cite{Perso} Let $m\in\mathbb{N}$, $m\geq\left[\frac{n}{2}+2\right]$ . For $u_0\in H^{m+1}(\mathbb{R}^n)$ and $u_1\in H^m(\mathbb{R}^n)$ such that the results of Theorem~\ref{ThMainWPnu0} hold for $s=m$, let $u_0$ and $u_1$ be also small enough in the sense of an energy defined in Point~3 of Theorem~1.1 in Ref.~\cite{Perso}.
Then there exists a generic constant $C>0$ independent on $\varepsilon$ such that
$T_\varepsilon\geq \frac{C}{\varepsilon}.$

\item \cite{Perso} Let $m\in\mathbb{N}$, $m\geq n+2$ if $n$ is even and $m\geq n+1$ if $n$ is odd. For $u_0\in H^{m+1}(\mathbb{R}^n)$ and $u_1\in H^m(\mathbb{R}^n)$ such that the results of Theorem~\ref{ThMainWPnu0} hold for $s=m$, let $u_0$ and $u_1$ be also small enough in the sense of a generalized energy defined in Theorem 3.3 in Ref.~\cite{Perso}. 
Then there exists a generic constant $C>0$ independent on $\varepsilon$ such that
$$T_\varepsilon\geq \frac{C}{\varepsilon^2} \hbox{ for }n=2,\;T_\varepsilon\geq \exp\left(\frac{C}{\varepsilon}\right)-1\hbox{ for }n=3\hbox{ and }T_\varepsilon= +\infty\hbox{ for }n\geq4.$$
\item \cite{Alinhac} In dimension $n=2$ and $3$, there exist functions $u_0\in \mathcal{D}(\mathbb{R}^n)$ and $u_1\in \mathcal{D}(\mathbb{R}^n)$ such that the solution $u$ of the Cauchy problem for the inviscid Kuznetsov equation~(\ref{CauProbKuznonvisc}) has a geometric blow-up for the time of order $T_\varepsilon=O\left(\frac{1}{\varepsilon^2}\right)$ and $T_\varepsilon=O\left(\exp\left(\frac{1}{\varepsilon}\right)\right)$  respectively.
\end{enumerate}
\end{theorem}
\begin{remark}
 In $\mathbb{R}^2$ and $\mathbb{R}^3$ we see that the lifespan of the inviscid Kuznetsov equation corresponds to the blow-up time estimation for the compressible isentropic Euler system in Theorems~\ref{ThminTepsSideris} and~\ref{ThWpEulaxisphesym}, a result in accordance with the fact that the inviscid Kuznetsov equation is an approximation of  the Euler system.We also notice that in the two cases (for the Euler system and the Kuznetsov equation) having a longer existence time requires more regularity on the initial data.
 \end{remark}

\begin{theorem}\label{ThExistUUbarUEulK}
Let $n=2$ or $3$. If the initial data $u_0\in H^4(\mathbb{R}^n)$ and $u_1\in H^3(\mathbb{R}^n)$ for the Cauchy problem for the inviscid Kuznetsov equation~(\ref{CauProbKuznonvisc}) satisfy 
$$\Vert u_0\Vert_{H^4(\mathbb{R}^n)}+\Vert u_1\Vert_{H^3(\mathbb{R}^n)}\leq l$$ 
with $l$ small enough, there exists $T^*_\varepsilon>0$ and $C>0$, independent on $\varepsilon$, satisfying
 $$ T^*_{\varepsilon} \geq\frac{C}{\varepsilon} $$
  such that there exist local in time solutions
 $$\overline{\mathbf{U}}_\eps=(\overline{\rho}_{\varepsilon},\; \overline{\rho}_{\varepsilon} \overline{ \mathbf{v}}_{\varepsilon})^t  \hbox{ and } \mathbf{U}_\eps=(\rho_{\varepsilon},\;\rho_{\varepsilon}\mathbf{v}_{\varepsilon})^t \hbox{ on }[0,T^*_{\varepsilon}[$$
 of   the approximate Euler system given by~(\ref{NSmatrA}) and  of the exact Euler system given by~(\ref{NSmatr}) with $\nu=0$, both considered with the state law~(\ref{press}) and with the same initial data~(\ref{bban}). In addition, the solutions have the same regularity corresponding to
\begin{equation}\label{regUepsEul}
\mathbf{U}_{\varepsilon}-(\rho_0,0)^t\in \bigcap_{k=0}^3 C^k([0,T_{\varepsilon}^*[; [H^{3-k}(\mathbb{R}^n)]^{n+1}) .
\end{equation}  
Here $\bar{\rho}_\eps|_{t=0}$ and $\bar{\mathbf{v}}_\eps|_{t=0}$ are constructed as the functions of the initial data for the Kuznetsov equation $u_0$ and $u_1$ by formulas~(\ref{InCondrob})--(\ref{InCondub}) according to~(\ref{rhoKuz})--(\ref{uKuz}) and~(\ref{rho1K})--(\ref{rho2K}) taken with $\nu=0$.

\end{theorem}
\begin{proof}
Taking $u_0\in H^4(\mathbb{R}^n)$ and $u_1\in H^3(\mathbb{R}^n)$ with $\Vert u_0\Vert_{H^4(\mathbb{R}^n)} +\Vert u_1\Vert_{H^3(\mathbb{R}^n)}\leq l $ and $l$ small enough, the Cauchy problem for the inviscid Kuznetsov equation~(\ref{CauProbKuznonvisc}) is locally well-posed according to Theorem~\ref{ThTimeExistNu0}. 
Moreover the solution $u$ belongs to $\bigcap_{k=0}^4 C^k([0,T_{\varepsilon,1}[; H^{4-k}(\mathbb{R}^n)) $ with $T_{\varepsilon,1}\geq \frac{C_1}{\varepsilon}$ and $C_1>0$ independent of $\varepsilon$.
\\As $u_0\in H^4(\mathbb{R}^n)$ and $u_1\in H^3(\mathbb{R}^n)$, it  ensures that 
$$\rho_\eps-\rho_0|_{t=0}\in H^3(\mathbb{R}^n)\hbox{ and }\mathbf{v}_\eps|_{t=0}\in [H^3(\mathbb{R}^n)]^3.$$
Therefore $\rho_\eps\vert_{t=0}>0$ if $u_0$ and $u_1$ small enough.

By Theorem~\ref{ThminTepsSideris} it is sufficient to have a local solution $\mathbf{U}_{\varepsilon}$ on $[0,T_{\varepsilon,2}[$ of the exact Euler system (see~(\ref{NSmatr}) with $\nu=0$) verifying~(\ref{regUepsEul}) with $T_\varepsilon^*$ corresponding to $T_{\varepsilon,2}$,
 $T_{\varepsilon,2}\geq \frac{C_2}{\varepsilon}$ with $C_2>0$ independent on $\varepsilon$. 

Now we consider $T^*_{\varepsilon}=\min(T_{\varepsilon,1},T_{\varepsilon,2})$, and we have $T^*_{\varepsilon}\geq\frac{C}{\varepsilon}$ with $C>0$ independent on $\varepsilon$.
As $\overline{\rho}_{\varepsilon}$ and $\overline{ \mathbf{v}}_{\varepsilon}$ are defined by \textit{ansatz} ~(\ref{rhoKuz})-(\ref{uKuz}) with  $\rho_1$ and $\rho_2$ given in Eqs.~(\ref{rho1K})--(\ref{rho2K}),
the regularity of $u$ implies for $\overline{\mathbf{U}}_{\varepsilon} $ at least the same regularity as given in~(\ref{regUepsEul}). 
To find it we use the Sobolev embedding~(\ref{Sobolalg}) for the multiplication.
\end{proof}

Knowing the existence results for the two problems, we validate the approximation of $\mathbf{U}_{\varepsilon}$ by the solution of the Kuznetsov equation, $i.e.$ by $\overline{\mathbf{U}}_{\varepsilon}$, following Ref.~\cite{Roz3}.
\begin{theorem}\label{ThapproxEulKuz}\textbf{(Approximation of the Euler system by the Kuznetsov equation)}
Let $n=2$ or $3$ and $u_0\in H^4(\mathbb{R}^n)$, $u_1\in H^3(\mathbb{R}^n)$ be the initial data for the Kuznetsov equation and $\mathbf{U}_{\varepsilon}(0)=\overline{\mathbf{U}}_{\varepsilon}(0)$ for the Euler energy respectively. 
For 
$$\Vert u_0\Vert_{H^4(\mathbb{R}^n)} +\Vert u_1\Vert_{H^3(\mathbb{R}^n)}\leq l $$
 with $l$ small enough, there is the local existence of $\mathbf{U}_{\varepsilon}$ and $\overline{\mathbf{U}}_{\varepsilon}$ for $t\in[0,T^*_{\varepsilon}[$ with~$T^*_{\varepsilon}$ given by Theorem~\ref{ThExistUUbarUEulK} and the same regularity~(\ref{regUepsEul}). 
Moreover there exist $C>0$ and $K>0$ independent on $\varepsilon$ and on the time $t$, such that
\begin{equation}\label{inegapproxeulkuz}
\forall t\leq\frac{C}{\varepsilon} \;\;\;\;\;\;\;\;\Vert (\mathbf{U}_{\varepsilon}-\overline{\mathbf{U}}_{\varepsilon})(t)\Vert^2_{L^2(\mathbb{R}^3)}\leq  Kt \varepsilon^3 e^{K\varepsilon t}\leq 4 \varepsilon^2.
\end{equation}
\end{theorem}
\begin{proof}
The local existence of $\mathbf{U}_{\varepsilon}$ and $\overline{\mathbf{U}}_{\varepsilon}$ comes from Theorem~\ref{ThExistUUbarUEulK}.

We make use of the convex entropy as in Ref.~\cite{Dafermos} for the isentropic Euler equation and the rest follows exactly as in the proof of Theorem~\ref{ThapproxNSKuz} except that $\nu=0$.

We finish the proof with the remark on the minimal regularity of the initial data for the Kuznetsov equation such that the approximation is possible, $i.e.$ the remainder terms $R_1^{NS-Kuz}$ and $\mathbf{R}_2^{NS-Kuz}$ keep bounded for a finite time interval. Indeed, if $u_0\in H^{s+2}(\mathbb{R}^n)$ and $u_1\in H^{s+1}(\mathbb{R}^n)$ with $s>\frac{n}{2}$ then $u\in C([0,T_{\varepsilon}^*[;H^{s+2}(\mathbb{R}^n))$ and
\begin{align*}
u_t\in  C([0,T_{\varepsilon}^*[;H^{s+1}(\mathbb{R}^n)),\quad
u_{tt}\in C([0,T_{\varepsilon}^*[;H^{s}(\mathbb{R}^n)).
\end{align*}
Since $\overline{\rho}_{\varepsilon}$ is defined by~(\ref{rhoKuz}) with~(\ref{rho1K})--(\ref{rho2K}) and $\overline{\textbf{v}}_{\varepsilon}$ by~(\ref{uKuz}), with $\nu=0$, respectively,  we exactly find the regularity
\begin{align*}
\overline{\rho}_{\varepsilon}\in C([0,T_{\varepsilon}^*[;H^{s+1}(\mathbb{R}^n))\cap C^1([0,T_{\varepsilon}^*[;H^{s}(\mathbb{R}^n)),\\
\overline{\textbf{v}}_{\varepsilon}\in C([0,T_{\varepsilon}^*[;H^{s+1}(\mathbb{R}^n))\cap C^1([0,T_{\varepsilon}^*[;H^{s}(\mathbb{R}^n)).
\end{align*} 
Thus by the regularity of the left-hand side part for the approximated Navier-Stokes system~(\ref{NSAp1})--(\ref{NSAp2}) we obtain the desired regularity for the right-hand side.
\end{proof}
\begin{theorem}\textbf{(Approximation of the Euler system by the NPE equation)}
Let $n=2$ or $3$.
 There exists  a constant $k>0$ such that if the initial datum 
 $\xi_0\in H^5(\mathbb{T}_{z}\times\mathbb{R}^{2})$ for the Cauchy problem for the NPE equation~(\ref{npecau}) with $\nu=0$ 
 is sufficiently small
 $$\Vert \xi_0\Vert_{H^5(\mathbb{T}_{z}\times\mathbb{R}^{n-1})}<k\varepsilon,$$
and has a zero mean value then there exist local in time solutions $\overline{\mathbf{U}}_\eps$ of the approximate Euler system~(\ref{NSmatrA})  and $\mathbf{U}_\eps$ of the exact  Euler system~(\ref{NSmatr}) with $\nu=0$ respectively, with the same regularity corresponding to~(\ref{EqregRhoNSper2}) and~(\ref{EqregvNSper}) on $[0,T_{\varepsilon}^*[$ instead of $[0,+\infty[$ and a zero mean value in the $x_1$-direction, both considered with the state law~(\ref{press}) and with the same initial data (\ref{bbannpe})
where $\bar{\rho}_\eps|_{t=0}$ and $\bar{\mathbf{v}}_\eps|_{t=0}$ are constructed as the functions of the initial datum for NPE equation $\xi_0$ according to formulas~(\ref{vNPE})--(\ref{P2NPE}) with $\nu=0$.
Moreover there exists $C>0$ independent of $\varepsilon$ such that $T_{\varepsilon}^*> \frac{C}{\varepsilon}$ and for $t\leq \frac{C}{\varepsilon}$ we have inequality (\ref{inegapproxeulkuz}) on $\mathbb{T}_{x_1}\times\mathbb{R}^{n-1}$.

\end{theorem}
\begin{proof}
The work of Dafermos in Ref.~\cite{Dafermos} can always be applied on $\mathbb{T}_{x_1}\times\mathbb{R}^{n-1}$ for $n=2$ or $3$ instead of $\mathbb{R}^n$ so we have an equivalent of Theorem~\ref{ThWPEulDafermos} and we also have the same equivalent of Theorem~\ref{ThminTepsSideris}. 
This is due to the fact that the energy estimate in the articles of Sideris\cite{Sideris2,Sideris3,Sideris1,Sideris2d} are always true on $\mathbb{T}_{x_1}\times \mathbb{R}$ and $\mathbb{T}_{x_1}\times \mathbb{R}^2$. In all this cases we must also suppose that we have a mean value equal to zero in the direction $x_1$. 
As by Theorem~\ref{wpglopkzknpe} the NPE equation is locally well posed on $[0,T_{\varepsilon}[$ with $ T_{\varepsilon}\geq \frac{C}{\varepsilon}$ if $\Vert \xi_0\Vert_{H^5(\mathbb{T}_{z}\times\mathbb{R}^{n-1})}<k\varepsilon,$ we have an equivalent of Theorems~\ref{ThExistUUbarUEulK} and~\ref{ThapproxEulKuz} for the exact compressible isentropic Euler system and  its approximation by the NPE equation on $\mathbb{T}_{x_1}\times\mathbb{R}^{n-1}$ for $n=2$ or $3$ as $ \xi_0\in H^5(\mathbb{T}_{z}\times\mathbb{R}^{n-1}) $ also implies $\bar{\rho}_\eps|_{t=0}$ and $\bar{\mathbf{v}}_\eps|_{t=0}$ in $H^3(\mathbb{T}_{x_1}\times\mathbb{R}^{n-1}) $.

It is also easy to see using the previous arguments that the minimum regularity of the initial data (see Table~\ref{TABLE}) to have the remainder terms 
$$
R_1^{NS-NPE}\hbox{ and }\textbf{R}_2^{NS-NPE}\in  C([0,T_{\varepsilon}^*[;L^2(\mathbb{T}_{x_1}\times \mathbb{R}^{n-1}))
$$
corresponds to
$\xi_0\in H^s(\mathbb{T}_{x_1}\times \mathbb{R}^{n-1})$ with $s \geq 4$ since then for $0\leq k \leq 2$
$$\xi(\tau,z,y)\in C^k([0,T_{\varepsilon}^*[\rbrace;H^{s-2k}(\mathbb{T}_{z}\times \mathbb{R}^{n-1})),$$
which finally implies with formulas~(\ref{rhoNPE}),~(\ref{vNPE}),~(\ref{P1NPE}) and~(\ref{P2NPE}) with $\nu=0$ that
\begin{align*}
\overline{\rho}_{\varepsilon}(t,x_1,x')\in C([0,T_{\varepsilon}^*[;H^2(\mathbb{T}_{x_1}\times \mathbb{R}^{n-1}))\cap C^1([0,T_{\varepsilon}^*[;L^2(\mathbb{T}_{x_1}\times \mathbb{R}^{n-1})),\\
\overline{\textbf{v}}_{\varepsilon}(t,x_1,x')\in C([0,T_{\varepsilon}^*[;H^3(\mathbb{T}_{x_1}\times \mathbb{R}^{n-1}))\cap C^1([0,T_{\varepsilon}^*[;H^1(\mathbb{T}_{x_1}\times \mathbb{R}^{n-1})).
\end{align*}

\end{proof} 
\begin{remark}
 If we allow the Euler system to have not the classical, but an admissible weak solution with the bounded energy (see Definition~\ref{DefAdm} and take $\nu=0$) taking the initial data in a small on $\eps$ $L^2$-neighborhood of $\overline{\mathbf{U}}_{\varepsilon}(0)$, then we also formally have estimate~(\ref{validaproxintro}).
But, thanks to Ref.~\cite{LUO-2016} it is known that the Euler system can provide infinitely many admissible weak solutions, and thus there are no  sense to approximate them.
\end{remark}
 For the approximation by the KZK equation the inviscid case has already been studied in Ref.~\cite{Roz3}. The key point is that we must restrict our spacial domain to a cone in order to take into account the fact that the KZK equation is only locally well posed.
 \begin{theorem}\label{ThAproEulKZK}
 \cite{Roz3}
 Suppose that there exists the solution $I$ of the KZK Cauchy problem~(\ref{NPEcau2}) with $I_0\in H^s(\mathbb{T}_{\tau}\times \mathbb{R}^{n-1})$ for $s>\max\lbrace 10, \left[\frac{n}{2}\right]+1\rbrace,$ and $\nu=0$ such that $I(\tau,z,y)$ is $L-$periodic with respect to $\tau$ and defined for $\vert z\vert \leq R$ and $y\in \mathbb{R}_y^{n-1}$. Also  we assume
 $$z\mapsto I(\tau,z,y)\in C(]-R,R[;H^s(\mathbb{T}_{\tau}\times \mathbb{R}_y^{n-1}))\cap C^1(]-R,R[;H^{s-2}(\mathbb{T}_{\tau}\times \mathbb{R}_y^{n-1}))$$
 (the uniqueness and the existence of such a solution is proved by Theorem~\ref{wpglopkzknpe}).
 
 Let $\overline{\mathbf{U}}_{\varepsilon}=(\overline{\rho}_{\varepsilon},\overline{\rho}_{\varepsilon}\overline{\textbf{v}}_{\varepsilon})^t$ be the approximate solution of the isentropic Euler system~(\ref{massaproxKZK})--(\ref{momaproxKZK}) with $\nu=0$ deduced from a solution of the KZK equation. Then the function $\overline{\mathbf{U}}_{\varepsilon}(t,x_1,x')$ is defined in 
 $$\mathbb{T}_t\times(\Omega_{\varepsilon}=\lbrace x_1\vert x_1<\frac{R}{\varepsilon}-ct\rbrace\times \mathbb{R}_{x'}^{n-1})$$
 and is smooth enough according to the regularity of $I$.
 
 Let us now consider the solution $\mathbf{U}_{\varepsilon}$ of the Euler System~(\ref{NSmatr}) with $\nu=0$ in a cone
 $$C(t)=\lbrace 0<s<t\rbrace \times Q_{\varepsilon}(s)=\lbrace x=(x_1,x'):\vert x_1\vert\leq \frac{R}{\varepsilon}-Ms, M\geq c, x'\in \mathbb{R}^{n-1}\rbrace$$
 with the initial data 
 $$(\rho_{\varepsilon}-\overline{\rho}_{\varepsilon})\vert_{t=0}=0,\;\;\;(\mathbf{v}_{\varepsilon}- \overline{ \mathbf{v}}_{\varepsilon})\vert_{t=0}=0.$$
 Consequently, (see Ref.~\cite{Dafermos} p. 62) there exists $T_0$ such that for the time interval $0\leq t\leq \frac{T_0}{\varepsilon}$ there exists the classical solution $\mathbf{U}_{\varepsilon}=(\rho_{\varepsilon},\rho_{\varepsilon}\mathbf{v}_{\varepsilon})$ of the Euler system~(\ref{NSmatr}) with $\nu=0$ in a cone
 $$C(T)=\lbrace 0<t<T \vert T<\frac{T_0}{\varepsilon}\rbrace\times Q_{\varepsilon}(t)$$
 with
 $$\Vert \nabla \mathbf{U}_{\varepsilon} \Vert_{L^{\infty}([0,\frac{T_0}{\varepsilon}[;H^{s-1}(Q_{\varepsilon}))}<\varepsilon C \hbox{ for }s>\left[\frac{n}{2}\right]+1.$$
 
 Moreover, there exists $K>0$ such that for any $\varepsilon$ small enough, the solutions $\mathbf{U}_{\varepsilon}$ and $\overline{\mathbf{U}}_{\varepsilon}$ which where determined as above in cone $C(T)$ with the same initial data, satisfy the estimate for $0<t<\frac{T_0}{\varepsilon} $
 $$\Vert (\mathbf{U}_{\varepsilon}-\overline{\mathbf{U}}_{\varepsilon})(t)\Vert_{L^2(Q_{\varepsilon}(t))}^2\le c_0^2 \varepsilon^3 t e^{2K\varepsilon t}\leq 4 \varepsilon^2$$
 with $c_0^2>0$.
  
 \end{theorem}

\begin{remark}
The regularity of $I_0\in H^s(\mathbb{T}_t\times\mathbb{R}^{n-1})$ with $s>8$ (see Table~\ref{TABLE}) is minimal to ensure  that $R_1^{NS-KZK}$ and $\textbf{R}_2^{NS-KZK}$, see~\ref{Apen1}, are in $C([0,\frac{T_0}{\varepsilon}[;L^2(Q_\varepsilon))$. 
$$I(\tau,z,y)\in C^k(]-R,R[;H^{s-2k}(\mathbb{T}_{\tau}\times \mathbb{R}^{n-1})).$$
Let us denote $\Omega=\mathbb{T}_{\tau}\times \mathbb{R}^{n-1}$.
Given the equations for $\overline{\rho}_{\varepsilon}$ by~(\ref{rhokzk}) with~(\ref{Ikzk}) and~(\ref{Jkzk}) and for $\overline{\textbf{v}}_{\varepsilon}$  by~(\ref{vkzk}) with $\nu=0$ respectively,  we have for $0\leq k\leq 2$
\begin{align*}
&\partial^k_z\overline{\rho}_{\varepsilon}(\tau,z,y)\in  C(]-R,R[;H^{s-2k}(\Omega)),\;
\partial^k_z\overline{\textbf{v}}_{\varepsilon}(\tau,z,y) \in  C(]-R,R[;H^{s-2-2k}(\Omega)),
\end{align*}
but we can also say that 
\begin{align*}
&\partial^k_z\overline{\rho}_{\varepsilon}(\tau,z,y)\in  L^2(]-R,R[;H^{s-2k}(\Omega)),\;
\partial^k_z\overline{\textbf{v}}_{\varepsilon}(\tau,z,y) \in  L^2(]-R,R[;H^{s-2-2k}(\Omega)).
\end{align*}
This implies for $0\leq k\leq 2$ (as $s>8$)  that $s-2-2k>2$ and
\begin{align*}
\partial^k_z\overline{\rho}_{\varepsilon}(\tau,z,y)\in &C(\mathbb{T}_{\tau}; L^2(\lbrace x_1\vert x_1<\frac{R}{\varepsilon}-ct\rbrace;H^{s-2k}( \mathbb{R}^{n-1}))),\\
\partial^k_z\overline{\textbf{v}}_{\varepsilon}(\tau,z,y) \in &C(\mathbb{T}_\tau; L^2(\lbrace x_1\vert x_1<\frac{R}{\varepsilon}-ct\rbrace;H^{s-2-2k}( \mathbb{R}^{n-1}))).
\end{align*}
Hence we find 
\begin{align*}
\overline{\rho}_{\varepsilon}(t,x_1,x'),\; \overline{\textbf{v}}_{\varepsilon}(t,x_1,x')\in & C([0,\frac{T_0}{\varepsilon}[;H^2(Q_{\varepsilon})).\end{align*}
As in addition for $0\leq k\leq 1$, considering $\overline{\rho}_{\varepsilon}$ and $\overline{\mathbf{v}}_{\varepsilon}$ as functions of $(\tau,z,y)$,
\begin{align*}
&\partial^k_z\partial_{\tau}\overline{\rho}_{\varepsilon}\in  C(]-R,R[;H^{s-1-2k}(\Omega)),\;
\partial^k_z\partial_{\tau}\overline{\mathbf{v}}_{\varepsilon} \in  C(]-R,R[;H^{s-3-2k}(\Omega)),
\end{align*}
 we deduce in the same way that
\begin{align*}
\partial_t \overline{\rho}_{\varepsilon}(t,x_1,x'),\; \partial_t \overline{\textbf{v}}_{\varepsilon}(t,x_1,x')\in & C([0,\frac{T_0}{\varepsilon}[;H^1(Q_{\varepsilon})).
\end{align*}
These regularities of  $\overline{\rho}_{\varepsilon}$ and $\overline{\mathbf{v}}_{\varepsilon}$ viewed as functions of $(t,x_1,x')$ allow to have all left-hand terms in the approximated Euler system (\ref{massaproxKZK})--(\ref{momaproxKZK}) with $\nu=0$ of the regularity $C([0,\frac{T_0}{\varepsilon}[;L^2(Q_\varepsilon))$ and the remainder terms in the right-hand side inherit it.
\end{remark}

\section{The Kuznetsov equation and the KZK equation.}
\label{secKuzKZK}
\subsection{Derivation of the KZK equation from the Kuznetsov equation.}\label{secderKuzKZK}
If  the velocity potential is given\cite{Kuznetsov} by Eq.~(\ref{paraxpot}), we directly obtain from the Kuznetsov equation~(\ref{KuzEqA}) with the paraxial change of variable~(\ref{chvarkzk}) that
\begin{align}
& \partial^2_t u -c^2 \Delta u-\varepsilon\partial_t\left( (\nabla u)^2+\frac{\gamma-1}{2c^2}(\partial_t u)^2+\frac{\nu}{\rho_0}\Delta u\right)\nonumber\\
&=\varepsilon\left[2c\partial_{\tau z}^2\Phi-\frac{\gamma+1}{2c^2}\partial_{\tau}(\partial_{\tau} \Phi)^2-\frac{\nu}{\rho_0c^2}\partial^3_{\tau}\Phi- c^2 \Delta_y \Phi\right]+\varepsilon^2 R^{Kuz-KZK}\label{KZKpoten}\end{align}
with 
\begin{align}
\varepsilon^2 R^{Kuz-KZK}= &\varepsilon^2 \left(-c^2 \partial^2_z \Phi+\frac{2}{c} \partial_{\tau}(\partial_{\tau}\Phi \partial_z \Phi)-\partial_{\tau}(\nabla_y\Phi)^2+\frac{2\nu}{c\rho_0}\partial^2_{\tau}\partial_z \Phi-\frac{\nu}{\rho_0}\partial_{\tau} \Delta_y \Phi\right)\nonumber\\
&+\varepsilon^3\left(-\partial_{\tau}(\partial_z\Phi)^2-\frac{\nu}{\rho_0}\partial_{\tau}\partial^2_z\Phi\right).\label{remkuzkzk}
\end{align}

Therefore,
we find that the right-hand side $\epsilon$-order terms in Eq.~(\ref{KZKpoten}) is exactly the KZK equation~(\ref{KZKI}). Due to its well posedness domain, to validate the approximation between the solutions of the KZK and the Kuznetsov equations, we need to study the well posedness of the Kuznetsov equation on the half space with boundary conditions coming from the initial condition for the KZK equation.

\subsection{Well posedness of the models.}\label{secWPresults}

\subsubsection{Well posedness of the Kuznetsov equation in the half space with periodic boundary conditions.}\label{WPKuzhalf1}
Let us consider the following periodic in time problem for the Kuznetsov equation in the half space $\mathbb{R}_+\times\mathbb{R}^{n-1}$ with periodic in time Dirichlet boundary conditions:
\begin{equation}\label{kuzper}
\left\lbrace
\begin{array}{l}
u_{tt}-c^2 \Delta u-\nu \varepsilon \Delta u_t=\alpha \varepsilon u_t\; u_{tt}+\beta \varepsilon \nabla u\; \nabla u_t\;\;\;\;\hbox{ on }\; \mathbb{T}_t\times \mathbb{R}_+\times\mathbb{R}^{n-1},\\
u\vert_{x_1=0}=g\;\;\;\;\hbox{ on }\; \mathbb{T}_t\times \mathbb{R}^{n-1},
\end{array}\right.
\end{equation}
where $g$ is a $L$-periodic in time and of mean value zero function. 
For this we use Ref.~\cite{Celik} and thus we  directly obtain the following result of maximal regularity:
\begin{theorem}~\cite{Celik}\label{ThCelik}
Let $n=3$, $ \Omega=\mathbb{R}_+\times \mathbb{R}^{n-1}$ and $p\in ]1,+\infty[$. Then there exits a unique solution $u\in W^{2}_p(\mathbb{T}_t;L^p(\Omega))\cap W^{1}_p(\mathbb{T}_t;W^2_p(\Omega))$ with the mean value zero
\begin{equation}\label{EqMV0U}
 \forall x\in \Omega \quad \int_{\mathbb{T}_t}u(s,x)\;ds=0
\end{equation}
 of the following system
\begin{equation}\label{kuzlininhomper}
\left\lbrace\begin{array}{l}
u_{tt}-c^2 \Delta u - \nu \varepsilon \Delta u_t=f \;\;\;\hbox{ on }\;\mathbb{T}_t\times \Omega,\\
u =g\;\;\;\hbox{ on }\;\mathbb{T}_t\times \partial\Omega\
\end{array}\right.
\end{equation}
if and only if the functions
\begin{equation}\label{EqfTh51}
 f\in L^p(\mathbb{T}_t;L^p(\Omega)) \hbox{ and } g\in W^{2-\frac{1}{2p}}_p(\mathbb{T}_t;L^p(\partial\Omega))\cap W^1_p(\mathbb{T}_t;W^{2-\frac{1}{p}}_p(\partial\Omega))
\end{equation}
and are  of mean value zero:
\begin{equation}\label{EqMV0}
 \forall x\in \Omega\quad \int_{\mathbb{T}_t}f(l,x)\;dl=0 \hbox{ and }\forall x'\in \partial\Omega \quad \int_{\mathbb{T}_t}g(l,x')\;dl=0.
\end{equation}
Moreover, we have the following stability estimate
\begin{align*}
\Vert u\Vert_{W^{2}_p(\mathbb{T}_t;L^p(\Omega))\cap W^{1}_p(\mathbb{T}_t;W^2_p(\Omega))}\leq C & \left(\Vert f\Vert_{L^p(\mathbb{T}_t;L^p(\Omega))}\right.\\
&+\left.\Vert g\Vert_{ W^{2-\frac{1}{2p}}_p(\mathbb{T}_t;L^p(\partial\Omega))\cap W^1_p(\mathbb{T}_t;W^{2-\frac{1}{p}}_p(\partial\Omega)) } \right).
\end{align*}
\end{theorem}
\begin{proof}
On one hand,
if $f$ and $g$ satisfy~(\ref{EqfTh51})--(\ref{EqMV0}), the necessity of the conditions is shown in Ref.~\cite{Celik}. 
On the other hand, the conditions~(\ref{EqfTh51})--(\ref{EqMV0}) are sufficient by a direct application of the trace theorems recalled in Ref.~\cite{Celik} and proved in Ref.~\cite{Denk} for example.
\end{proof}
The results of Ref.~\cite{Celik} allow to see that Theorem~\ref{ThCelik} does not depend on~$n$, moreover if we look at the case $p=2$ the linearity of the operator $\partial^2_t-c^2\Delta-\nu \Delta\partial_t$ from~(\ref{kuzlininhomper}) implies that we can work with $H^s(\Omega)$ instead of $L^2(\Omega)$ : 
\begin{lemma}\label{thmLinPer}
Let $n\in \mathbb{N}^*$, $ \Omega=\mathbb{R}_+\times \mathbb{R}^{n-1}$, $s\geq 0$ then 
there exits a unique solution 
\begin{equation}\label{Timepersolspac}
u\in X=\left\lbrace u\in H^2(\mathbb{T}_t;H^s(\Omega))\cap H^1(\mathbb{T}_t;H^{s+2}(\Omega))\vert \forall x\in \Omega \;\int_{\mathbb{T}_t}u(s,x)\;ds=0\right\rbrace
\end{equation}
with the mean value zero (see Eq.~(\ref{EqMV0U})) of  system~(\ref{kuzlininhomper})
if and only if
\begin{equation}\label{timeperboundreg}
f\in L^2(\mathbb{T}_t;H^s(\Omega))\hbox{ and }g\in \mathbb{F}_{\mathbb{T}}= H^{\frac{7}{4}}(\mathbb{T}_t;H^s(\partial\Omega))\cap H^1(\mathbb{T}_t;H^{s+\frac{3}{2}}(\partial\Omega))
\end{equation}
both satisfying~(\ref{EqMV0}).

Moreover we have the following stability estimate
$$\Vert u\Vert_{X}\leq C (\Vert f\Vert_{L^2(\mathbb{T}_t;H^s(\Omega))}+\Vert g\Vert_{ \mathbb{F}_{\mathbb{T}}} ).$$
Here $H^2(\mathbb{T}_t;H^s(\Omega))\cap H^1(\mathbb{T}_t;H^{s+2}(\Omega))$ is endowed with its usual norm denoted here and in the sequel by $\Vert.\Vert_X$.
\end{lemma}

To prove the global well-posedness of problem~(\ref{kuzper}) for the Kuznetsov equation we use the following theorem~\cite{Sukhinin}:
\begin{theorem}\label{thSuh}\cite{Sukhinin}
Let $X$ be a Banach space, let $Y$ be a separable
topological vector space, let $L : X \rightarrow Y$ be a linear
continuous operator, let $U$ be the open unit ball in $X$, let ${\rm
P}_{LU}:LX \to [0,\infty [$ be the Minkowski functional of the set
$LU$, and let $\Phi :X \to LX$ be a mapping satisfying the condition
\begin{equation*}
 {\rm P}_{LU} \bigl(\Phi (x) -\Phi (\bar{x})\bigr) \leq
\Theta (r) \left\|x -\bar{x} \right\|\quad \text{for} \quad \left\|x
-x_0 \right\| \leqslant r,\quad \left\|\bar{x} -x_0 \right\| \leq r
\end{equation*} for some $x_0 \in X,$ where $\Theta :[0,\infty [ \to [0,\infty [$ is a monotone
non-decreasing function. Set $b(r) =\max \bigl(1 -\Theta (r),0
\bigr)$ for $r \geq 0$.

 Suppose that \begin{align*}
               &w =\int\limits_0^\infty b(r)\,dr \in ]0,\infty ], \quad r_* =\sup \{ r
\geq 0|\;b(r) >0 \},\\
&w(r) =\int\limits_0^r b(t)dt \quad (r \geq 0) \quad\hbox{and} \quad f(x) =Lx
+\Phi(x) \quad \hbox{for} \quad x \in X.
              \end{align*}
Then for any $r \in
[0,r_*[$ and $ y \in f(x_0) +w(r)LU$, there exists an
 $ x \in x_0 +rU$ such that $f(x) =y$.
\end{theorem}
Now we can use the maximal regularity result for system~(\ref{kuzlininhomper}) with Theorem~\ref{thSuh} and the same method as for the Cauchy problem associated with the Kuznetsov equation used in our previous work~\cite{Perso}. We will just have to use the boundary conditions of problem~(\ref{kuzper}) as the initial condition of the corresponding Cauchy problem in $\mathbb{R}^n$.
\begin{theorem}\label{globwellposKuzper}
Let $\nu>0$, $n\in \mathbb{N}^*$, $\Omega=\mathbb{R}_+\times \mathbb{R}^{n-1} $, $s>\frac{n}{2}$. 
Let 
$
X\hbox{ be defined by }(\ref{Timepersolspac})
$
and the boundary condition 
$
g\in \mathbb{F}_{\mathbb{T}}$ be defined by~(\ref{timeperboundreg}) and in addition, let $g$ be of the mean value zero (see Eq.~(\ref{EqMV0})).

Then there exist $r^*=O(1)$ and $C_1=O(1)$ such that for all $r\in [0,r^*[$, if
$\Vert g\Vert_{\mathbb{F}_{\mathbb{T}}}\leq \frac{\sqrt{\nu \varepsilon}}{C_1}r,$
there exists a unique solution $u\in X$ of the periodic problem~(\ref{kuzper}) satisfying~(\ref{EqMV0U}) and such that $\Vert u\Vert_X\leq 2r$.
\end{theorem}

\begin{proof}
For $g\in\mathbb{F}_{\mathbb{T}}$ defined in~(\ref{timeperboundreg}) and satisfying~(\ref{EqMV0}), let us denote by $u^*\in X$ the unique solution of the linear problem~(\ref{kuzlininhomper}) with $f=0$ and $g\in \mathbb{F}_{\mathbb{T}}$.

In addition, according to Theorem~\ref{thmLinPer}, we take $X$ defined in~(\ref{Timepersolspac}),
 this time for $s>\frac{n}{2}$ (we need it to control the non-linear terms),  and introduce the Banach spaces 
\begin{equation}
 X_0:=\lbrace u\in X|\;\;\;u\vert_{\partial\Omega}=0\hbox{ on } \mathbb{T}_t\times\partial\Omega \rbrace
\end{equation}
and 
$$Y=\left\lbrace f\in L^2(\mathbb{T}_t;H^s(\Omega))\vert\; \forall x\in \Omega \quad \int_{\mathbb{T}_t}f(s,x)\;ds=0\right\rbrace.$$ Then by Lemma~\ref{thmLinPer}, the linear operator 
$$L:X_0\rightarrow Y,\quad  u\in X_0\mapsto\;L(u):=u_{tt}-c^2\Delta u-\nu \varepsilon \Delta u_t\in Y,$$
 is a bi-continuous isomorphism. 
 
 Let us now notice that if $v$ is the unique solution of the non-linear Dirichlet problem   
 \begin{equation}\label{SystkuznV}
\left\lbrace
\begin{array}{lll}
v_{tt}-c^2\Delta v-\nu\varepsilon \Delta v_t=&\alpha \varepsilon (v+u^*)_t(v+u^*)_{tt}&\hbox{ on } \mathbb{T}_t\times\Omega ,\\

&+\beta \varepsilon \nabla (v+u^*).\nabla(v+u^*)_t &  \\
v=0\hbox{ on } \mathbb{T}_t\times\partial\Omega ,& &\end{array}
\right.
\end{equation}
 then $u=v+u^*$ is the unique solution of the periodic problem~(\ref{kuzper}).
 Let us prove the existence of a such $v$, using Theorem~\ref{thSuh}.
 
We suppose that $\Vert u^*\Vert_X\leq r$
and define for $v\in X_0$
$$\Phi(v):=\alpha \varepsilon (v+u^*)_t(v+u^*)_{tt}+\beta \varepsilon \nabla (v+u^*).\nabla(v+u^*)_t.$$

For $w$ and $z$ in $X_0$ such that 
$\Vert w\Vert_X\leq r$ and $\Vert z\Vert_X\leq r$,
 we estimate $\Vert \Phi(w)-\Phi(z)\Vert_Y $.
By applying the triangular inequality we have
\begin{multline*}
\Vert \Phi(w)-\Phi(z)\Vert_Y\leq  \alpha \varepsilon \Big(\Vert u^*_t (w-z)_{tt} \Vert_Y+\Vert (w-z)_t u^*_{tt}\Vert_Y\\
+\Vert  w_t (w-z)_{tt}\Vert_Y+\Vert (w-z)_t z_{tt}\Vert_Y\Big)\\
+\beta \varepsilon\Big( \Vert \nabla u^* \nabla(w-z)_t \Vert_Y+\Vert \nabla (w-z)\nabla u^*_t \Vert_Y\\
+\Vert \nabla w\nabla (w-z)_t \Vert_Y+\Vert \nabla (w-z) \nabla z_t \Vert_Y\Big).
\end{multline*}
 Now, for all $a$ and $b$ in $X$ with $s\ge s_0> \frac{n}{2}$  it holds
 \begin{align*}
 \Vert a_t b_{tt}\Vert_Y\leq & \Vert a_t \Vert_{L^\infty(\mathbb{T}_t\times\Omega)} \Vert b_{tt}\Vert_Y\\
 \leq & C_{H^1(\mathbb{T}_t;H^{s_0}(\Omega))\to L^\infty(\mathbb{T}_t\times\Omega)} \Vert a_t\Vert_{H^1(\mathbb{T}_t;H^{s_0}(\Omega))} \Vert b\Vert_{X}\\
 \leq & C_{H^1(\mathbb{T}_t;H^{s_0}(\Omega))\to L^\infty(\mathbb{T}_t\times\Omega)} \Vert a\Vert_{X} \Vert b\Vert_{X},
\end{align*}
where $C_{H^1(\mathbb{T}_t;H^{s_0}(\Omega))\to L^\infty(\mathbb{T}_t\times\Omega)}$ is the embedding constant of $H^1(\mathbb{T}_t;H^{s_0}(\Omega))$ in $L^\infty(\mathbb{T}_t\times\Omega)$, independent on $s$, but depending only on the dimension $n$.
In the same way, for all $a$ and $b$ in $X$ it holds
$$\Vert \nabla a \nabla  b_t\Vert_Y\leq C_{H^1(\mathbb{T}_t;H^{s_0}(\Omega))\to L^\infty(\mathbb{T}_t\times\Omega)} \Vert a\Vert_{X} \Vert b\Vert_{X}.$$
Taking $a$ and $b$ equal to $u^*$, $w$, $z$ or $w-z$, as $\Vert u^*\Vert_X\leq r$, $\Vert w\Vert_X\leq r$ and $\Vert z\Vert_X\leq r$, we obtain
\begin{align*}
\Vert \Phi(w)-\Phi(z)\Vert_Y\leq %
4 (\alpha+\beta)C_{H^1(\mathbb{T}_t;H^{s_0}(\Omega))\to L^\infty(\mathbb{T}_t\times\Omega)} \varepsilon r \Vert w-z\Vert_X.
\end{align*}
By the fact that $L$ is a bi-continuous isomorphism, there exists a minimal constant $C_\eps=O\left(\frac{1}{\eps \nu} \right)>0$, coming from the inequality $C_0 \eps \nu\|u\|_X^2\le \|f\|_Y\|u\|_X$ for $u$, a solution of the linear problem~(\ref{kuzlininhomper}) with homogeneous boundary data (for a maximal constant $C_0=O(1)>0$)
such that
$$\forall u\in X_0 \quad \Vert u\Vert_X\leq C_\eps \Vert Lu\Vert_Y.$$
Hence, for all $f\in Y$
$$P_{LU_{X_0}}(f)\leq C_\eps P_{U_Y}(f)=C_\eps\Vert f\Vert_Y.$$
Then we find for $w$ and $z$ in $X_0$, such that $\|w\|_X\le r$, $\|z\|_X\le r$, and also with $\|u^*\|_X\le r$, that with $\Theta(r):= 4 C_\eps (\alpha+\beta)C_{H^1(\mathbb{T}_t;H^{s_0}(\Omega))\to L^\infty(\mathbb{T}_t\times\Omega)}\varepsilon r$ it holds
$$P_{LU_{X_0}}(\Phi(w)-\Phi(z))\leq \Theta(r) \Vert w-z\Vert_X.$$
Thus we apply Theorem~\ref{thSuh} with $f(x)=L(x)-\Phi(x)$ and $x_0=0$. Therefore, knowing that $C_\eps=\frac{C_0}{\eps \nu}$, we have, that for all  $r\in[0,r_{*}[$ with 
\begin{equation}\label{Eqret}
 r_{*}=\frac{\nu}{4 C_0 (\alpha+\beta)C_{H^1(\mathbb{T}_t;H^{s_0}(\Omega))\to L^\infty(\mathbb{T}_t\times\Omega)}}=O(1),
\end{equation}
  for all $y\in \Phi(0)+w(r) L U_{X_0}\subset Y$
with $$w(r)= r-2 \frac{C_0}{\nu} C_{H^1(\mathbb{T}_t;H^{s_0}(\Omega))\to L^\infty(\mathbb{T}_t\times\Omega)} (\alpha+\beta) r^2,$$
there exists a unique $v\in 0+r U_{X_0}$ such that $L(v)-\Phi(v)=y$.
But, if we want that $v$ be the solution of the non-linear  problem~(\ref{SystkuznV}), then we need to impose $y=0$ and thus, to ensure that $0\in \Phi(0)+w(r) L U_{X_0}$.
Since $-\frac{1}{w(r)}\Phi(0)$ is an element of $Y$ and $LX_0=Y$, there exists a unique $z\in X_0$ such that
\begin{equation}\label{Eqz}
 L z=-\frac{1}{w(r)}\Phi(0).
\end{equation}
Let us show that $\|z\|_X\le 1$, what will implies that $0\in \Phi(0)+w(r) L U_{X_0}$.
Noticing that
\begin{align*}
\Vert \Phi(0)\Vert_Y & \leq \alpha \varepsilon \Vert v_t v_{tt}\Vert_Y +\beta \varepsilon \Vert \nabla v \nabla v_t\Vert_Y\\
& \leq  (\alpha+\beta) \varepsilon C_{H^1(\mathbb{T}_t;H^{s_0}(\Omega))\to L^\infty(\mathbb{T}_t\times\Omega)}\Vert v\Vert_X^2 \\
& \leq (\alpha+\beta) \varepsilon C_{H^1(\mathbb{T}_t;H^{s_0}(\Omega))\to L^\infty(\mathbb{T}_t\times\Omega)}r^2
\end{align*}
and using~(\ref{Eqz}), we find
\begin{align*}
 & \Vert z\Vert_X \leq C_\eps\Vert L z\Vert_Y=C_\eps\frac{\Vert \Phi(0)\Vert_Y}{w(r)}\\
 &\leq \frac{C_\eps C_{H^1(\mathbb{T}_t;H^{s_0}(\Omega))\to L^\infty(\mathbb{T}_t\times\Omega)} (\alpha+\beta) \varepsilon r}{(1-2 C_\eps C_{H^1(\mathbb{T}_t;H^{s_0}(\Omega))\to L^\infty(\mathbb{T}_t\times\Omega)} (\alpha+\beta)\varepsilon r)}<\frac{1}{2},
\end{align*}
as soon as $r<r^*$. 

Consequently, $z\in U_{X_0}$ and $\Phi(0)+w(r) Lz=0$.
Then we conclude that  for all  $r\in[0,r_{*}[$, if $\|u^*\|_X\le r$, there exists a unique $v\in r U_{X_0}$ such that $L(v)-\Phi(v)=0$, $i.e.$  the solution of the non-linear problem~(\ref{SystkuznV}).
Thanks to the maximal regularity and a priori estimate following from Theorem~\ref{thmLinPer} with $f=0$, 
there exists a constant $C_1=O(\eps^0)>0$, such that
$$\|u^*\|_X\le \frac{C_1}{\sqrt{\nu \eps}}\Vert g\Vert_{\mathbb{F}_{\mathbb{T}}}.$$
Thus, for all  $r\in[0,r_{*}[$ and $\Vert g\Vert_{\mathbb{F}_{\mathbb{T}}}\le \frac{\sqrt{\nu \eps}}{C_1}r$, the function $u=u^*+v\in X$ is the unique solution of the time periodic problem for the Kuznetsov equation and $\Vert u\Vert_X\leq 2 r$.

\end{proof}

\subsubsection{Well posedness of the initial boundary value problem in the half space for the Kuznetsov equation.}\label{WPKuzhalf2}
We work on $\Omega=\mathbb{R}_+\times \mathbb{R}^{n-1}$ and we are going to study the initial boundary value problem for the Kuznetsov equation on this space, \textit{i.e.} the perturbation of an imposed initial condition by a source on the boundary, which will later be determined by the solution of the KZK equation.
\begin{lemma}\label{maxregstrdamphalf}
Let $s\geq 0$, $n\in \mathbb{N}$. There exists a unique solution 
\begin{equation}\label{Halfspaceregsol}
u\in \mathbb{E}:=H^2(\mathbb{R_+};H^s(\Omega))\cap H^1( \mathbb{R_+};H^{s+2}(\Omega))
\end{equation}
of the linear problem
\begin{equation}\label{eqlinhalf}
\left\lbrace
\begin{array}{c}
u_{tt}-c^2 \Delta u -\nu \varepsilon \Delta u_t=f\;\;\;\hbox{ in }\;\mathbb{R_+}\times \Omega,\\
u=g\;\;\;\hbox{ on }\; \mathbb{R_+}\times\partial\Omega,\\
u(0)=u_0, \;\;\;u_t(0)=u_1\;\;\;\hbox{ in }\; \Omega
\end{array}\right.
\end{equation}
if and only if the data satisfy the following conditions
\begin{itemize}
\item $f\in L^2(\mathbb{R}_+;H^s(\Omega)),$
\item for the boundary condition
\begin{equation}\label{Halfspaceregbound}
g\in \mathbb{F}_{\mathbb{R}_+}= H^{7/4}(\mathbb{R}_+;H^s(\partial\Omega))\cap H^1(\mathbb{R}_+;H^{s+3/2}(\partial\Omega));
\end{equation}
\item $u_0\in H^{s+2}(\Omega)$ and $u_1\in H^{s+1}(\Omega)$;
\item $g(0)=u_0$ and $g_t(0)=u_1$ on $\partial\Omega $ in the trace sense.
\end{itemize}
In addition, the solution satisfies the stability estimate
$$\Vert u\Vert_{\mathbb{E}}\leq C (\Vert f\Vert_{L^2(\mathbb{R}_+;H^s(\Omega))}+\Vert g\Vert_{\mathbb{F}_{\mathbb{R}_+}}+\Vert u_0\Vert_{H^{s+2}}+\Vert u_1\Vert_{H^{s+1}}).$$
\end{lemma} 
In order to prove this result we will use the subsequent lemma to remove the inhomogeneity $g$.
\begin{lemma}\label{lemwpheat}
Let $s\geq 0$, $n\in \mathbb{N}$. There exists a unique solution $ w\in \mathbb{E}$ defined by~(\ref{Halfspaceregsol}) of the following linear problem
\begin{equation}\label{heatderiv}
\left\lbrace
\begin{array}{c}
w_{tt} -\nu \varepsilon \Delta w_t=0\;\;\;\hbox{ in }\;\mathbb{R_+}\times \Omega,\\w=g\;\;\;\hbox{ on }\; \mathbb{R_+}\times\partial\Omega,\\
w(0)=0, \;\;\;w_t(0)=0\;\;\;\hbox{ in }\; \Omega
\end{array}\right.
\end{equation}
if and only if the data satisfy the following conditions
\begin{itemize}
\item $g\in \mathbb{F}_{\mathbb{R}_+}$ defined in~(\ref{Halfspaceregbound}),
\item for the compatibility: for all $x\in\partial\Omega$, $g(0)=0$ and $g_t(0)=0$.
\end{itemize}
Moreover, the solution $w$ satisfies the stability estimate
$$\Vert w\Vert_{\mathbb{E}}\leq C \Vert g\Vert_{\mathbb{F}_{\mathbb{R}_+}}.$$
\end{lemma} 

\begin{proof}
First we prove the sufficiency. By assumption~(\ref{Halfspaceregbound}), we have
$$\partial_t g\in H^{3/4}(\mathbb{R}_+;H^s(\partial\Omega))\cap L^2(\mathbb{R}_+;H^{s+3/2}(\partial\Omega)).$$ 
Thanks to $\S$~3 p.~288 in Ref.~\cite{Ladyzhenskaya}, we obtain a unique solution 
$$v\in H^1(\mathbb{R_+};H^s(\Omega))\cap L^2( \mathbb{R_+};H^{s+2}(\Omega))$$
of the parabolic problem
$$v_t-\nu \varepsilon \Delta v=0\;\hbox{ in }\;\mathbb{R}_+\times \Omega,\;v=\partial_t g\;\hbox{ on }\;\mathbb{R}_+\times\partial\Omega,\;v(0)=0\;\hbox{ in }\;\Omega.$$
Next we define for $t\in\mathbb{R}_+$ and $x\in\Omega$ the function
$$w(t,x):=\int_0^t v(l,x)dl. $$ 
We have $w(0)=0$ and $w_t(0)=0$. Moreover, it satisfies
$$w_{tt}-\nu\varepsilon \Delta w_t=0, \quad
w(t)\vert_{\partial\Omega}=\int_0^t g_t(l)\;dl=g(t),$$
as $g(0)=0$. Therefore, $w$ is a solution of problem~(\ref{heatderiv}).The necessity follows from the spatial trace theorem ensuring that the trace operator $Tr_{\partial\Omega}:u  \mapsto u\vert_{\partial\Omega}$, considering as a map 
\begin{align}
&H^1(\mathbb{R_+};H^s(\Omega))\cap L^2( \mathbb{R_+};H^{s+2}(\Omega))  \rightarrow H^{3/4}(\mathbb{R}_+;H^s(\partial\Omega))\cap L^2(\mathbb{R}_+;H^{s+3/2}(\partial\Omega)),\label{Trachalf1}
\end{align}
is bounded and surjective by Lemma~3.5 in Ref.~\cite{Denk}.
For the compatibility condition, thanks to Lemma~11 in Ref.~\cite{Blasio}, we also know that the temporal trace $Tr_{t=0}:g\mapsto~g\vert_{t=0}$, considered as a map 
\begin{align}
&H^{3/4}(\mathbb{R}_+;H^s(\partial\Omega))\cap L^2(\mathbb{R}_+;H^{s+3/2}(\partial\Omega)) \rightarrow H^{s+1/2}(\partial\Omega)
\label{Trachalf2}
\end{align}
is well defined and bounded. Moreover, the spatial trace 
\begin{equation}\label{Trachalf3}
H^{s+1/2}(\Omega)\rightarrow H^s(\partial\Omega)
\end{equation}
is bounded by Theorem~1.5.1.1 from Ref.~\cite{Grisvard}.

To obtain uniqueness, let $w$ be a solution to~(\ref{heatderiv}) with $g=0$. Since $w_t$ solve a heat problem with homogeneous data, we obtain $w_t=0$ and therefore also $w=0$ by the initial condition $w(0)=0$. 
The stability estimate follows from the closed graph theorem. 
\end{proof}
Let us prove Lemma~\ref{maxregstrdamphalf}:
\begin{proof}
We obtain the uniqueness for~(\ref{eqlinhalf}) from the fact that in the case $g=0$ we can consider $-\Delta$ as a self-adjoint and non negative operator with homogeneous Dirichlet boundary conditions and we can use Ref.~\cite{Haraux}. 
To verify the necessity of the conditions on the data, we suppose that $u\in\mathbb{E}$ defined in~(\ref{Halfspaceregsol}) is a solution of~(\ref{eqlinhalf}). 
Then 
$$u,\;u_t\in  H^1(\mathbb{R_+};H^s(\Omega))\cap L^2( \mathbb{R_+};H^{s+2}(\Omega)) \hbox{ and thus } f\in L^2(\mathbb{R}_+;H^s(\Omega)).$$ 
Taking as in the previous proof the spatial trace $Tr_{\partial\Omega}$ as in~(\ref{Trachalf1}) we have 
$$g,\;g_t\in H^{3/4}(\mathbb{R}_+;H^s(\partial\Omega))\cap L^2(\mathbb{R}_+;H^{s+3/2}(\partial\Omega)), \hbox{ which implies }g\in \mathbb{F}_{\mathbb{R}_+}.$$ 
By the Sobolev embedding $H^1(\mathbb{R}_+;H^{s+2}(\Omega))\hookrightarrow C(\mathbb{R}_+;H^{s+2}(\Omega))$, it follows that $u_0\in H^{s+2}(\Omega)$ and we also have the temporal trace
$$u\mapsto u\vert_{t=0}:H^1(\mathbb{R_+};H^s(\Omega))\cap L^2( \mathbb{R_+};H^{s+2}(\Omega))\rightarrow H^{s+1}(\Omega)$$
by Lemma~3.7 in Ref.~\cite{Denk}. For the compatibility condition we use~(\ref{Trachalf2})
and~(\ref{Trachalf3}) as in the proof of Lemma~\ref{lemwpheat}.

It remains to prove the sufficiency of the conditions. We extend $u_0$, $u_1$ and $f$ in odd functions among $x_1$ on $\mathbb{R}^n$ so that we have $\tilde{u}_0\in H^{s+2}(\mathbb{R}^n)$, $\tilde{u}_1\in H^{s+1}(\mathbb{R}^n)$  and $\tilde{f}\in L^2(\mathbb{R}_+;H^s(\mathbb{R}^n))$. We consider the problem
$$
\left\lbrace
\begin{array}{l}
\tilde{u}_{tt}-c^2\Delta \tilde{u} -\nu \varepsilon \Delta \tilde{u}_t=\tilde{f}\;\;\;\hbox{ in }\;\mathbb{R_+}\times \mathbb{R}^n,\\
\tilde{u}(0)=\tilde{u}_0, \;\;\;\tilde{u}_t(0)=\tilde{u}_0\;\;\;\hbox{ in }\; \mathbb{R}^n.
\end{array}\right.$$
By Theorem~4.1 in Ref.\cite{Perso} we obtain the existence of its unique solution
$$\tilde{u}\in H^2(\mathbb{R_+};H^s(\mathbb{R}^n))\cap H^1( \mathbb{R_+};H^{s+2}(\mathbb{R}^n)).$$
Let $\overline{u}\in \mathbb{E}$, defined in~(\ref{Halfspaceregsol}), denote the restriction of $\tilde{u}$ to $\Omega$ and let $\overline{g}:= g-\overline{u}\vert_{\partial\Omega}$. 
By the spatial trace theorem $\overline{u}\vert_{\partial\Omega}\in \mathbb{F}_{\mathbb{R}_+}$, and hence $\overline{g}\in  \mathbb{F}_{\mathbb{R}_+}$. Then the solution $u$ of the non homogeneous linear problem~(\ref{eqlinhalf}) is given by $u=v+\overline{u}$, where $v$ solves  probleme~(\ref{eqlinhalf}) with $f=u_0=u_1=0$ and $g=\overline{g}$. %
From Lemma~\ref{lemwpheat} we have a unique solution $\overline{v}\in \mathbb{E}_u$ of the problem~(\ref{heatderiv}) with $g=\overline{g}$.
Then the function $w:=v-\overline{v}$ solves the following system
$$
\left\lbrace
\begin{array}{l}
w_{tt}-\Delta w -\nu \varepsilon \Delta w_t= c^2 \Delta \overline{v}\;\;\;\hbox{ in }\;\mathbb{R_+}\times \Omega,\\
w=0\;\;\;\hbox{ on }\; \mathbb{R_+}\times\partial\Omega,\\
w(0)=0, \;\;\;w_t(0)=0\;\;\;\hbox{ in }\; \Omega,
\end{array}\right.
$$
which thanks to Theorem~2.6 in Ref.~\cite{Haraux} has a unique solution $w\in \mathbb{E}$ defined in~(\ref{Halfspaceregsol}). The function $u:=w+\overline{v}+\overline{u}$ is the desired solution of~(\ref{eqlinhalf}) and the stability estimate follows from the closed graph theorem. This concludes the proof of Lemma~\ref{maxregstrdamphalf}.
\end{proof}

 The next theorem follows from  the maximal regularity result  and Theorem~\ref{thSuh}. Its  proof is similar to the proof of Theorem~\ref{globwellposKuzper} and hence is omitted.
\begin{theorem}\label{ThMainWPnuPGlobHalf}
 Let $\nu>0$, $n\in \mathbb{N}^*$, $\Omega=\mathbb{R}_+\times \mathbb{R}^{n-1}$ and $s>\frac{n}{2}$. Considering the initial boundary value problem for the Kuznetsov equation in the half space with the Dirichlet boundary condition
 \begin{equation}\label{kuzhalfinitbound}
 \left\lbrace
 \begin{array}{l}
 u_{tt}-c^2\Delta u-\nu\varepsilon \Delta u_t=\alpha\varepsilon u_t u_{tt}+\beta \varepsilon \nabla u \nabla u_t\;\;\;\hbox{ in }\;[0,+\infty[\times \Omega,\\
 u=g\;\;\;\hbox{ on }\; [0,\infty[\times\partial\Omega,\\u(0)=u_0, \;\;\;u_t(0)=u_1\;\;\;\hbox{ in }\; \Omega,
\end{array}
\right.
\end{equation}
the following results hold: there exists constants $r^*=O(1)$ and $C_1=O(1)$, such that for all  initial data satisfying 
\begin{itemize}
\item $g\in \mathbb{F}_{\mathbb{R}^+}:= H^{7/4}([0,\infty[;H^s(\partial\Omega))\cap H^1([0,\infty[;H^{s+3/2}(\partial\Omega))$,
\item $u_0\in H^{s+2}(\Omega)$, $u_1\in H^{s+1}(\Omega)$,
\item $g(0)=u_0\vert_{\partial\Omega}$ and $g_t(0)=u_1\vert_{\partial\Omega}$,
\end{itemize} and such that for $r\in [0,r^*[$
\begin{equation*}
 \Vert u_0\Vert_{H^{s+2}(\Omega)}+\Vert u_1\Vert_{H^{s+1}(\Omega)}+\Vert g\Vert_{\mathbb{F}_{[0,T]}}\leq \frac{\nu \varepsilon}{C_1}r,
\end{equation*}
 there exists  a unique solution of problem~(\ref{kuzhalfinitbound})
  $$u\in H^2([0,\infty[;H^s(\Omega))\cap H^1( [0,\infty[;H^{s+2}(\Omega)), $$  
  such that 
 $ \Vert u\Vert_{H^2([0,\infty[;H^s(\Omega))\cap H^1( [0,\infty[;H^{s+2}(\Omega))}\leq 2r.$
 \end{theorem}
\subsection{Approximation of the solutions of the Kuznetsov equation by the solutions of the KZK equation.}\label{secValKuzKZK}
Given Theorem~\ref{globwellposKuzper} for the viscous case, we   consider  the Cauchy problem associated to the KZK equation~(\ref{NPEcau2}) for small enough initial data in order to have a time periodic solution $I$ defined on $\mathbb{R}_+\times\mathbb{R}^{n-1}$. If $\nu>0$, to compare the solutions of the Kuznetsov and the KZK equations we consider two cases. The first case is considered in Subsubsection~\ref{sssKKZKperbc}, when 
the Kuznetsov equation can be considered as a time periodic boundary problem coming just from the initial condition $I_0$ of problem~(\ref{NPEcau2}). In Subsubsection~\ref{sssKKZKperbcIn} we study the second case, when  the solution of the KZK equation taken for $\tau=0$ gives $I(0,z,y) $ defined on $\mathbb{R}_+\times\mathbb{R}^{n-1}$ from which we deduce according to the derivation \textit{ansatz} both an initial condition for the Kuznetsov equation at $t=0$ and a corresponding boundary condition.
In this second situation, it aslo makes sense to consider the inviscid case, breifly commented in the end of
Subsubsection~\ref{sssKKZKperbcIn}.
\subsubsection{Approximation problem for the Kuznetsov with periodic boundary conditions.}\label{sssKKZKperbc}
Let $\Omega_1=\mathbb{T}_{\tau}\times \mathbb{R}^{n-1}$ and $s\geq \left[ \frac{n}{2}\right]+1$. Suppose that a function $I_0(t,y)=I_0(t,\sqrt{\varepsilon }x')$ is such that $I_0 \in H^s(\Omega_1)$ small enough and $\int_{\mathbb{T}_{\tau}} I_0(s,y) ds=0$. Then by Theorem~\ref{wpglopkzknpe} there is a unique solution $I(\tau,z,y)$ of the Cauchy problem for the KZK equation~(\ref{NPEcau2}) such that
\begin{equation}\label{regsolKZK}
z\mapsto I(\tau,z,y) \in C([0,\infty[,H^s(\Omega_1))
\end{equation}
with $\int_{\mathbb{T}_{\tau}} I(l,z,y) dl=0$.
We use the operator $\partial_\tau^{-1}$ defined in~(\ref{invdtau}).
Formula~(\ref{invdtau}), which implies that $ \partial_{\tau}^{-1} I$ is $L$-periodic in $\tau$ and of mean value zero, gives us the estimate
$$\Vert \partial_{\tau}^{-1} I\Vert_{H^s(\Omega_1)} \leq C \Vert \partial_{\tau} \partial_{\tau}^{-1} I\Vert_{H^s(\Omega_1)} = C \Vert  I\Vert_{H^s(\Omega_1)}.$$
So $\partial_{\tau}^{-1} I\vert_{z=0}\in H^s(\Omega_1)$, and hence by~(\ref{regsolKZK})
$$
z\mapsto \partial_{\tau}^{-1} I(\tau,z,y) \in C([0,\infty[,H^s(\Omega_1)),
$$
with $\int_{\mathbb{T}_{\tau}}\partial_{\tau}^{-1} I(s,z,y) ds=0$.

We define on $\mathbb{T}_t\times \mathbb{R}_+\times \mathbb{R}^{n-1}$ 
\begin{equation}\label{potentialkzk}
\overline{u}(t,x_1,x'):=\frac{c^2}{\rho_0}\partial_{\tau}^{-1} I(\tau,z,y)=\frac{c^2}{\rho_0}\partial_{\tau}^{-1} I\left(t-\frac{x_1}{c},\varepsilon x_1,\sqrt{\varepsilon}x'\right)
\end{equation}
 with the paraxial change of variable~(\ref{chvarkzk}) associated to the KZK equation. Thus $\overline{u}$ is $L$-periodic in time and of mean value zero. 
Now we consider the Kuznetsov problem~(\ref{kuzper}) associated to the following boundary condition, imposed by the initial condition for the KZK equation:
\begin{equation}\label{boundcondkuz}
g(t,x'):=\overline{u}(t,0,x')=\frac{c^2}{\rho_0}\partial_{\tau}^{-1} I_0(\tau,y).
\end{equation} 
Taking $\tilde{I}:=\frac{\rho_0}{c^2}\partial_{\tau}\Phi$ (see Eq.~(\ref{Ikzk})), let $\tilde{I}$ be the solution of the Kuznetsov equation written in the following form with the remainder $R^{Kuz-KZK}$ defined in Eq.~(\ref{remkuzkzk}):
\begin{equation}\label{KZKwithre}
\left\lbrace
\begin{array}{l}
c\partial_{ z} \tilde{I} -\frac{(\gamma+1)}{4\rho_0}\partial_\tau
\tilde{I}^2-\frac{\nu}{2 c^2\rho_0}\partial^2_\tau \tilde{I}-\frac{c^2}2 \Delta_y\partial_{\tau}^{-1}
\tilde{I}+\varepsilon \frac{\rho_0}{2c^2}R^{Kuz-KZK}=0,\\
\tilde{I}\vert_{z=0}=I_0,
\end{array}
\right.
\end{equation}
where we can recognize the system associated to the KZK equation~(\ref{NPEcau2}).

Now we can formulate the following approximation result 
\begin{theorem}\label{AproxKuzKZK}
Let $\nu>0$. 
For $s>\frac{n}{2}+2$ and $I_0\in H^{s+\frac{3}{2}}(\mathbb{T}_{\tau}\times \mathbb{R}^{n-1})$ small enough in $H^{s+\frac{3}{2}}(\mathbb{T}_{\tau}\times \mathbb{R}^{n-1}) $, there exists a unique global solution $I$ of the Cauchy problem for the KZK equation~(\ref{NPEcau2}) such that
$$z\mapsto I(\tau,z,y) \in C([0,\infty[,H^{s+\frac{3}{2}}(\mathbb{T}_{\tau}\times \mathbb{R}^{n-1})).$$
In addition, there exists a unique global solution $\tilde{I}$ of the Kuznetsov problem~(\ref{KZKwithre}), in the sense
$\tilde{I}:=\frac{\rho_0}{c^2}\partial_{\tau}\Phi,$
with $\Phi(\tau,z,y):=u(t,x_1,x')$ with the paraxial change of variable~(\ref{chvarkzk}) and 
$$u\in H^2(\mathbb{T}_t;H^s(\mathbb{R}^+\times\mathbb{R}^{n-1}))\cap H^1(\mathbb{T}_t;H^{s+2}(\mathbb{R}^+\times\mathbb{R}^{n-1})),$$
is the global solution of the periodic problem~(\ref{kuzper}) for the Kuznetsov equation with $g$ defined by $I_0$ as in Eq.~(\ref{boundcondkuz}). Moreover there exist $C_1>0$ and $C_2>0$ such that
$$\frac{1}{2}\frac{d}{dz}\Vert I-\tilde{I}\Vert^2_{L^{2}(\mathbb{T}_\tau\times \mathbb{R}^{n-1})}\leq C_1 \Vert I-\tilde{I}\Vert^2_{L^{2}(\mathbb{T}_\tau\times \mathbb{R}^{n-1})}+C_2\varepsilon \Vert I-\tilde{I}\Vert_{L^{2}(\mathbb{T}_\tau\times \mathbb{R}^{n-1})},$$
which implies 
$$\Vert I-\tilde{I}\Vert_{L^{2}(\mathbb{T}_\tau\times \mathbb{R}^{n-1})}(z)
\leq \frac{C_2}{2}\eps z e^{\frac{C_1}{2} z}
\leq \frac{C_2}{C_1}\varepsilon(e^{\frac{C_1}{2}z}-1)$$
and $\Vert I-\tilde{I}\Vert_{L^{2}(\mathbb{T}_\tau\times \mathbb{R}^{n-1})}(z)\leq K\varepsilon$ while $z\leq C$ with $K>0$, and $C>0$ independent of $\varepsilon$.
\end{theorem}
\begin{proof}
For $s> \frac{n}{2}+2$, the global well-posedness of $I$ comes from Theorem~\ref{wpglopkzknpe} if $I_0\in H^{s+\frac{3}{2}}(\mathbb{T}_{\tau}\times \mathbb{R}^{n-1})$ is small enough. Moreover, since $g$ is given by Eq.~(\ref{boundcondkuz}), thanks to the definition  of $\partial_{\tau}^{-1}$ in~(\ref{invdtau}) and the fact that $I_0\in H^{s+\frac{3}{2}}(\mathbb{T}_{\tau}\times \mathbb{R}^{n-1})$, we have 
$$g\in H^{s+\frac{3}{2}}(\mathbb{T}_t\times \mathbb{R}^{n-1}) \hbox{ and }\partial_t g\in H^{s+\frac{3}{2}}(\mathbb{T}_t\times \mathbb{R}^{n-1}).$$ 
And thus 
$$g\in H^{\frac{7}{4}}(\mathbb{T}_t;H^s(\mathbb{R}^{n-1}))\cap H^1(\mathbb{T}_t;H^{s+2-\frac{1}{2}}(\mathbb{R}^{n-1})).$$
Therefore we can use Theorem~\ref{globwellposKuzper} which implies the global existence of the periodic in time solution 
$$u\in H^2(\mathbb{T}_t;H^s(\mathbb{R}^+\times\mathbb{R}^{n-1}))\cap H^1(\mathbb{T}_t;H^{s+2}(\mathbb{R}^+\times\mathbb{R}^{n-1})),$$
of the Kuznetsov periodic boundary value problem~(\ref{kuzper}) as $I_0$ is small enough in $H^{s+\frac{3}{2}}(\mathbb{T}_{\tau}\times \mathbb{R}^{n-1})$. 
Therefore, it also implies the global existence of $\tilde{I}$ defined in~(\ref{Ikzk}) which is the solution of the exact Kuznetsov system~(\ref{KZKwithre}).

Now we subtract the equations in  systems~(\ref{NPEcau2}) and~(\ref{KZKwithre}):
\begin{align*}
c\partial_z (I-\tilde{I})-\frac{\gamma+1}{2\rho_0}(I-\tilde{I})\partial_{\tau} I -\frac{\gamma+1}{2\rho_0} \tilde{I}&\partial_{\tau}(I-\tilde{I}) -\frac{\nu}{2c^2\rho_0}\partial^2_{\tau}(I-\tilde{I})\\
&-\frac{c^2}{2}\partial_{\tau}^{-1} \Delta_y(I-\tilde{I})=\varepsilon\frac{\rho_0}{2c^2}R^{Kuz-KZK}.
\end{align*}
Denoting $\Omega_1=\mathbb{T}_{\tau}\times \mathbb{R}^{n-1}$, we multiply this equation by $(I-\tilde{I})$, integrate over $\mathbb{T}_{\tau}\times \mathbb{R}^{n-1}$ and perform a standard integration by parts which gives
\begin{align*}
&\frac{c}{2}\frac{d}{dz}\Vert I-\tilde{I}\Vert_{L^2(\Omega_1)}^2-\frac{\gamma+1}{2\rho_0}\int_{\Omega_1}\partial_{\tau} I( I-\tilde{I})^2d\tau dy\\
&-\frac{\gamma+1}{2\rho_0} \int_{\Omega_1} \tilde{I} (I-\tilde{I}) \partial_{\tau}(I-\tilde{I})d\tau dy\\
&+\frac{\nu}{2c^2\rho_0}\int_{\Omega_1} (\partial_{\tau}(I-\tilde{I}))^2d\tau dy=\varepsilon\frac{\rho_0}{2c^2}\int_{\Omega_1} R^{Kuz-KZK} (I-\tilde{I})d\tau dy.
\end{align*}
Let us notice that
\begin{align*}
&\int_{\Omega_1} \tilde{I} (I-\tilde{I}) \partial_{\tau}(I-\tilde{I})d\tau dy= \int_{\Omega_1} [(\tilde{I}-I)+I)]\frac{1}{2}\partial_{\tau}(I-\tilde{I})^2 d\tau dy=\\
&=-\frac{1}{2}\int_{\Omega_1} \partial_{\tau}I (I-\tilde{I})^2 d\tau dy,
\end{align*}
and as for $s> \frac{n}{2}+2$ and $u\in H^2(\mathbb{T}_t;H^s(\Omega))\cap H^1(\mathbb{T}_t;H^{s+2}(\Omega))$ we also have 
\begin{equation}\label{remboundkuzkzk}
R^{Kuz-KZK}\in C(\mathbb{R}_+;L^2(\mathbb{T}_{\tau}\times \mathbb{R}^{n-1})).
\end{equation}
This comes from the fact that in system~(\ref{KZKwithre}) the worst term outside the remainder is $\partial_{\tau}^2\tilde{I}$ with $\tilde{I}$ given by
Eq.~(\ref{Ikzk}). As 
$\partial_t^3 u\in L^2(\mathbb{T}_t;H^{s-2}(\Omega)),$
we need to take $s> \frac{n}{2}+2$ to have $\partial_{\tau}^2\tilde{I}$ in $L^{\infty}(\mathbb{R}_+;L^2(\mathbb{T}_{\tau}\times \mathbb{R}^{n-1}))$.
Therefore
$$\left\vert \int_{\Omega_1} R^{Kuz-KZK} (I-\tilde{I})d\tau dy\right\vert\leq \Vert R^{Kuz-KZK}\Vert_{L^2(\Omega_1)} \Vert I-\tilde{I}\Vert_{L^2(\Omega_1)}\leq C \Vert I-\tilde{I}\Vert_{L^2(\Omega_1)}$$
with a constant $C>0$ independent on $z$ thanks to~(\ref{remboundkuzkzk}). It leads to the estimate
$$\frac{1}{2}\frac{d}{dz}\Vert I-\tilde{I}\Vert_{L^2(\Omega_1)}^2\leq K \sup_{\Omega_1}\vert \partial_{\tau}I(\tau,z,y)\vert \;\;\Vert I-\tilde{I}\Vert_{L^2(\Omega_1)}^2+C\varepsilon  \Vert I-\tilde{I}\Vert_{L^2(\Omega_1)},$$
in which, due to the regularity of $I$ for $s$ and $I_0$ (see also Point 1 and 3 of Theorem~\ref{wpglopkzknpe}) the term $\sup_{\Omega_1}\vert \partial_{\tau}I(\tau,z,y)\vert$ is bounded by a constant $C>0$ independent on $z$. With this we have the desired estimate and the other results follow from Gronwall's Lemma.
\end{proof}
\begin{remark}
Here the regularity $I_0\in H^{s+\frac{3}{2}}(\mathbb{T}_{\tau}\times\mathbb{R}^{n-1})$ for $s>\frac{n}{2}+2$ is the minimal regularity to ensure (\ref{remboundkuzkzk}).
\end{remark}

\subsubsection{Approximation problem for the Kuznetsov equation with initial-boundary conditions.}\label{sssKKZKperbcIn}
Let as previously $\Omega_1=\mathbb{T}_{\tau}\times \mathbb{R}^{n-1}$, but $s\geq \left[ \frac{n+1}{2}\right]$. Suppose that a function $I_0(t,y)=I_0(t,\sqrt{\varepsilon }x')$ is such that $I_0 \in H^s(\Omega_1)$ and $\int_{\mathbb{T}_{\tau}} I_0(s,y) ds=0$. Then by Theorem~\ref{wpglopkzknpe} there is a unique solution $I(\tau,z,y)$ of the Cauchy problem~(\ref{NPEcau2}) for the KZK equation such that
$$z\mapsto I(\tau,z,y) \in C([0,\infty[,H^s(\Omega_1)).$$
We define $\overline{u}$ and $g$ as in Eqs.~(\ref{potentialkzk}) and~(\ref{boundcondkuz}) respectively.
Thus, for $R^{Kuz-KZK}$ defined in Eq.~(\ref{remkuzkzk}), $\overline{u}$ is the solution of the following system 
\begin{equation}\label{aproxkuzkzkrem}
\left\lbrace
\begin{array}{l}
\partial^2_t \overline{u} -c^2 \Delta \overline{u}-\varepsilon\partial_t\left( (\nabla \overline{u})^2+\frac{\gamma-1}{2c^2}(\partial_t \overline{u})^2+\frac{\nu}{\rho_0}\Delta \overline{u}\right)=\varepsilon^2 R^{Kuz-KZK}\;\;\;\hbox{ in }\;\mathbb{T}_t\times \Omega,\\
\overline{u}=g\;\;\;\hbox{ on }\; \mathbb{T}_t\times\partial\Omega.
\end{array}\right.
\end{equation}

We study for $T>0$ the solution $u$ of the Dirichlet boundary-value problem~(\ref{kuzhalfinitbound}) for the Kuznetsov equation on $[0,T]\times\mathbb{R}_+\times \mathbb{R}^{n-1}$, taking $u_0:=\overline{u}(0)$ and $u_1:=\overline{u}_t(0)$ and considering the time periodic function $g$ defined by Eq.~(\ref{boundcondkuz}) as a function on $[0,T]$.
Now we have the following stability result.
\begin{theorem}\label{approxKuzKZKbis}
Let $T>0$, $\nu>0$, $n\geq 2$,  $\Omega=\mathbb{R}^+\times\mathbb{R}^{n-1}$ and $I_0\in H^{s}(\mathbb{T}_{\tau}\times\mathbb{R}^{n-1})$, $s\in \R^+$. 
Let $I$ be  the solution of the KZK equation. 
By $I$  the solution $\overline{u}$ of the approximated Kuznetsov problem~(\ref{aproxkuzkzkrem}) is constructed using~(\ref{potentialkzk}) and with $g$ defined in~(\ref{boundcondkuz}). 

Then there hold
\begin{enumerate}
\item If $s\geq 6$ for $n=2,3$, or else $\left[\frac{s}{2}\right]>\frac{n}{2}+1$, there exists $k>0$ such that $\Vert I_0\Vert_{H^{s}}<k$ implies the global well-posedness of the Cauchy problem for the KZK equation. Its solution is denoted for $0\leq k\leq \left[\frac{s}{2}\right] $ by 
$$I\in C^k(\lbrace z >0\rbrace;H^{s-2k}(\mathbb{T}_{\tau}\times \mathbb{R}^{n-1})),$$
thus
$$
\overline{u}\in  C^k(\lbrace z >0\rbrace;H^{s-2k}(\mathbb{T}_{\tau}\times \mathbb{R}^{n-1})),\;
\partial_t\overline{u}\in  C^k(\lbrace z >0\rbrace;H^{s-2k}(\mathbb{T}_{\tau}\times \mathbb{R}^{n-1})),$$
or again
\begin{equation}\label{regubarkuzkzk}
\overline{u}\in H^2(\mathbb{T}_t,H^{\left[\frac{s}{2} \right]-1}(\Omega))\cap H^1(\mathbb{T}_t,H^{\left[\frac{s}{2} \right]}(\Omega)) .
\end{equation}
The regularity of $I_0\in H^s(\mathbb{T}_t\times\mathbb{R}^{n-1})$  (see Table~\ref{TABLE2}) is minimal to ensure  that $R^{Kuz-KZK}$, see Eq.~(\ref{remkuzkzk}), is in $C([0,+\infty[;L^2(\mathbb{R}_+\times\mathbb{R}^{n-1}))$.
 \item If $\left[\frac{s}{2}\right]>\frac{n}{2}+2$, taking the same initial data for the exact boundary-value problem for the Kuznetsov equation~(\ref{kuzhalfinitbound}) as for $\overline{u}$, $i.e.$   \begin{align*}                                                                                                                                                                                                         
   &u(0)=\overline{u}(0)=\frac{c^2}{\rho_0}\partial_{\tau}^{-1}I(-\frac{x_1}{c},\varepsilon x_1,\sqrt{\varepsilon}x')\in H^{\left[\frac{s}{2} \right]}(\Omega),    \\
   &u_t(0)=\overline{u}_t(0)=\frac{c^2}{\rho_0}\partial_{\tau}I(-\frac{x_1}{c},\varepsilon x_1,\sqrt{\varepsilon}x')\in H^{\left[\frac{s}{2} \right]-1}(\Omega),   
   \end{align*}
 there exists $k>0$  such that $\Vert I_0\Vert_{H^{s}}<k$ implies the well-posedness of the exact Kuznetsov equation~(\ref{kuzhalfinitbound}) considered with Dirichlet boundary condition 
 \begin{align*}
 g= \frac{c^2}{\rho_0}\partial_{\tau}^{-1}I_0\in H^s(\mathbb{T}_t\times\mathbb{R}^{n-1})\subset H^{7/4}([0,T]&;H^{\left[\frac{s}{2} \right]-2}(\partial\Omega))\\
& \cap H^1([0,T];H^{\left[\frac{s}{2} \right]-2+3/2}(\partial\Omega))
 \end{align*}
  and the regularity
\begin{equation}\label{regukuzkzk}
u\in H^2([0,T],H^{\left[\frac{s}{2} \right]-1}(\Omega))\cap H^1([0,T],H^{\left[\frac{s}{2} \right]}(\Omega)) .
\end{equation}
Moreover, there exists $K>0$, and $C>0$ independent of $\varepsilon$ such that for all
$t\leq\frac{C}{\varepsilon}$
we have $C_1>0$ and $C_2>0$ with
\begin{equation}\label{estimKuzKZKex}
\sqrt{\Vert (u -\overline{u})_t(t)\Vert_{L^2(\Omega)}^2+ \Vert \nabla (u-\overline{u})(t)\Vert_{L^2(\Omega)}^2}\leq C_1\eps^2t e^{C_2\eps t}\leq K\varepsilon. 
\end{equation}

 \item In addition, let $u$ be a solution of the Dirichlet boundary-value problem~(\ref{kuzhalfinitbound}) for the Kuznetsov equation, with $g$ defined by Eq.~(\ref{boundcondkuz}) and 
$u_0\in H^{m+2}(\Omega)$, $u_1\in H^{m+1}(\Omega)$ with $m>\frac{n}{2}$ and
\begin{equation}\label{EqSmallInD}
 \Vert (u -\overline{u})_t(0)\Vert_{L^2(\Omega)}^2+ \Vert \nabla (u-\overline{u})(0)\Vert_{L^2(\Omega)}^2\leq \delta^2 \leq \varepsilon^2.
\end{equation}
There exists $K>0$ and $C>0$ independent of $\varepsilon$ such that for all
$t\leq\frac{C}{\varepsilon}$
we have $C_1>0$ and $C_2>0$ with
\begin{equation}\label{estimKuzKZK}
 \sqrt{\Vert (u -\overline{u})_t(t)\Vert_{L^2(\Omega)}^2+ \Vert \nabla (u-\overline{u})(t)\Vert_{L^2(\Omega)}^2}\leq  C_1(\eps^2t+\delta^2)e^{C_2\eps t}\le K\eps.
\end{equation}
\end{enumerate}
\end{theorem}
\begin{proof}
Let $\overline{u}$ and $g$ be defined by~(\ref{potentialkzk}) and~(\ref{boundcondkuz}) by the solution $I$ of the Cauchy problem~(\ref{NPEcau2}) for the KZK equation with $I\vert_{z=0}=I_0\in H^s(\mathbb{T}_t\times\mathbb{R}^{n-1})$ and $s\geq 6$ for $n=2,3$, or else $\left[\frac{s}{2}\right]>\frac{n}{2}+1$. In this case, $\overline{u}$ is the global solution of the approximated Kuznetsov system~(\ref{aproxkuzkzkrem}), what is a direct consequence of Theorem~\ref{wpglopkzknpe}.
If $I_0\in H^s(\mathbb{T}_t\times \mathbb{R}^{n-1})$ with the chosen $s$, then for $0\leq k \leq \left[\frac{s}{2} \right]$
$$I(\tau,z,y)\in C^k(\lbrace z >0\rbrace;H^{s-2k}(\mathbb{T}_{\tau}\times \mathbb{R}^{n-1})).$$
Let us denote $\Omega_1=\mathbb{T}_{\tau}\times \mathbb{R}^{n-1}$.
Given the equation for $\overline{u}$ by~(\ref{potentialkzk}),  we have  \begin{align*}
 \overline{u}(\tau,z,y)\hbox{ and }\partial_\tau\overline{u}(\tau,z,y)\in & C^k(\lbrace z >0\rbrace;H^{s-2k}(\Omega_1)),\hbox{ if }0\leq k\leq \left[\frac{s}{2} \right] ,\\
\partial_\tau^2\overline{u}(\tau,z,y)\in & C^k(\lbrace z >0\rbrace;H^{s-1-2k}(\Omega_1)),\hbox{ if }0\leq k\leq \left[\frac{s}{2} \right]-1,
\end{align*}
but we can also say\cite{Ito} thanks to Point~4 of Theorem~\ref{wpglopkzknpe} that 
\begin{align*}
\overline{u}(\tau,z,y)\hbox{ and }\partial_\tau\overline{u}(\tau,z,y)\in &  H^k(\lbrace z >0\rbrace;H^{s-2k}(\Omega_1)),\\
\partial_\tau^2\overline{u}(\tau,z,y)\in & H^k(\lbrace z >0\rbrace;H^{s-1-2k}(\Omega_1)).
\end{align*}
This implies as for the chosen $s$ that 
\begin{align*}
&\overline{u}(t,x_1,x')\hbox{ and } \partial_t\overline{u}(t,x_1,x')\in  L^2(\mathbb{T}_{t};H^{\left[\frac{s}{2} \right]}(\Omega)\cap H^2(\mathbb{T}_{t};H^{\left[\frac{s}{2} \right]-1}(\Omega),\\
&\partial_t^2\overline{u}(t,x_1,x')\in  L^2(\mathbb{T}_{t};H^{\left[\frac{s}{2} \right]-1}(\Omega)\cap  H^2(\mathbb{T}_{t};H^{\left[\frac{s}{2} \right]-2}(\Omega).
\end{align*}
This implies 
\begin{align*}
\overline{u}(t,x_1,x')\in & C^1([0,+\infty[;H^{\left[\frac{s}{2} \right]-1}(\Omega),\\
\partial_t^2\overline{u}(t,x_1,x')\in & C([0,+\infty[;H^{\left[\frac{s}{2} \right]-2}(\Omega).
\end{align*}
With the chosen $s$, these regularities of  $\overline{u}(t,x_1,x')$ give us the regularity (\ref{regubarkuzkzk}) and allow to have all left-hand terms in the approximated Kuznetsov system~(\ref{aproxkuzkzkrem}) of the desired regularity, $i.e$ in $C([0,+\infty[;L^2(\Omega))$.
In addition for $\left[\frac{s}{2}\right]>\frac{n}{2}+2$ 
with the chosen $g$, $u_0=\overline{u}(0)$ and $u_1=\overline{u}_t(0)$ in the conditions of the theorem
 we have 
 $$u_0\in H^{\left[\frac{s}{2} \right]}(\Omega),\hbox{ }u_1\in H^{\left[\frac{s}{2} \right]-1}(\Omega) $$
 with
 $$g\in H^s(\mathbb{T}_t\times\mathbb{R}^{n-1}) \hbox{ and }\partial_t g \in H^s(\mathbb{T}_t\times\mathbb{R}^{n-1}),$$
 which implies
 $$g\in  H^{7/4}(]0,T[;H^{\left[\frac{s}{2} \right]-2}(\partial\Omega))\cap H^1(]0,T[;H^{\left[\frac{s}{2} \right]-2+3/2}(\partial\Omega))$$ 
 with $\left[\frac{s}{2} \right]-2 > \frac{n}{2}$
 as required by Theorem~\ref{ThMainWPnuPGlobHalf} to have the local well-posedness of $u$, the solution of the Kuznetznov equation associated to  system~(\ref{kuzhalfinitbound}). 
This completes the local well-posedness results and we deduce that $u$ have the desired regularity (\ref{regukuzkzk}) announced in the Theorem.
Moreover, we have $R^{Kuz-KZK}$ in $C([0,+\infty[,L^2(\Omega)).$

To validate the approximation we will only demonstrate the estimate in point $(3)$ as it directly implies the estimate in point $(2)$.
We take again $I_0\in H^s(\mathbb{T}_t\times\mathbb{R}^{n-1})$ with $\left[\frac{s}{2}\right]>\frac{n}{2}+2$ to define $\overline{u}$ and $g$ and consider $u$ to be a solution of the Dirichlet boundary-value problem~(\ref{kuzhalfinitbound}) for the Kuznetsov equation under the conditions $u_0\in H^{m+2}(\Omega)$, $u_1\in H^{m+1}(\Omega)$ with $m>\frac{n}{2}$ satisfying (\ref{EqSmallInD}).
Now we subtract the Kuznetsov equation from the approximated Kuznetsov equation (see system~(\ref{aproxkuzkzkrem})), multiply by $(u-\overline{u})_t$ and integrate over $\Omega$ to obtain as in Ref.~\cite{Perso} the following stability estimation:
\begin{align*}
\frac{1}{2}\frac{d}{dt}\Big(\int_{\Omega}A(t,x)\; (u-\overline{u})_t^2+ & c^2 (\nabla(u-\overline{u}))^2dx\Big) \\
\leq C \varepsilon &\sup(\Vert u_{tt}\Vert_{L^{\infty}(\Omega)};\Vert \Delta u\Vert_{L^{\infty}(\Omega)};\Vert \nabla \overline{u}_t\Vert_{L^{\infty}(\Omega)})\\
 & \cdot\left(\Vert (u -\overline{u})_t\Vert_{L^2(\Omega)}^2+ \Vert \nabla (u-\overline{u})\Vert_{L^2(\Omega)}^2\right)\\
 &+\varepsilon^2\int_{\Omega}R^{Kuz-KZK} (u-\overline{u})_tdx
\end{align*}
where $\frac{1}{2}\leq A(t,x)\leq \frac{3}{2}$ for $0\leq t\leq T$ and $x\in \Omega$. By regularity of the solutions $\sup(\Vert u_{tt}\Vert_{L^{\infty}(\Omega)};\Vert \Delta u\Vert_{L^{\infty}(\Omega)};\Vert \nabla \overline{u}_t\Vert_{L^{\infty}(\Omega)})$ is bounded in time on $[0,T]$. 
Moreover, we have $\Vert R^{Kuz-KZK}(t)\Vert_{L^2(\Omega)}$ bounded for $t\in[0,T]$ by the regularity of $\overline{u}$ where $R^{Kuz-KZK}$ is defined in~(\ref{remkuzkzk}). Then after integration on $[0,t]$, we can write 
\begin{align*}
\Vert (u -\overline{u})_t(t)\Vert_{L^2(\Omega)}^2+ &\Vert \nabla (u-\overline{u})(t)\Vert_{L^2(\Omega)}^2\\
\leq & 3( \Vert (u -\overline{u})_t(0)\Vert_{L^2(\Omega)}^2+ \Vert \nabla (u-\overline{u})(0)\Vert_{L^2(\Omega)}^2)\\
& C_1 \varepsilon \int_0^t \Vert (u -\overline{u})_t(s)\Vert_{L^2(\Omega)}^2+ \Vert \nabla (u-\overline{u})(s)\Vert_{L^2(\Omega)}^2 ds\\
&+C_2 \varepsilon^2 \int_0^t\sqrt{\Vert (u -\overline{u})_t(s)\Vert_{L^2(\Omega)}^2+ \Vert \nabla (u-\overline{u})(s)\Vert_{L^2(\Omega)}^2 }ds.
\end{align*}
As $\Vert (u -\overline{u})_t(0)\Vert_{L^2(\Omega)}^2+ \Vert \nabla (u-\overline{u})(0)\Vert_{L^2(\Omega)}^2\leq \delta^2 \leq \varepsilon^2$,  we finally find by the Gronwall Lemma  
$$\sqrt{\Vert (u -\overline{u})_t(t)\Vert_{L^2(\Omega)}^2+ \Vert \nabla (u-\overline{u})(t)\Vert_{L^2(\Omega)}^2 }\leq C_1(\eps^2t+\delta^2)e^{C_2\eps t}\le K\eps$$
for $t\leq \frac{C}{\varepsilon}$
what allows us to conclude.
\end{proof}
For the inviscid media we use~(\ref{validaproxintro}) on the cone $C(t)$ defined in Theorem~\ref{ThAproEulKZK} 
instead of $\mathbb{R}^n$ when we compare the Euler system and the inviscid Kuznetsov equation. Therefore the triangular inequality permits us  to validate the approximation between the Kuznetsov and KZK equations in the inviscid case as their respective approximations with the Euler system are validated by~(\ref{validaproxintro}) in the cone.

\section{Approximation of the solutions of the Kuznetsov equation with the solutions of the NPE equation.}\label{secKuzNPE}
Now let us go back to the NPE equation introduced in Section~\ref{secNSNPE} and consider its \textit{ansatz}~(\ref{vNPE})--(\ref{P2NPE}). As previously we start with the viscous case $\nu>0$.

Then we can rewrite the Kuznetsov equation
\begin{align*}
&\partial_t^2 u-c^2\Delta u- \varepsilon \partial_t\left((\nabla u)^2+\frac{\gamma-1}{2c^2}(\partial_t u)^2+\frac{\nu}{\rho_0}\Delta u\right)\\
&=\varepsilon \left(-2c \partial^2_{\tau z} \Psi-c^2\Delta_y\Psi+\frac{\nu}{\rho_0} c\partial^3_z \Psi+\frac{\gamma+1}{2}c \partial_z(\partial_z \Psi)^2\right)+\varepsilon^2 R^{Kuz-NPE}
\end{align*}
with
\begin{align}
\varepsilon^2 R^{Kuz-NPE}=&\varepsilon^2 \big( \partial^2_{\tau}\Psi-\frac{\nu}{\rho_0} \partial^2_z\partial_{\tau}\Psi+\frac{\nu}{\rho_0}c\Delta_y\partial_z\Psi-(\gamma-1)\partial_{\tau}\Psi \;\partial^2_z\Psi\label{approxeqKuzNPE}\\
&\;\;\;\;\;-2(\gamma-1)\partial_z\Psi\; \partial^2_{\tau z}\Psi-2 \partial_z\Psi \;\partial^2_{\tau z}\Psi+2c \nabla_y\Psi\;\nabla_y\partial_z\Psi \big)\nonumber\\
&+\varepsilon^3 \big( -\frac{\nu}{\rho_0}\Delta_y \partial_{\tau}\Psi+2\frac{\gamma-1}{c}\partial_{\tau}\Psi\;\partial^2_{\tau z}\Psi+\frac{\gamma-1}{c}\partial_z \Psi\; \partial^2_{\tau}\Psi\nonumber\\
&\;\;\;\;\;\;\;\;\;\; -2\nabla_y\Psi\;\nabla_y\partial_{\tau}\Psi\big)
+\varepsilon^4 (-\frac{\gamma-1}{c^2}\partial_{\tau}\Psi \partial^2_{\tau}\Psi).\nonumber
\end{align}
We obtain the NPE equation satisfied by $\partial_z\Psi$ modulo a multiplicative constant:
$$
\partial^2_{\tau z}\Psi-\frac{\gamma+1}{4}\partial_z(\partial_z \Psi)^2-\frac{\nu}{2\rho_0}\partial^3_z \Psi+\frac{c}{2}\Delta_y \Psi=0.
$$
In the sequel we will work with $\xi$ defined by~(\ref{P1NPE}) which satisfies the Cauchy problem~(\ref{npecau}) for the NPE equation.
This time in relation with the KZK equation we used the bijection~(\ref{bijKZKNPE}). We also update our notation for 
$\Omega_1=\mathbb{T}_z\times \mathbb{R}^{n-1}_y$  and $s>\frac{n}{2}+1$. Suppose that $\xi_0\in H^{s+2}(\mathbb{T}_z\times \mathbb{R}^{n-1}_y)$ and $\int_{\mathbb{T}_z}\xi_0(z,y)\;dz=0$. 
Then there is a constant $r>0$ such that if $\Vert \xi_0\Vert_{H^{s+2}(\mathbb{T}_z\times \mathbb{R}^{n-1}_y)}<r$, then, by Theorem~\ref{wpglopkzknpe}, there is a unique solution 
$\xi\in C([0,\infty[;H^{s+2}(\mathbb{T}_z\times \mathbb{R}^{n-1}_y))$ of
 the NPE Cauchy problem~(\ref{npecau})
satisfying
$$\int_{\mathbb{T}_z} \xi(\tau,z,y) \;dz=0 \;\;\hbox{ for any }\;\tau\geq 0,\; y\in\mathbb{R}^{n-1}.$$
We define $\partial_{x_1} \overline{u}(t,x_1,x'):=-\frac{c}{\rho_0}\xi(\tau,z,y)$ with the change of variable~(\ref{chvarnpe}) and 
$$\overline{u}(t,x_1,x')=-\frac{c}{\rho_0} \partial_z^{-1} \xi(\tau,z,y)=\left(-\frac{c}{\rho_0}\right)\left(\int_0^z \xi(\tau,s,y)ds+\int_0^L \frac{s}{L} \xi(\tau,s,y)ds\right).$$
We notice $u_1(x_1,x'):= \partial_t\overline{u}(0,x_1,x') $ and $u_0(x_1,x'):=-\frac{c}{\rho_0}\partial_z^{-1}\xi_0(z,y)$ and consequently we have  $u_0\in H^{s+2}(\mathbb{T}_{x_1}\times \mathbb{R}^{n-1}_{x'})$, $u_1\in H^{s}(\mathbb{T}_{x_1}\times \mathbb{R}^{n-1}_{x'})$.
Thus for these initial data there exists
$$\overline{u}\in C([0,\infty[;H^{s+1}(\mathbb{T}_{x_1}\times \mathbb{R}^{n-1}_{x'}))\cap C^1([0,\infty[;H^{s}(\mathbb{T}_{x_1}\times \mathbb{R}^{n-1}_{x'}))$$ 
the unique solution on $\mathbb{T}_{x_1}\times \mathbb{R}^{n-1}_{x'}$ of the approximated Kuznetsov system
   \begin{equation}\label{CauchyaproxKuzNPE}
\left\lbrace
\begin{array}{c}
\overline{u}_{tt}-c^2 \Delta \overline{u}-\nu \varepsilon \Delta \overline{u}_t-\alpha \varepsilon \overline{u}_t \overline{u}_{tt}-\beta \varepsilon \nabla \overline{u} \nabla \overline{u}_t=\varepsilon^2 R^{Kuz-NPE},\\
\overline{u}(0)=u_0\in H^{s+2}(\mathbb{T}_{x_1}\times \mathbb{R}^{n-1}_{x'}),\;\;\;\overline{u}_t(0)=u_1\in H^{s+1}(\mathbb{T}_{x_1}\times \mathbb{R}^{n-1}_{x'})
\end{array}\right.
\end{equation}
with $R^{Kuz-NPE}$ defined in~(\ref{approxeqKuzNPE}).
If we consider the Cauchy problem~(\ref{CauProbKuz}) for the Kuznetsov equation on $\mathbb{T}_{x_1}\times \mathbb{R}^{n-1}_{x'}$ with $u_0$ and $u_1$ derived from $\xi_0$ we have 
$$\Vert u_0\Vert_{H^{s+2}(\mathbb{T}_{x_1}\times \mathbb{R}^{n-1}_{x'})}+\Vert u_1\Vert_{H^{s}(\mathbb{T}_{x_1}\times \mathbb{R}^{n-1}_{x'})}\leq C \Vert \xi_0\Vert_{H^{s+2}(\mathbb{T}_z\times \mathbb{R}^{n-1}_y)}.$$
Hence, if $\Vert \xi_0\Vert_{H^{s+2}(\mathbb{T}_z\times \mathbb{R}^{n-1}_y)}$ small enough\cite{Perso}, we have a unique solution 
$$u\in C([0,\infty[;H^{s+1}(\Omega))\cap C^1([0,\infty[;H^{s}(\Omega))$$ bounded in time of the Kuznetsov equation. 
\begin{theorem}\label{approxKuzNPE}
For the defined above solutions $u$ of the exact Cauchy problem~(\ref{CauProbKuz}) and $\overline{u}$ of the approximated Cauchy problem~(\ref{CauchyaproxKuzNPE}) for the Kuznetsov equation on $\Omega= \mathbb{T}_{x_1}\times \mathbb{R}^{n-1}_{x'}$. Then 
there exist $K>0$, $C>0$, $C_1>0$ and $C_2>0$  such that for all 
$t<\frac{C}{\varepsilon} $ 
we have estimate~(\ref{estimKuzKZKex}) and in addition Point~3 of Theorem~\ref{approxKuzKZKbis}.
\end{theorem}
\begin{proof}
The global existence of $u$ and $\overline{u}$ has already been shown. 
The proof of the approximation estimate follows exactly as in Theorem~\ref{approxKuzKZKbis} and is thus omitted.
\end{proof}
\begin{remark}
The case $\nu=0$ implies the same approximation result except that $u$ and $\overline{u}$ are only locally well posed on an interval $[0,T]$.
\end{remark}
\begin{remark}
We can see see for $n=2$ or $3$, using the previous arguments that the minimum regularity of the initial data (see Table~\ref{TABLE2}) to have the remainder terms 
$$
R^{Kuz-NPE}\in  C([0,+\infty[;L^2(\mathbb{T}_{x_1}\times \mathbb{R}^{n-1}))
$$
corresponds to
$\xi_0\in H^s(\mathbb{T}_{x_1}\times \mathbb{R}^{n-1})$ with $s \geq 4$ since then for $0\leq k \leq 2$
$$\xi(\tau,z,y)\in C^k([0,+\infty[\rbrace;H^{s-2k}(\mathbb{T}_{z}\times \mathbb{R}^{n-2})),$$
which finally implies with formula $\overline{u}=-\frac{c}{\rho_0}\partial_z^{-1}\xi $ that with $\Omega=\mathbb{T}_{x_1}\times \mathbb{R}^{n-1}$
\begin{align*}
&\overline{u}(t,x_1,x')\in C([0,+\infty[;H^4(\Omega)),\;
\partial_t\overline{u}(t,x_1,x')\in C([0,+\infty[;H^2(\Omega)),\\
&\partial_t^2\overline{u}(t,x_1,x')\in C([0,+\infty[;L^2(\Omega)).
\end{align*}
In the same way for $n\geq 4$ we can take $\xi_0\in H^s(\Omega)$ with $s >\frac{n}{2}+2$ for the minimal regulatity as it implies
\begin{align*}
&\overline{u}(t,x_1,x')\in C([0,+\infty[;H^{s}(\Omega)),\;
\partial_t\overline{u}(t,x_1,x')\in C([0,+\infty[;H^{s-2}(\Omega)),\\
&\partial_t^2\overline{u}(t,x_1,x')\in C([0,+\infty[;H^{s-4}(\Omega)).
\end{align*}
\end{remark}
\section{Kuznetsov equation and the Westervelt equation}\label{SecWest}
\subsection{Derivation of the Westervelt equation from the Kuznetsov equation.}
We consider the Kuznetsov equation~(\ref{KuzEqA}).
Similarly as in Ref.~\cite{Aanonsen} we set
\begin{equation}\label{ansatzKuzWes}
\Pi=u+\frac{1}{2 c^2}\varepsilon \partial_t[u^2]
\end{equation}
and obtain
$$\partial_t^2\Pi-c^2\Delta \Pi=\varepsilon \partial_t\left(\Delta u+\frac{\gamma+1}{2c^2} (\partial_t u)^2+\frac{1}{c^2}u(\partial_t^2-c^2\Delta u)\right).$$
By Definition~(\ref{ansatzKuzWes}) of $\Pi$ we have
$$
\partial_t^2\Pi-c^2\Delta \Pi=\varepsilon \partial_t\left(\Delta \Pi+\frac{\gamma+1}{2c^2} (\partial_t \Pi)^2\right)+\varepsilon^2 R^{Kuz-Wes},
$$
where
\begin{align}
\varepsilon^2 R^{Kuz-Wes}=&\varepsilon^2 \partial_t\left[ -\frac{1}{2c^2}\Delta(u \partial_t u)-\frac{\gamma+1}{2c^4}\partial_t u\partial^2_t(u^2)\right.\nonumber\\
&\left.\;\;\;\; +\frac{1}{c^2}u\partial_t\left( (\nabla u)^2+\frac{\gamma-1}{2c^2}(\partial_t u)^2+\frac{\nu}{\rho_0}\Delta u\right)\right]\nonumber\\
&+\varepsilon^3 \partial_t \left[-\frac{\gamma+1}{8c^6}[\partial^2_t(u^2)]^2\right].\label{resWes}
\end{align}
We recognize the Westervelt equation
\begin{equation}\label{West}
 \partial_t^2\Pi-c^2\Delta \Pi=\varepsilon \partial_t\left(\Delta \Pi+\frac{\gamma+1}{2c^2} (\partial_t \Pi)^2\right).
\end{equation}

\subsection{Approximation of the solutions of the Kuznetsov equation by the solutions of the Westervelt equation}
For the well-posedness of the Westervelt equation we refer to our work\cite{Perso} on the Kuznetsov equation where our results can be directly applied.
For $u$ solution of the Cauchy problem~(\ref{CauProbKuz}) for the Kuznetsov equation we set
$$\overline{\Pi}=u+\frac{1}{2 c^2}\varepsilon \partial_t[u^2],$$
and we have $\overline{\Pi}$ solution of the Cauchy problem
\begin{equation}\label{CauPbWesRem}
\left\lbrace
\begin{array}{l}
\partial_t^2\overline{\Pi}-c^2\Delta \overline{\Pi}=\varepsilon \partial_t\left(\Delta \overline{\Pi}+\frac{\gamma+1}{2c^2} (\partial_t \overline{\Pi})^2\right)+\varepsilon^2 R^{Kuz-Wes},\\
\overline{\Pi}(0)=\Pi_0\hbox{, }\partial_t\overline{\Pi}(0)=\Pi_1
\end{array}
\right.
\end{equation}
with $R^{Kuz-Wes}$ defined by~(\ref{resWes}) and in accordance with the definition of $\overline{\Pi}$
\begin{align}
\Pi_0=&u_0+\frac{1}{c^2}\varepsilon u_0 u_1,\label{pi0}\\
\Pi_1=&u_1+\frac{1}{c^2}\varepsilon u_1^2+ \frac{1}{c^2}\varepsilon u_0 \partial_t^2u(0)\label{pi1}\\
=&u_1+\frac{1}{c^2}\varepsilon u_1^2+ \frac{1}{c^2}\varepsilon u_0\frac{1}{1-\frac{\gamma-1}{c^2}\varepsilon u_1}\left(c^2\Delta u_0 +\frac{\nu}{\rho_0}\varepsilon \Delta u_1+2\varepsilon \nabla u_0 \nabla u_1\right)\nonumber
\end{align}
with $u_0$ and $u_1$ initial data of the the Cauchy problem~(\ref{CauProbKuz}) for the Kuznetsov equation.

For $s>\frac{n}{2}$, if we take $u_0\in H^{s+4}(\mathbb{R}^n)$ and $u_1\in H^{s+3}(\mathbb{R}^3)$, we have $\Pi_0\in H^{s+3}(\mathbb{R}^n)\subset H^{s+2}(\mathbb{R}^n)$ and $\Pi_1 \in H^{s+1}(\mathbb{R}^n)$ with
$$\Vert \Pi_0\Vert_{H^{s+2}(\mathbb{R}^n)}+\Vert \Pi_1\Vert_{H^{s+1}(\mathbb{R}^n)}\leq C (\Vert u_0\Vert_{H^{s+4}(\mathbb{R}^n)}+\Vert u_1\Vert_{H^{s+3}(\mathbb{R}^n)} ),$$
so similarly to our previous work \cite{Perso} we obtain
\begin{theorem}\label{ThWPWesR3}
Let $n\geq 1$, $s>\frac{n}{2}$, $u_0\in H^{s+4}(\mathbb{R}^n)$ and $u_1\in H^{s+3}(\mathbb{R}^n)$. Then there exists a constant $k_2>0$ such that if 
\begin{equation}\label{smallcondWPwes}
\Vert u_0\Vert_{H^{s+4}(\mathbb{R}^n)}+\Vert u_1\Vert_{H^{s+3}(\mathbb{R}^n)} <k_3,
\end{equation}
then the Cauchy problem for the Westervelt equation 
\begin{equation}\label{CauPbWes}
\left\lbrace
\begin{array}{l}
\partial_t^2\Pi-c^2\Delta \Pi=\varepsilon \partial_t\left(\Delta \Pi+\frac{\gamma+1}{2c^2} (\partial_t \Pi)^2\right),\\
\overline{\Pi}(0)=\Pi_0\hbox{, }\partial_t\overline{\Pi}(0)=\Pi_1
\end{array}
\right.
\end{equation}
with $\Pi_0$ and $\Pi_1$ defined by Eqs.~(\ref{pi0}) and~(\ref{pi1}),
 has a unique global in time solution 
\begin{align}
\Pi  & \in H^2([0,+\infty[,H^s(\mathbb{R}^n))\cap H^1([0,+\infty[,H^{s+2}(\mathbb{R}^n))\label{RegWes31}\\
\hbox{ and } &\hbox{if } s\geq 1  \nonumber\\
\Pi  &\in C([0,+\infty[,H^{s+2}(\mathbb{R}^n))\cap C^1([0,+\infty[,H^{s+1}(\mathbb{R}^n))\cap C^2([0,+\infty[,H^{s-1}(\mathbb{R}^n))\label{RegWes32}
\end{align}
Moreover we have $\overline{\Pi}$ global in time solution of the approximated Cauchy problem~(\ref{CauPbWesRem}) with the same regularity.
\end{theorem}
For $\Pi$ solution of the Caucchy problem (\ref{CauPbWes}) we set $\overline{u}$ such that $\Pi=\overline{u}+\frac{\varepsilon}{c^2}\overline{u}\partial_t\overline{u}$ and we obtain
\begin{align*}
\partial_t^2 \overline{u}-c^2 \Delta \overline{u}-\varepsilon\frac{\nu}{\rho_0}\Delta &\partial_t\overline{u}-\varepsilon \frac{\gamma -1}{c^2}\partial_t  \overline{u} \partial^2_t\overline{u}-2  \varepsilon \nabla \overline{u}.\nabla\partial_t  \overline{u}\\
+&\varepsilon\left(\frac{1}{c^2}\partial_t  \overline{u} \partial^2_t\overline{u}-\partial_t  \overline{u}\Delta \overline{u} +\frac{1}{c^2} \overline{u} \partial_t^3 \overline{u} -\varepsilon \overline{u}\Delta\partial_t  \overline{u}\right)=\varepsilon^2 R^{Wes-Kuz}_1 
\end{align*}
with 
\begin{align*}
R^{Wes-Kuz}_1=\left[ \frac{\nu}{\rho_0 c^2}\right.& (2\partial_t  \overline{u} \Delta\partial_t \overline{u}+2(\nabla \partial_t\overline{u})^2+\partial_t^2 \overline{u}\Delta \overline{u}+\overline{u}\Delta\partial^2_t+2\nabla \overline{u}.\nabla \partial^2_t \overline{u})\\
 &\left.+\frac{\gamma +1}{c^4}((\partial_t\overline{u})^2+\overline{u}\partial^2_t \overline{u})\partial^2_t\overline{u}+\frac{\gamma +1}{c^4} (3\partial_t \overline{u} \partial^2_t \overline{u}+\overline{u}\partial^3_t \overline{u})\partial_t \overline{u}\right]\\
 +\varepsilon \frac{\gamma +1}{c^6}& ((\partial_t\overline{u})^2+\overline{u}\partial^2_t \overline{u})(3\partial_t \overline{u} \partial^2_t \overline{u}+\overline{u}\partial^3_t \overline{u}).
\end{align*}
And as $$\partial_t^2 \overline{u}-c^2 \Delta \overline{u}=O(\varepsilon) $$ if we inject this in the term $\left(\frac{1}{c^2}\partial_t  \overline{u} \partial^2_t\overline{u}-\partial_t  \overline{u}\Delta \overline{u} +\frac{1}{c^2} \overline{u} \partial_t^3 \overline{u} -\varepsilon \overline{u}\Delta\partial_t  \overline{u}\right)$
we have
\begin{equation}\label{eqWesKuzRem}
\partial_t^2 \overline{u}-c^2 \Delta \overline{u}-\varepsilon\frac{\nu}{\rho_0}\Delta \partial_t\overline{u}-\varepsilon \frac{\gamma -1}{c^2}\partial_t  \overline{u} \partial^2_t\overline{u}-2  \varepsilon \nabla \overline{u}.\nabla\partial_t  \overline{u} =\varepsilon^2 R^{Wes-Kuz}.
\end{equation}

Now we can write the following approximation result for the Westervelt equation
\begin{theorem}\label{ApproxKuzWes}
Let $\nu>0$, $n\geq 2$, $s>\frac{n}{2}$ with $s\geq 1$, $\overline{u}_0\in H^{s+4}(\mathbb{R}^n)$ and $\overline{u}_1\in H^{s+3}(\mathbb{R}^n)$, there exists $k>0$ such that $\Vert \overline{u}_0\Vert_{H^{s+4}(\mathbb{R}^n)}+\Vert \overline{u}_1\Vert_{H^{s+3}(\mathbb{R}^n)}<k$ implies the global existence of $\Pi$ with the regularityies (\ref{RegWes31}) and (\ref{RegWes32}) which is the solution of the Cauchy problem~(\ref{CauPbWes}) with $\Pi_0$ and $\Pi_1$ defined by Eqs.~(\ref{pi0}) and~(\ref{pi1}). 
Moreover for $u_0\in H^{s+2}(\mathbb{R}^n)$ and $u_1\in H^{s+1}(\mathbb{R}^n)$ we have $u$ exact solution of the Cauchy problem~(\ref{CauProbKuz}) for the Kuznetsov equation. 
Let $\overline{u}$ such that 
$$\Pi=\overline{u}+\frac{\varepsilon}{c^2} \overline{u}\partial_t \overline{u},$$ 
as a consequence $\overline{u}$ is a solution of the approximated Kuznetsov equation~(\ref{eqWesKuzRem}) and 
 if $u$ and $\overline{u}$ satisfies~(\ref{EqSmallInD}) with $u(0)=u_0$, $\partial_t u(0)=u_1$, $\overline{u}(0)=\overline{u}_0$, $\partial_t\overline{u}(0)=\overline{u}_1$,  there exists $K>0$ and $C>0$ independent of $\varepsilon$ such that for all
$t\leq\frac{C}{\varepsilon}$
we have $C_1>0$ and $C_2>0$ with (\ref{estimKuzKZK}).
\end{theorem}
\begin{proof}
The existence of $u$ and $\overline{u}$ has already been shown. 
The proof of the approximation estimate follows exactly the proof of Theorem~\ref{approxKuzKZKbis} and hence it is omitted.
\end{proof}
\begin{remark}
For the minimal  regularity (see Table~\ref{TABLE2}) of $u_0$ and $u_1$ to ensure  that $R^{Kuz-Wes}$, see Eq.~(\ref{resWes}), is in $C([0,+\infty[;L^2(\mathbb{R}_+\times\mathbb{R}^{n-1}))$, if $u_0\in H^{s+2}(\mathbb{R}^3)$ and $u_1\in H^{s+1}(\mathbb{R}^3)$ for $s\geq 3$ then 
\begin{align*}
u\in & C([0,+\infty[;H^5(\mathbb{R}^3)),\;
\partial_t u\in  C([0,+\infty[;H^4(\mathbb{R}^3)),\\
\partial_t^2 u\in & C([0,+\infty[;H^2(\mathbb{R}^3)),\;
\partial_t^3 u\in   C([0,+\infty[;L^2(\mathbb{R}^3)).
\end{align*}
Taking $\overline{\Pi}$ as in (\ref{ansatzKuzWes})
we obtain
\begin{align*}
\overline{\Pi}\in  C([0,+\infty[;H^4(\mathbb{R}^3)),\;
\partial_t \overline{\Pi}\in C([0,+\infty[;H^2(\mathbb{R}^3)),\;
\partial_t^2 \overline{\Pi}\in  C([0,+\infty[;L^2(\mathbb{R}^3)).
\end{align*}
Injecting this result in the approximated Westervelt equation in system~(\ref{CauPbWesRem}) we obtain
$R^{Kuz-Wes}\in C([0,+\infty[;L^2(\mathbb{R}^3)).$
In the same way if $n\geq 4$ we take $u_0\in H^{s+2}(\mathbb{R}^n)$ and $u_1\in H^{s+1}(\mathbb{R}^n)$ with $s>\frac{n}{2}+1$.
\end{remark}
\section{Conclusion}
We summarize all obtained approximation results in two comparatif tables: Table~\ref{TABLE} for the approximations of the Navier-Stokes and Euler systems and Table~\ref{TABLE2} for the approximations of the Kuznetsov equation.
\renewcommand{\arraystretch}{2.0}

\begin{sidewaystable}[ph!]
\vspace{0cm}\centering \caption{Approximation results for models derived from Navier-Stokes and Euler systems}\label{TABLE}

\begin{tabular}{@{} c | c | c || c | c || c | c|}
    \cline{2-7}
    & \multicolumn{2}{| c ||}{ Kuznetsov }
    &\multicolumn{2}{| c ||}{  KZK}
    & \multicolumn{2}{| c |}{ NPE}
\\
     \cline{2-7}
     & Navier-Stokes & Euler
     & Navier-Stokes & Euler
     & Navier-Stokes & Euler
     \\
     \hline
\multicolumn{1}{| c |}{Theorem}
    &  Theorem~\ref{ThapproxNSKuz}
    &  Theorem~\ref{ThapproxEulKuz}
    &  Theorem~\ref{ThAprKZKNS}
    &  Theorem~\ref{ThAproEulKZK}
    &  Theorem~\ref{ThapproxNSNPE}
    &  Theorem~\ref{ThapproxEulKuz}
\\ \hline
\multicolumn{1}{| c |}{\textit{Ansatz}}
    &\multicolumn{2}{| c ||}{\shortstack{$\rho_{\varepsilon}=\rho_0+\varepsilon \rho_1+\varepsilon^2 \rho_2 ,$\\
    $ \mathbf{v}_{\varepsilon}=-\varepsilon \nabla u,$\\
    $\rho_1=\frac{\rho_0}{c^2}\partial_t u,$ \\ 
    $\rho_2$ from (\ref{rho2K}) 
    } }
    &\multicolumn{2}{| c ||}{\shortstack{ paraxial approximation\\
    $u=\Phi(t-\frac{x_1}{c},\varepsilon x_1,\sqrt{\varepsilon} \textbf{x}')$\\
    $\rho_{\varepsilon}=\rho_0+\varepsilon I+\varepsilon^2 J,$\\
    $ \mathbf{v}_{\varepsilon}$ from (\ref{vkzk}), 
    $I =\frac{\rho_0}{c^2}\partial_{\tau}\Phi,$ \\ 
    $J$ from (\ref{Jkzk})
    }}
    &\multicolumn{2}{| c |}{\shortstack{ paraxial approximation\\
    $u=\Psi(\varepsilon t,x_1-ct,\sqrt{\varepsilon} \textbf{x}')$\\
    $\rho_{\varepsilon}=\rho_0+\varepsilon \xi+\varepsilon^2 \chi,$\\
    $ \mathbf{v}_{\varepsilon}$ from (\ref{vNPE}), 
    $\xi =-\frac{\rho_0}{c}\partial_{z}\Psi,$ \\ 
    $\chi$ from (\ref{P2NPE})
    }}
\\
    \hline
    \multicolumn{1}{| c |}{Models}
    &\multicolumn{2}{| c ||}{\shortstack{$\partial^2_t u -c^2 \Delta u=$\\$\varepsilon\partial_t\left( (\nabla u)^2+\frac{\gamma-1}{2c^2}(\partial_t u)^2 \right.$\\
    $\left.+\frac{\nu}{\rho_0}\Delta u\right)$}}
    &\multicolumn{2}{| c ||}{\shortstack{$c\partial^2_{\tau z} I -\frac{(\gamma+1)}{4\rho_0}\partial_\tau^2
I^2$\\$-\frac{\nu}{2 c^2\rho_0}\partial^3_\tau I-\frac{c^2}2 \Delta_y
I=0$}}
    &\multicolumn{2}{| c |}{\shortstack{$\partial^2_{\tau z} \xi+\frac{(\gamma+1)c}{4\rho_0}\partial_z^2(\xi^2)$\\$-\frac{\nu}{2\rho_0}\partial^3_z  \xi+\frac{c}{2}\Delta_y  \xi=0$}}
\\
    \hline
  \multicolumn{1}{| c |}{\shortstack{Approxi-\\mation\\ Order}}
  &\multicolumn{6}{| c |}{$O(\varepsilon^3)$}
  \\
    \hline
    \multicolumn{1}{| c |}{ Domain $\Omega$}
    &\multicolumn{2}{| c ||}{$\mathbb{R}^3 $}
    & \shortstack{the half space \\$\{x_1>0,
x'\in \R^{n-1}\}$} 
& \shortstack{the cone \\$\{|x_1|<\frac{R}{\eps}-ct\}$\\$\times
    \R^{n-1}_{x'}$} 
    &\multicolumn{2}{| c |}{$\mathbb{T}_{x_1}\times\mathbb{R}^2$}
    \\
    \hline
    \multicolumn{1}{| c |}{\shortstack{Approxi-\\mation 
    }}
    &\multicolumn{6}{| c |}{$\Vert U_{\varepsilon}-\overline{U}_{\varepsilon}\Vert_{L^2}\leq \varepsilon $ for $t\leq\frac{T}{\varepsilon}$}
    \\
    \hline
    \multicolumn{1}{| c |}{\shortstack{Initial\\
    data\\
    regularity}}
    & \shortstack{$u_0\in H^5(\Omega)$\\ $u_1\in H^4(\Omega)$}
    & \shortstack{$u_0\in H^4(\Omega)$\\ $u_1\in H^3(\Omega)$}
    & $I_0\in H^{10}(\Omega)$
    & $I_0\in H^{10}(\Omega)$
    & $\xi_0\in H^5(\Omega)$
    & $\xi_0\in H^5(\Omega)$
    \\
    \hline
    \multicolumn{1}{|c|}{\shortstack{Data\\regularity\\
    for remainder\\
    boundness}}
    &\shortstack{$u_0\in H^{s+2}(\Omega)$\\
    $u_1\in H^{s+1}(\Omega)$\\
    $s>\frac{n}{2}$ }
    &\shortstack{$u_0\in H^{s+2}(\Omega)$\\
    $u_1\in H^{s+1}(\Omega)$\\
    $s>\frac{n}{2}$ }
    & $I_0\in H^8(\Omega)$
    & $I_0\in H^8(\Omega)$
    & $\xi_0\in H^4(\Omega)$
    & $\xi_0\in H^4(\Omega)$\\
    \hline
    \end{tabular}

\end{sidewaystable}

\begin{sidewaystable}[ph!]
\vspace{0cm}\centering \caption{Approximation results for models derived from the Kuznetsov equation}\label{TABLE2}
\begin{tabular}{@{} c | c | c || c || c |}
\cline{2-5}
    & \multicolumn{2}{| c ||}{\quad KZK \quad}
    &\quad NPE\quad
    & \quad Westervelt\quad
\\
\cline{2-3}
    &\shortstack{periodic\\
    boundary condition\\
    problem}
    & \shortstack{initial  \\
    boundary value\\
    problem}
    & 
    &
    \\
    \hline
     \multicolumn{1}{| c |}{Theorem}
     &  Theorem~\ref{AproxKuzKZK}
    &  Theorem~\ref{approxKuzKZKbis}
    &  Theorem~\ref{approxKuzNPE}
    &  Theorem~\ref{ApproxKuzWes}
     \\
    \hline
    \multicolumn{1}{| c |}{Derivation}
    &\multicolumn{2}{| c ||}{\shortstack{ paraxial approximation\\
    $u=\Phi(t-\frac{x_1}{c},\varepsilon x_1,\sqrt{\varepsilon} \textbf{x}')$} }
    & \shortstack{ paraxial approximation\\
    $u=\Psi(\varepsilon t,x_1-ct,\sqrt{\varepsilon} \textbf{x}')$}
    & $\Pi=u+\frac{1}{c^2}\varepsilon u \partial_t u$
    \\
    \hline
     \multicolumn{1}{| c |}{\shortstack{Approxi-\\mation\\
     domain}}
     & \multicolumn{2}{| c ||}{\shortstack{the half space \\$\{x_1>0,
x'\in \R^{n-1}\}$} } 
& $\mathbb{T}_{x_1}\times\mathbb{R}^2$
& $\mathbb{R}^3$
 \\
    \hline
     \multicolumn{1}{| c |}{\shortstack{Approxi-\\mation \\
     order}}
     & \multicolumn{2}{| c ||}{$ O(\varepsilon)$}
     & $O(\varepsilon)$ 
     & $O(\varepsilon^2)$ 
      \\
    \hline
    \multicolumn{1}{| c |}{Estimation}
    & \shortstack{ $\Vert I-I_{aprox}\Vert_{L^2(\mathbb{T}_t\times\mathbb{R}^{n-1})}\leq \varepsilon $\\
    $z\leq K$}
    &\shortstack{ $\Vert (u -\overline{u})_t(t)\Vert_{L^2}$\\$+ \Vert \nabla (u-\overline{u})(t)\Vert_{L^2}$\\$\leq K \varepsilon.$\\
    $t<\frac{T}{\varepsilon} $}
    & \shortstack{ $\Vert (u -\overline{u})_t(t)\Vert_{L^2}$\\$+ \Vert \nabla (u-\overline{u})(t)\Vert_{L^2}$\\
    $\leq K \varepsilon$\\
    $t<\frac{T}{\varepsilon} $}
    & \shortstack{ $\Vert (u -\overline{u})_t(t)\Vert_{L^2}$\\
    $+ \Vert \nabla (u-\overline{u})(t)\Vert_{L^2}$\\
    $\leq K \varepsilon$\\
   $t<\frac{T}{\varepsilon} $ }
    \\
    \hline
    \multicolumn{1}{| c |}{\shortstack{Initial\\
    data\\
    regularity}}
    & \shortstack{$ I_0\in H^{s+\frac{3}{2}}(\mathbb{T}_{t}\times \mathbb{R}^{n-1}_{x'})$\\ for $s \geq \frac{n}{2}+2$} 
    & \shortstack{$ I_0\in H^{s}(\mathbb{T}_{t}\times \mathbb{R}^{n-1}_{x'})$\\
    for $\left[\frac{s}{2}\right]>\frac{n}{2}+2 $}
    &\shortstack{$ \xi _0\in H^{s+2}(\mathbb{T}_{x_1}\times \mathbb{R}^{n-1}_{x'})$\\ for $s>\frac{n}{2}+1$}
    &\shortstack{$u_0\in H^{s+4}(\mathbb{R}^n)$\\ $u_1\in H^{s+3}(\mathbb{R}^3)$\\
    for $s>\frac{n}{2}$}
    \\
    \hline
    \multicolumn{1}{|c|}{\shortstack{Data\\regularity\\
    for remainder\\
    boundness}}
    & \shortstack{$ I_0\in H^{s+\frac{3}{2}}(\mathbb{T}_{t}\times \mathbb{R}^{n-1}_{x'})$\\ for $s \geq \frac{n}{2}+2$}
    & \shortstack{$ I_0\in H^{6}(\mathbb{T}_{t}\times \mathbb{R}^{n-1}_{x'})$\\
    for $n= 2,3$,\\
    $ I_0\in H^{s}(\mathbb{T}_{t}\times \mathbb{R}^{n-1}_{x'})$\\
    for $\left[\frac{s}{2}\right]>\frac{n}{2}+1 $ and $n\geq 4$}
    & \shortstack{$ \xi _0\in H^{4}(\mathbb{T}_{x_1}\times \mathbb{R}^{n-1}_{x'})$\\
    for $n=2,3$.\\
    $ \xi _0\in H^{s}(\mathbb{T}_{x_1}\times \mathbb{R}^{n-1}_{x'})$\\ for $s>\frac{n}{2}+2$ and $n\geq4$.}
     &\shortstack{$u_0\in H^{s+2}(\mathbb{R}^n)$\\ $u_1\in H^{s+1}(\mathbb{R}^n)$\\
     for $s\geq 3$ and $n=2,3$.\\
     $u_0\in H^{s+2}(\mathbb{R}^n)$\\ $u_1\in H^{s+1}(\mathbb{R}^n)$\\
     for $s\geq \frac{n}{2}+1$ and $n\geq4$.}\\
     \hline
    \end{tabular}
\end{sidewaystable}

\appendix
\section{Expressions of the remainder terms.}\label{Apen1}
The expression of $H$, the profile of $\rho_2$, in the paraxial variables of the KZK \textit{ansatz}:
\begin{align}
H(\tau,z,y)= &-\frac{\rho_0 (\gamma-1)}{2c^4}(\partial_\tau \Phi)^2-\frac{\nu}{c^4} \partial^2_\tau \Phi\nonumber\\
&+\varepsilon\left[-\frac{\rho_0}{2c^2}[(\nabla_y \Phi)^2-\frac{2}{c} \partial_z \Phi\;\partial_\tau\Phi]-\frac{\nu}{c^2}[\Delta_y\Phi-\frac{2}{c}\partial^2_{z\tau}\Phi]\right]
\nonumber\\
&+\varepsilon^2[-\frac{\rho_0}{2c^2}(\partial_z\Phi)^2-\frac{\nu}{c^2}\partial^2_z\Phi],\label{rho2kzkH}
\end{align}

If we consider~(\ref{massKZK})-(\ref{momentKZK}) the expressions of $R_1^{NS-KZK}$ and $\mathbf{R}_2^{NS-KZK}$ are written with the terms $I$ and $J$ defined by~(\ref{Ikzk}) and~(\ref{Jkzk}) respectively.
\begin{align*}
&\varepsilon^3 R_1^{NS-KZK}=\\
&\;\;\;\;\;\varepsilon^3\left[-\rho_0\partial_{z}^2\Phi+\frac{1}{c}\partial_{z}I \partial_{\tau}\Phi+\frac{1}{c}\partial_{\tau}I \partial_z\Phi-\nabla_y I.\nabla_y\Phi\right.\\
&\;\;\;\;\;\;\left.+\frac{2}{c} I\partial^2_{\tau z}\Phi-I\Delta_y \Phi-\frac{1}{c^2}\partial_{\tau}J \partial_{\tau}\Phi-\frac{1}{c^2} J \partial_{\tau}^2\Phi \right] \\
&+\varepsilon^4\left[-\partial_{z} I\partial_{z}\Phi-I\partial_{z}^2\Phi+\frac{1}{c}\partial_{z}J\partial_{\tau}J+\frac{1}{c}\partial_{\tau}J\partial_{z}\Phi\right.\\
&\;\;\;\;\;\;\left.-\nabla_y J.\nabla_y\Phi+\frac{2}{c} J \partial^2_{\tau z}\Phi-J\Delta_y\Phi\right]\\
&+\varepsilon^5[-\partial_z J \partial_z \Phi-J\partial^2_{z}\Phi].
\end{align*}
Among the $x_1$ axis
\begin{align*}
&\varepsilon^3 \mathbf{R}_2^{NS-KZK}.\overrightarrow{e}_1=\\
&\;\;\;\;\;\varepsilon^3\left[-\frac{\rho_0}{2c}\partial_{\tau}[-\frac{2}{c}\partial_z\Phi\partial_{\tau}\Phi+(\nabla_y \Phi)^2]-\frac{\nu}{c}\partial_{\tau}[-\frac{2}{c}\partial^2_{\tau z}\Phi+\Delta_y\Phi]\right.\\
&\;\;\;\;\;\;\left.-\frac{I}{2c}\partial_{\tau}[\frac{1}{c^2}(\partial_{\tau}\Phi)^2]+\frac{J}{c}\partial^2_{\tau}\Phi\right]\\
&+\varepsilon^4\left[\frac{\rho_0}{2}\partial_z[-\frac{2}{c}\partial_z\Phi\partial_{\tau}\Phi+(\nabla_y \Phi)^2]+\nu \partial_z[-\frac{2}{c}\partial^2_{\tau z}\Phi+\Delta_y\Phi]\right.\\
&\;\;\;\;\;\;-\frac{I}{2c}\partial_{\tau}[-\frac{2}{c}\partial_z\Phi\partial_{\tau}\Phi+(\nabla_y \Phi)^2]+\frac{I}{2}\partial_z [\frac{1}{c^2}(\partial_{\tau}\Phi)^2]-J\partial^2_{\tau z}\Phi\\
&\;\;\;\;\;\;\left.-\frac{J}{2c}\partial_{\tau}[\frac{1}{c^2}(\partial_{\tau}\Phi)^2]-\frac{\rho_0}{2c}\partial_{\tau}[(\partial_z\Phi)^2]-\frac{\nu}{c}\partial_{\tau}\partial^2_z\Phi\right]\\
&+\varepsilon^5\left[-\frac{I}{2c}\partial_{\tau}[(\partial_z\Phi)^2]+\frac{I}{2}\partial_z[-\frac{2}{c}\partial_z\Phi\partial_{\tau}\Phi+(\nabla_y \Phi)^2]\right.\\
&\;\;\;\;\;\;+\frac{J}{2}\partial_z[\frac{1}{c^2}(\partial_{\tau}\Phi)^2]-\frac{J}{2c}\partial_{\tau}[-\frac{2}{c}\partial_z\Phi\partial_{\tau}\Phi+(\nabla_y \Phi)^2]\\
&\;\;\;\;\;\;\left.+\frac{\rho_0}{2}\partial_z [(\partial_z\Phi)^2]+\nu \partial_z^3\Phi\right]\\
&+\varepsilon^6\left[\frac{I}{2}\partial_z [(\partial_z\Phi)^2]-\frac{J}{2c}\partial_{\tau}[(\partial_z\Phi)^2]+\frac{J}{2}[-\frac{2}{c}\partial_z\Phi\partial_{\tau}\Phi+(\nabla_y \Phi)^2]\right]\\
&+\varepsilon^7\left[\frac{J}{2}\partial_z [(\partial_z\Phi)^2]\right]
\end{align*}
and in the hyperplane orthogonal to the $x_1$ axis
\begin{align*}
&\sum_{i=2}^n(\mathbf{R}_2^{NS-KZK}.\overrightarrow{e}_i)\overrightarrow{e}_i=\\
&\;\;\;\;\;\varepsilon^{\frac{7}{2}}\left[\frac{\rho_0}{2}\nabla_y[-\frac{2}{c}\partial_z\Phi\partial_{\tau}\Phi+(\nabla_y \Phi)^2]+\nu \nabla_y[-\frac{2}{c}\partial^2_{\tau z}\Phi+\Delta_y\Phi]\right.\\
&\;\;\;\;\;\;\left.+\frac{I}{2}\nabla_y[\frac{1}{c^2}(\partial_{\tau}\Phi)^2]-J\nabla_y[\partial_{\tau}\Phi]\right]\\
&+\varepsilon^{\frac{9}{2}}\left[\frac{I}{2}\nabla_y[-\frac{2}{c}\partial_z\Phi\partial_{\tau}\Phi+(\nabla_y \Phi)^2]+\frac{J}{2}\nabla_y[\frac{1}{c^2}(\partial_{\tau}\Phi)^2]\right.\\
&\;\;\;\;\;\;\left.+\frac{\rho_0}{2}\nabla_y[(\partial_z\Phi)^2]+\nu\nabla_y[\partial^2_z \Phi]\right]\\
&+\varepsilon^{\frac{11}{2}}\left[\frac{I}{2}\nabla_y[(\partial_z\Phi)^2]+\frac{J}{2}\nabla_y[-\frac{2}{c}\partial_z\Phi\partial_{\tau}\Phi+(\nabla_y \Phi)^2]\right]\\
&+\varepsilon^{\frac{13}{2}}\left[\frac{J}{2}\nabla_y[(\partial_z\Phi)^2]\right]
\end{align*}




\end{document}